\documentclass[12pt]{amsart}
\usepackage[margin=1in]{geometry}
\usepackage{danpack}
\usepackage{hyperref}
\usepackage{bm}
\usepackage{mabliautoref}
\usepackage{tikz-cd}

\newcommand{\sheaf}[1]{\mathscr{#1}}
\newcommand{\cartalg}[1]{\mathscr{#1}}
\newcommand{\ssheaf}{\mathscr O}
\newcommand{\multideal}{\mathscr J}
\newcommand{\fracsheaf}[1]{\mathscr K \left( #1 \right)}

\providecommand{\underscore}{\rule{1em}{0.5pt}}

\newcommand{\tauOfa}[1]{\underline{\mathfrak a}^{\lceil t(p^{#1}-1) \rceil}}
\newcommand{\tauOfaq}{\underline{\mathfrak a}^{\lceil t(q-1) \rceil}}
\newcommand{\idealsProd}{\underline{\mathfrak a}^t}
\newcommand{\compSubAlg}{\, \scalebox{0.75}{\rotatebox[origin=c]{90}{$\circlearrowleft$} } }
\renewcommand{\tilde}{\widetilde}
\renewcommand{\hat}{\widehat}
\renewcommand{\bar}{\overline}

\newcommand{\plainConR}{\mathscr{D}^{(2)}}
\newcommand{\ConR}{\plainConR (R)}
\newcommand{\ConProd}{\fullCA{R\otimes_k R,\, I_\Delta \compSubAlg}}
\newcommand{\fullCA}[1]{\cartalg C^{#1}}
\newcommand{\acptope}{P_{-K_X}}

\newcommand{\exclocus}{  \operatorname{exc} }

\newcommand{\cf}{\emph{Cf.}}

\title{A new subadditivity formula for test ideals}
\author{Daniel Smolkin}
\address{Department of Mathematics\\ University of Utah\\ Salt Lake City\\ UT 84112\\USA} 
\thanks{The author was partially supported by NSF grants DMS \#1246989, DMS \#1252860, and DMS \#1265261.}
\email{\href{mailto:smolkin@math.utah.edu}{smolkin@math.utah.edu}}

\begin{document}

\begin{abstract}
  We exhibit a new subadditivity formula for test ideals on singular varieties using an argument similar to \cite{DELsubadd} and \cite{HaraYoshidaSubadd}. Any subadditivity formula for singular varieties must have a correction term that measures the singularities of that variety. Whereas earlier subadditivity formulas accomplished this by multiplying by the Jacobian ideal, our approach is to use the formalism of Cartier algebras \cite{BlickleCartier}. We also show that our subadditivity containment is sharper than ones shown previously in \cite{TakagiSingMultIdeals} and \cite{EisensteinRestrictionThm}. The first of these results follows from a Noether normalization technique due to Hochster and Huneke. The second of these results is obtained using ideas from \cite{TakagiCharPadjoint} and \cite{EisensteinRestrictionThm} to show that the adjoint ideal $\multideal_X(A, Z) $ reduces mod $p$ to Takagi's adjoint test ideal, even when the ambient space is singular, provided that $A$ is regular at the generic point of $X$. One difficulty of using this new subadditivity formula in practice is the computational complexity of computing its correction term. Thus, we discuss a combinatorial construction of the relevant Cartier algebra in the toric setting. 
\end{abstract}

\maketitle

\section{Introduction}
Test ideals are an important measure of singularity in characteristic-$p$ commutative algebra. The test ideals in a regular ambient ring $R$ enjoy a property called \emph{subadditivity}. Namely, if $\mf a$ and $\mf b$ are ideals of $R$ and $s,t \geq 0$ are real numbers, then 
  \[
    \tau(R, \mf a^s \mf b^t) \subseteq \tau(R, \mf a^s) \tau(R, \mf b^t).
  \]
  Here, and throughout this paper, we refer only to the big/non-finitistic test ideal. The analogous formula was originally proven for multiplier ideals in characteristic 0 in \cite{DELsubadd}. Later, Hara and Yoshida defined the test ideal of a pair $(R, \mathfrak a^s)$ and showed the above relationship of test ideals holds in characteristic $p$ \cite{HaraYoshidaSubadd}. Together, these two formulas have numerous applications. Perhaps the most striking one is Ein, Lazarsfeld, and Smith's proof that smooth varieties satisfy the so-called ``Uniform Symbolic Topologies Property'' \cite{ELSsymbolicPowers}. 

One can ask to what extent the subadditivity property holds for non-regular rings. The first result in this direction came from \cite[Theorem 2.7]{TakagiSingMultIdeals}, where Takagi showed 
\[
  \jac(R) \tau(R, \mf a^s \mf b^t) \subseteq \tau(R, \mf a^s) \tau(R, \mf b^t)
\]
for any equidimensional reduced affine algebra over a perfect field of positive characteristic. Since multiplier ideals are test ideals mod $p \gg 0$ \cite{SmithMultIdealIsTestIdeal, HaraMultIdealIsTestIdeal, HaraYoshidaSubadd} and inclusion mod $p \gg 0$ implies inclusion in characteristic $0$, the same formula holds for multiplier ideals.

Our proof of subadditivity is similar to the original ones in \cite{DELsubadd} and \cite{HaraYoshidaSubadd}. There, the idea is to notice that, if $k$ is a field of positive characteristic and $R$ is a $k$-algebra essentially of finite type, then  $\tau(R\otimes_k R, (\mathfrak a\otimes_k R)^s (R \otimes_k \mathfrak b)^t) \subseteq \tau(R, \mathfrak a^s) \otimes_k \tau(R, \mathfrak b^t)$. Then it's enough to check that 
\[
  \tau\left( R, \mathfrak a^s \mathfrak b^t \right) \subseteq \mu \left( \tau(R\otimes_k R, (\mathfrak a\otimes R)^s (R \otimes_k \mathfrak b)^t) \right)
\]
where $\mu\colon R \otimes_k R \to R$ is the multiplication map. This is accomplished by the following restriction theorem: 
\begin{theorem*}[{\cf \cite[Theorem 4.1]{HaraYoshidaSubadd}}]
  Let $R$ be a normal $\bQ$-Gorenstein ring, and that suppose $S = R/(x)$ is normal and $\bQ$-Gorenstein as well, for some regular element $x\in R$. Then $\tau(S, (\mathfrak aS)^s ) \subseteq \tau(R, \mathfrak a^s) S$. 
\end{theorem*}
\noindent
If $R$ is smooth\footnote{For instance, if $R$ is regular and $k$ is perfect.}, then the kernel of $\mu$ is generated by a regular sequence, so repeated application of the restriction theorem yields the desired containment. 

%consider the isomorphism $R\otimes_k R/I_{\Delta} \cong R$ where $I_\Delta = \langle 1 \otimes x - x \otimes 1 | x \in R \rangle$. If $x\in R$ is a regular element, then \ldots

Eisenstein \cite{EisensteinRestrictionThm} uses a similar argument to obtain his subadditivity theorem for multiplier ideals. Namely, he proves a more general restriction theorem for multiplier ideals and carefully studies its implications in the case where one restricts to the diagonal subscheme of $\spec(R\otimes_k R)$. 

Our approach is to follow this diagonal argument. To get around the restriction theorem, we will use the formalism of \emph{Cartier algebras} \cite{SchwedeNonQGor, BlickleCartier}. In particular, for all rings $R$ and ideals $I\subseteq R$, we construct a Cartier algebra $\cartalg C$ on $R/I$ such that $\tau(R, \mathfrak a^s)  R/I \supseteq \tau(R/I, \cartalg C, (\mathfrak aR/I)^s)$. Applying this construction to the quotient $R \cong (R\otimes_k R)/I_\Delta$, we call this Cartier algebra $\ConR$ and we get
\begin{proposition*}[\autoref{thm:MySubadd}]
  Let $R$ be a $k$-algebra of positive characteristic. Then for all ideals $\mathfrak a, \mathfrak b \subseteq R$ and all real numbers $s,t \geq 0$, we have
  $\tau(R, \ConR, \mathfrak a^s \mathfrak b^t) \subseteq \tau(R, \mathfrak a^s ) \tau(R, \mathfrak b^t)$.
\end{proposition*}
We show in \autoref{thm:jacThm} and \autoref{cor:mainresult} that this containment is sharper than the previously-known subadditivity results for test ideals:
\begin{theoremA*}[\autoref{thm:jacThm}]
Let $k$ be a perfect field and $R$ a $k$-algebra essentially of finite type. Suppose also that $R$ is equidimensional and reduced. Then 
  \[
    \jac(R) \tau\left(R, \prod_i \mathfrak a_i^{t_i} \right) \subseteq \tau\left(R, \plainConR, \prod_i \mathfrak a_i^{t_i} \right)
  \]
  for all formal products of ideals $\prod \mathfrak a_i^{t_i}$ such that each $\mathfrak a_i$ contains a regular element.
\end{theoremA*}
\begin{theoremB*}[\autoref{cor:mainresult}]
Let $R$ be a $\bQ$-Gorenstein ring of finite type over a field of characteristic $0$ and let $\mathfrak a, \mathfrak b \subseteq R$ be ideals. Let $R_p$ denote the mod-$p$ reduction of $R$, and similarly for $\mathfrak a_p$ and $\mathfrak b_p$. Then
\[
  \overline{\jac(R_p)} \tau(R_p, (\mathfrak a_p)^s (\mathfrak b_p)^t) \subseteq \tau(R, \ConR, (\mathfrak a_p)^s (\mathfrak b_p)^t)
\]
for all $p \gg 0$. 
\end{theoremB*}

% this uses really simple stuff. we don't really make estimates. Just measuring the difference from our Cartier algebra and all the maps on the product. 

The main barrier to applying this new subadditivity formula in practice is the computation of $\ConR$. We give a method for doing so in the affine toric case in section 6. Recall that, for any toric ring $R$ over a field $k$, one can construct the so-called anti-canonical polytope $\acptope$, where $X = \spec R$. Then the set $\homgp_R(F^e_*R, R)$ is generated as a $k$-vector space by maps $\pi_a$ corresponding to fractional lattice points $a \in \frac{1}{p^e}\bZ^n \cap \operatorname{int}(\acptope)$. We determine which of these generators belong to $\ConR$ in \autoref{prop:resCriterion}:
\begin{theoremC*}[\autoref{prop:resCriterion}]
  Let $R$ be a toric ring with anticanonical polytope $\acptope$. Then $\plainConR_e(R)$ is generated as a $k$-vectorspace by the maps $\pi_a$ where $a\in \frac{1}{p^e}\bZ^n \cap \operatorname{int}(\acptope)$ and the interior of $\acptope \cap\left( a - \acptope \right)$ contains a representative of each equivalence class in $\frac{1}{p^e}\bZ^n / \bZ^n$. 
\end{theoremC*}

In Section 2, we review some background information on test ideals and Cartier algebras. In Section 3 we introduce our new subadditivity formula and prove Theorem A. In Section 4, we discuss adjoint ideals in characteristic 0 and introduce a characteristic-$p$ analog which generalizes earlier constructions found in  \cite{TakagiAdjointHighCodim}, \cite{SchwedeFadj}, and \cite{BSTZ_discreteness_and_rationality}. We use these constructions in Section 5 to prove Theorem B. This is accomplished by showing that adjoint ideals in characteristic 0 reduce modulo $p\gg 0$ to our analogous test ideal (\autoref{thm:tauEqualsJ}). Finally, in Section 6, we give a combinatorial criterion for computing $\ConR$ in the toric setting. 

\textbf{Acknowledgements:} I'd like to thank my advisor, Karl Schwede, for his numerous insights and generous support. I'd also like to thank Adam Boocher, Javier Carvajal-Rojas, Rankeya Datta, James Farre, Elo\' isa Grifo, Jack Jeffries, Linquan Ma, Takumi Murayama, Thomas Polstra, and Derrick Wigglesworth for their willingness to answer questions and play math together. Thanks in particular to Srikanth Iyengar for pointing out a slick proof of \autoref{cor:homAndTensor}.
Many thanks to the referee for their numerous helpful comments, and especially for pointing out an error in an earlier proof of \autoref{thm:hardContainment}.

\section{Background on Cartier algebras and  test ideals}
In this section, we provide the basic definitions and results from the theory of test ideals and Cartier algebras. A more complete account of these theories may be found in the surveys, \cite{SchwedeTuckerSurvey, BlickleSchwedeSurvey}. 

Let $R$ be a Noetherian ring of characteristic $p>0$. We let $F\colon R \to R$ denote the Frobenius map and let $F^e$ denote the $e^{\textrm{th}}$ iterate of $F$. In particular, $F^e(x) = x^{p^e}$ for all $x \in R$. We define $F^e_*R$ to be the $R$-module given by restriction of scalars via $F^e$. In other words, $F^e_*R \coloneqq \left\{ F^e_* r \suchthat r \in R \right\}$ as a set, and the $R$-module structure on $F^e_* R$ is given by  $s F^e_* r = F^e_* s^{p^e}r$ for all $r,s \in R$.  We give $\homgp_R(F^e_* R, R)$ an $F^e_* R$ module structure defined by pre-multiplication. That is, we define:
  \[
    (F^e_* r \cdot \vp)(F^e_* x) \coloneqq \vp(F^e_*(rx))
  \]
  Similarly, we define a right $R$-module structure on $\homgp_R(F^e_*R, R)$ by setting $\vp \cdot r\coloneqq F^e_* r\vp$. The left $R$-module structure on $\homgp_R(F^e_*R, R)$ is given by post-multiplication, namely $r\cdot \vp(F^e_* x) \coloneqq r\vp(F^e_*x)$.  Note that $r\cdot \vp = \vp \cdot r^{p^e}$ for all $r$ and $\vp$. 
  Note also that the Frobenius map gives us a map of left $R$-modules, $F^e\colon R \to F^e_* R$, defined as $x \mapsto F^e_* x^{p^e}$. 
  \begin{remark}
    In the literature, people often use the notation $R^{1/p^e}$ instead of $F^e_* R$. If $R$ is a domain, we can identify $R^{1/p^e}$ with the set of $(p^e)^{\textrm{th}}$ roots of elements of $R$ in a fixed algebraic closure of $\operatorname{frac}(R)$. Even if $R$ is just reduced, the Frobenius map is injective and we have $R\subseteq R^{1/p^e}$. 
  \end{remark}

  \begin{definition}
    A ring $R$ is said to be $F$-finite if $F^e_*R $ is a finitely-generated $R$-module for some (equivalently, all) $e>0$. 
  \end{definition}
  For instance, perfect fields are $F$-finite. Further, any algebra essentially of finite type over an $F$-finite ring is $F$-finite.  
 \begin{globalsetting} 
  We will assume in this paper that all of our rings are $F$-finite and Noetherian.  % highlight. Maybe also assume reduced?
 \end{globalsetting}

  Perhaps the central idea in the study of so-called ``$F$-singularities'' is that the non-regularity of a ring $R$ in positive characteristic can be understood by studying the modules $F^e_*R$. This theory was initiated by the following result of Kunz. 
  \begin{theorem}[\cite{Kunz69}]
    Let $R$ be Noetherian a ring of positive characteristic. Then $R$ is regular if and only if $F^e_*R$ is a flat $R$-module for some (equivalently, all) $e > 0$.  
    \label{thm:kunz}
  \end{theorem}

  In particular, if $R$ is an $F$-finite local ring, then $R$ is regular if and only if $F^e_* R$ is a \emph{free} $R$-module. One way to detect the freeness of  $F^e_* R$ is by counting the number of distinct splittings of the inclusion $R \subseteq F^e_* R$---there should be many such splittings if $F^e_* R$ is free. This in turn can be detected by studying the \emph{test ideal} of $R$, though the reasons why may be unclear to the uninitiated. 

  \begin{definition}
    Let $R$ be a reduced ring of positive characteristic. The \emph{test ideal} of $R$, denoted $\tau(R)$,  is the unique, smallest ideal $J$ containing a regular element such that $\vp(F^e_* J) \subseteq J$ for all $e$ and all $\vp \in \homgp_R(F^e_* R, R)$. 
  \end{definition}
  Proving this ideal exists is non-trivial. Whenever $R$ is regular, we have $\tau(R) = R$. The converse is not true, however. Thus, we say that a ring $R$ is \emph{$F$-regular} if $\tau(R) = R$. One philosophy in the study of $F$-singularities is that rings with milder singularities have larger test ideals. 

  We can obtain variants of the test ideal by restricting the set of maps $\vp$ under consideration. One of the first instances of this idea was the construction of test ideals of pairs $(R, \mathfrak a^t)$. Here, $\mathfrak a$ is an ideal of $R$ and $t \geq 0$ is a real number. We don't actually define the power $\mathfrak a^t$ for arbitrary real numbers $t$. Howevever, this formal notation for the test ideal of a pair turns out to be quite useful. Indeed, it is unambiguous when $t$ is an integer. 

  \begin{definition}[\cite{HaraYoshidaSubadd}]
    Let $R$ be reduced,  $\mathfrak a\subseteq R$ an ideal containing a regular element, and $t$ a positive real number. Then we define $\tau(R, \mathfrak a^t)$ to be the unique, smallest ideal $J$ containing a regular element such that $\vp(F^e_* J )\subseteq J$ for all $e$ and all $\vp \in F^e_* \left( \mathfrak a^{\ceil{t(p^e-1)}} \right) \homgp_R \left( F^e_* R, R \right)$. 
  \end{definition}

  % triples? many ideals at once?
  People have also considered test ideals of triples $(R, \mathfrak a^t, \Delta)$, where one further restricts the maps $\vp$ under consideration using a divisor $\Delta$ on $\spec R$. This leads one to ask, more generally, for which sets of maps $F^e_* R \to R$ can one define the test ideal of that set? The answer is that this set of maps must form a \emph{Cartier Algebra}. 
  \begin{definition}[\cite{BlickleCartier, SchwedeNonQGor}]
    A \emph{Cartier algebra} on $R$ is an additive abelian group $\cartalg C = \bigoplus_e \cartalg C_e$, with $\cartalg C_e \subseteq \homgp_R (F^e_* R, R)$ for all $e$, that is closed under multiplication on the left and right by elements in $R$ and closed under composition.  In other words, given $\vp_1, \vp_2 \in \cartalg C_e$,  $\psi \in \cartalg C_d$, and $r \in R$, we have: $\vp_1+ \vp_2 \in \cartalg C_e$,  $r\cdot \vp \in \cartalg C_e$, $\vp \cdot r \in \cartalg C_e$, and $ \vp \cdot \psi \coloneqq \vp\circ (F^e_* \psi) \in \cartalg C_{d+e}$. By convention, we also assume that $\cartalg C_0 = R$. The \emph{full Cartier algebra} on $R$ is the Cartier algebra $\fullCA{R} \coloneqq \bigoplus_{e \geq 0} \homgp_R(F^e_*R, R).$
    \end{definition}
  Here, $F^e_* \psi$ denotes the map $ F^{e+d}_* R \to F^e_* R$ given by 
  \[
    (F^e_* \psi)\left( F_*^{e+d} x \right) \coloneqq F^e_* \psi(F^d_* x)
  \]
  We note that Cartier algebras are typically not commutative rings. Further, they're not necessarily $R$-algebras, in the sense that $R$ is typically not in the center of a given Cartier algebra. 
%  \begin{remark}
%    A map $F^e_* R \to R$ is the same information as an additive map $f\colon R\to R$ satisfying  $f(r^{p^e}x) = r f(x)$ for all $r, x \in R$. Such maps are called $p^{-e}$-linear. Thus, we could define a Cartier algebra as a collection $\cartalg C$ of $p^{-e}$-linear maps for various $e$ that's closed under addition, composition of maps, and multiplication by elements of $R$. Note that the composition of a $p^{-e}$-linear map with a $p^{-d}$-linear map is always a $p^{-e-d}$-linear map. 
%  \end{remark}

  Let $\Psi = \sum_{i} \psi_i \in \cartalg C$, where the sum is finite, each $\psi_i$ is nonzero, and $\psi_i \in \cartalg C_{e_i}$ for each $i$. Then $R$ has a natural left $\cartalg C$-module structure given by 
    \[
      \Psi \cdot r \coloneqq \sum_i \psi_i\left( F^{e_i}_* r \right)
    \]
  for each $r\in R$. Further, we say $\Psi$ has \emph{minimal degree $e_0$} if $e_0 = \min_i\left\{ e_i \right\}$. 

  If $\cartalg C$ is a Cartier algebra on a reduced ring $R$, then the test ideal $\tau(R, \cartalg C)$, if it exists, is defined to be the smallest ideal $J \subseteq R$ containing a regular element such that $\vp(F^e_* J) \subseteq J$ for all $e$ and for all $\vp \in \cartalg C_e$. Schwede showed in \cite{SchwedeNonQGor} that these test ideals exist whenever $\cartalg C$ is \emph{nondegenerate}\footnote{A map $\vp\colon F^e_*R \to R$ is called nondegenerate if $\vp(F^e_* R) R_\eta \neq 0$ for all minimal primes $\eta \in \spec R$. A Cartier algebra $\cartalg C$ is called nondegenerate if $\cartalg C_e$ contains a nondegenerate map for some $e > 0$. }. This is a very weak condition. For instance, if $R$ is a domain, then a Cartier algebra $\cartalg C$ on $R$ is nondegenerate whenever $\cartalg C_e \neq 0$ for some $e > 0$. 

  We can also define $\tau(R, \cartalg C, \prod_i \mathfrak a_i^{t_i} )$ for any nonzero ideals $\mathfrak a_i \subseteq R$ and real numbers $t_i \geq 0$: this is the smallest nonzero ideal $J \subseteq R$ such that $\vp(F^e_* J) \subseteq J$ for all $e$ and for all $\vp \in F^e_*\left( \prod_i \mathfrak a_i^{\ceil{t_i(p^e-1)}} \right) \cartalg C_e$. As before, this product $\prod_i \mathfrak a_i^{t_i} $ is just formal (though certainly useful) notation. 

  \begin{notation}
    Given a Cartier algebra $\cartalg C$ on $R$, a collection of nonzero ideals $\mathfrak a_i \subseteq R$, and rational numbers $t_i \geq 0$, we define 
    \[
      \cartalg C^{R, \prod_i \mathfrak a_i^{t_i}} \coloneqq 
      \bigoplus_{e \geq 0} F^e_* \left( \prod_i \mathfrak a_i^{\ceil{t_i(p^e-1)}} \right) \cartalg C_e.
    \]
    Further,  we define $\tau(R, \prod_i \mathfrak a_i^{t_i}) \coloneqq \tau(R, \fullCA{R}, \prod_i \mathfrak a_i^{t_i})$. In other words, when we omit $\cartalg C$ from the test ideal notation, $\cartalg C$ is understood to be the full Cartier algebra $\fullCA{R}$.
  Note that $\tau(R, \cartalg C, \prod_i \mathfrak a_i^{t_i} ) = \tau\left(R,  \cartalg C^{R, \prod_i \mathfrak a_i^{t_i}}  \right)$. We will stick to the first notation, mainly for historical reasons. 
  \end{notation}

  \begin{theorem}[\cf~\cite{SchwedeNonQGor}]
    Let $R$ be reduced, $\cartalg C$ a non-degenerate Cartier algebra on $R$, and $c \in \tau(R, \cartalg C)$ a regular element. Then 
    \[
      \tau(R, \cartalg C) = \sum_{e \geq 0} \sum_{\vp \in \cartalg C_e} \vp\left( F^e_* c \right).
    \]
    \label{thm:regTauIsSum}
  \end{theorem}
  %If $R$ is Gorenstein 
    %\item what happens in the Gorenstein case
    %\item What happens in the $\bQ$-gorenstein case
  % maybe compare this with matlis duality definition? talk about history of this construction ?

  Given any map $\vp\colon F^e_* R \to R$ and any multiplicative set $W\subseteq R$, we get an induced map $W\invrs \vp\colon W\invrs F^e_* R \to W\invrs R$. Note that $W\invrs F^e_* R = F^e_*\left( W\invrs R \right)$. Thus, any Cartier algebra $\cartalg C$ on $R$ induces a cartier algebra $W\invrs \cartalg C$ on $W\invrs R$. 

  \begin{notation}
    Let $\cartalg C$ be a Cartier algebra on $R$ and $\gamma \in R$ a regular element. Then $\cartalg C$ induces a Cartier algebra on the localization $R_\gamma$. We denote this induced Cartier algebra by $\cartalg C_\gamma$. 
  \end{notation}
  An important notion in the study of test ideals is that of  \emph{compatibility} of ideals and Cartier algebras. 
\begin{definition}
   Let $S$ be a ring, $J$ an ideal of $S$, and let $\vp\colon F^e_* S \to S$. We say  that $J$ is \emph{compatible} with $\vp$, or \emph{$\vp$-compatible}, or that $\vp$ is compatible with $J$, if $\vp(F^e_* J) \subseteq J$. Let $\cartalg D$ a Cartier algebra on $S$. Similarly, we say that $J$ is compatible with $\cartalg D$, or $\cartalg D$-compatible, or that $\cartalg D$ is compatible with $J$, if $J$ is compatible with each map in $\cartalg D$. 
 \end{definition}

  Finally, we collect some well-known and useful properties of test ideals. 

  \begin{lemma}[\cf~\cite{HHbasechange, HaraYoshidaSubadd, SchwedeNonQGor}]
    \label{lemma:wellKnown}
    Let $R$ be a reduced $F$-finite ring. 
    \begin{enumerate}
      \item Let $\cartalg C \subseteq \cartalg D$ be two non-degenerate Cartier algebras on $R$. Then $\tau(R, \cartalg C) \subseteq \tau(R, \cartalg D)$. 
      \item Let $\mathfrak a_i$ be a collection of ideals of $R$ and $t_i \geq 0$ a collection of rational numbers. For each $i$, let $\mathfrak b_i  = \overline {\mathfrak a}_i$ be the integral closure of $\mathfrak a_i$. Then $\tau(R, \prod_i \mathfrak a_i^{t_i}) = \tau(R, \prod_i \mathfrak b_i^{t_i})$. 
      \item Let $W\subseteq R$ be a multiplicative set consisting of regular elements. Then $W\invrs \tau(R, \cartalg C) = \tau(W\invrs R, W\invrs \cartalg C)$. 
      \item  The Cartier algebra $\fullCA R$ is non-degenerate.  \label{part:nondegen}
    \end{enumerate}
  \end{lemma}

%A couple more definitions that we'll need:
%\begin{definition}
%  A Cartier algebra $\cartalg C$ on a ring $R$ is said to be \emph{strongly $F$-regular} if $\tau(R, \cartalg C) = R$. Equivalently, if for all $c\in R$ there exists $e \in \bN$ and $\vp \in \cartalg C_e$ such that $\vp(F^e_* c) = 1$. 
%\end{definition}
%
% define ``center'' of a cartier algebra? Maybe I should just say ``compatible with the maps in C''
%$F$-regular cartier algebras?

\section{Diagonal Cartier Algebras and Subadditivity}
Let $R$ be a noetherian ring in characteristic $p$ and $I \subseteq R$ an ideal. Then we have  $F^e_* \left( R/I \right) = F^e_* R / F^e_* I$. Thus, each map $\vp\colon F^e_* R \to R$ satisfying $\vp(F^e_* I) \subseteq I$ induces a map $\overline \vp\colon F^e_* (R/I) \to R/I$. We can do something similar for Cartier algebras:
\begin{definition}
  Let $\cartalg D$ be a Cartier algebra on $R$ compatible with an ideal $I \subseteq R$. We define the \emph{restriction of $\cartalg D$ to $R/I$}, denoted $\cartalg D|_{R/I}$, to be the set of maps $\bigoplus_{e\geq 0} \cartalg D_e |_{R/I}$, where 
  \[
    \cartalg D_e |_{R/I} \coloneqq \left\{ \overline \vp\colon F^e_*(R/I) \to R/I \suchthat \vp\in \cartalg D_e \right\}.
  \]
\end{definition}
\begin{proposition}
  Let $\cartalg D$ be a Cartier algebra on $R$ compatible with an ideal $I \subseteq R$. Then $\cartalg D |_{R/I}$ is a Cartier algebra on $R/I$. 
\end{proposition}
\begin{proof}
  Let $\vp_1, \vp_2 \in \cartalg D_d |_{R/I}$,  $\psi \in \cartalg D_e |_{R/I}$, and $r\in R/I$. Then there exist some $\vp_1', \vp_2' \in \cartalg D_d$ with $\overline{\vp_1'} = \vp_1$ and $\overline{\vp_2'} = \vp_2$. As $\cartalg D$ is a Cartier algebra, we have $\vp_1' + \vp_2' \in \cartalg D_d$, and we see that
  \[
    \vp_1 + \vp_2 = \overline{\vp_1' + \vp_2'} \in \cartalg D_d |_{R/I}.
  \]
  A similar argument shows that $r \psi, \psi r \in \cartalg D_e |_{R/I}$ and $\vp_1 \circ F^d_* \psi \in \cartalg D_{d+ e} |_{R/I}$. It follows from the definitions that $\cartalg D_0 |_{R/I} = R/I$ provided that $\cartalg D_0 = R$. 
\end{proof}

We define another useful operation on Cartier algebras. 
\begin{definition}
  Let $\cartalg C$ be a Cartier algebra on $R$ and let $I\subseteq R$ be an ideal. We define the \emph{subalgebra of maps compatible with $I$}, denoted $\cartalg C^{I \compSubAlg}$, to be the set of maps $\bigoplus_{e \geq 0} \cartalg C^{I \compSubAlg}_e$, where
  \[
    \cartalg C^{I \compSubAlg}_e \coloneqq \left\{\vp \suchthat \vp \in \cartalg C_e, \; \vp(F^e_* I) \subseteq I \right\}.
  \]
\end{definition}
\begin{proposition}
   Let $\cartalg C$ be a Cartier algebra on $R$ and $I \subseteq R$ an ideal. Then $\cartalg C^{I \compSubAlg}$ is a Cartier algebra. 
\end{proposition}
\begin{proof}
  Suppose $\vp \in \cartalg C_e $ and $\psi\in \cartalg C_d$ are two maps satisfying $\vp\left( F^e_* I  \right) \subseteq I$ and $\psi\left( F^d_* I \right) \subseteq I$. Clearly, for all $x\in R$, $x \vp(F^e_*I )\subseteq I$ and $\vp(F^e_* x I)\subseteq I$. It's also clear that $\vp(F^e_* I) + \psi(F^d_* I) \subseteq I$. Further, 
  \[
    \vp \circ F^e_* \psi\left( F^{e+d}_* I \right) \subseteq \vp\left( F^e_* I  \right)\subseteq I.
  \]
  Finally, note that $\cartalg C_0^{I \compSubAlg} = R$ whenever $\cartalg C_0 = R$. 
\end{proof}

The next lemma is the  key ingredient in proving our new subadditivity formula. 
\begin{lemma}
  For any reduced ring $R$, Cartier algebra $\cartalg C$ on $R$, and radical ideal $I \subseteq R$ we have 
  \[
    \tau\left( R/I, \cartalg C^{I \compSubAlg} |_{R/I} \right) \subseteq \tau\left( R, \cartalg C \right) R/I,
  \]
  provided that the right-hand side contains a regular element of $R/I$ and the test ideal on the left, $\tau\left( R/I, \cartalg C^{I \compSubAlg} |_{R/I} \right)$, exists. 
  %\label{prop:tauUnderRestriction}
  \label{almostRestriction}
\end{lemma}
\begin{proof}
  Let  $\vp \in  \cartalg C_e |_{R/I}$. By definition there exists a lifting $\hat \vp \in \cartalg C_e$, so that the diagram commutes:
  \[
    \xymatrix{
      F^e_*R  \ar[r]^{\widehat \vp} \ar[d]_{F^e_* \pi} &  R \ar[d]^{\pi} \\
      F^e_*(R/I) \ar[r]^{\vp} & R/I
    }
  \]
  This means that 
  \[
    \vp\left( F^e_* \tau\left( R, \cartalg C \right) R/I\right) = \hat \vp\left(F^e_* \tau\left( R, \cartalg C \right) \right) R/I
  \]
  By definition of $\tau\left( R, \cartalg C \right)$, we see the right hand side is contained in $\tau\left( R, \cartalg C \right) R/I$. Then we are done by the minimality of $\tau\left( R/I, \cartalg C |_{R/I} \right)$. 
\end{proof}
Note that the above lemma does not apply when $\cartalg C$ is compatible with $I$, for then we would have $\tau(R, \cartalg C) \subseteq I$. This motivates  \autoref{def:tauAlongI}, \cf~\autoref{prop:restriction}. 

\begin{proposition}
  For all reduced rings $R$, Cartier algebras $\cartalg C$ on $R$, formal products $\prod_i \mathfrak a_i^{t_i}$ of ideals on $R$, and radical ideals $I \subseteq R$, we have 
  \[
    \tau\left( R/I, \cartalg C^{I \compSubAlg} |_{R/I}, \prod_i (\mathfrak a_iR/I)^{t_i} \right) \subseteq \tau\left( R, \cartalg C, \prod_i \mathfrak a_i^{t_i} \right) R/I,
  \]
  provided that the right-hand side contains a regular element of $R/I$ and the test ideal on the left-hand side exists. 
  \label{prop:tauUnderRestriction}
\end{proposition}
\begin{proof}
  Let $\cartalg D$ be a Cartier algebra on $R$ compatible with $I$. Then we have
\[
  \left( \cartalg D |_{R/I} \right)^{\prod_i (\mathfrak a_i R/I)^{t_i}} \subseteq \left( \cartalg D^{ \prod_i \mathfrak a_i^{t_i}} \right)\Big |_{R/I}.
\]
It follows that
\[
  \tau\left( R/I, \cartalg C^{I \compSubAlg} |_{R/I}, \prod_i (\mathfrak a_iR/I)^{t_i} \right) \subseteq 
  \tau\left( R/I, \left(\cartalg C^{I \compSubAlg}\right)^{\prod_i \mathfrak a_i^{t_i}} \Big |_{R/I} \right),
\]
by \autoref{lemma:wellKnown}. Similarly, we note that
\[
  \left(\cartalg C^{I \compSubAlg}\right)^{\prod_i \mathfrak a_i^{t_i}}  \subseteq \left(\cartalg C^{\prod_i \mathfrak a_i^{t_i}}\right)^{I \compSubAlg}.
\]
Then we get
\[
  \tau\left( R/I, \left(\cartalg C^{I \compSubAlg}\right)^{\prod_i \mathfrak a_i^{t_i}} \Big |_{R/I} \right) \subseteq
  \tau\left( R/I, \left(\cartalg C^{\prod_i \mathfrak a_i^{t_i}}\right)^{I \compSubAlg} \Big |_{R/I} \right) \subseteq
  \tau\left( R, \cartalg C, \prod_i \mathfrak a_i^{t_i} \right) R/I,
\]
where the second containment follows from \autoref{almostRestriction}.
\end{proof}

We obtain our subadditivity formula by applying \autoref{prop:tauUnderRestriction} to the case where we consider the ideal $I_\Delta \coloneqq \ker(R \otimes_k R \xrightarrow{\mu} R)$. First, we introduce some notation which will be used throughout the rest of this paper:
\begin{notation} Let $R$ be a $k$-algebra essentially of finite type, where $k$ is a perfect field of positive characteristic.
  \begin{itemize}
    \item $\mu\colon R\otimes_k R \to R$ is the map given by $x \otimes y \mapsto xy$.
    \item $I_\Delta \subseteq R\otimes_k R$ denotes the kernel of $\mu$. Geometrically, $I_\Delta$ cuts out the diagonal embedding $\spec R \subseteq \spec R\times_{\spec k} \spec R$. Note that we can generate $I_\Delta$ as an ideal in the following way: $I_\Delta = \braket{x \otimes 1 - 1 \otimes x  \suchthat x \in R}$.
    \item We let $\fullCA{R\otimes_k R,\, I_\Delta \compSubAlg}\coloneqq ( \fullCA{R\otimes_k R})^{I_\Delta \compSubAlg}$ denote the Cartier algebra on $R\otimes_k R$ of all maps compatible with $I_\Delta$. We say that such maps are \emph{compatible with the diagonal}.
    \item We define the \emph{second diagonal Cartier algebra on $R$} to be 
      \[
        \ConR\coloneqq \fullCA{R\otimes_k R, I_\Delta \compSubAlg}\big|_{(R\otimes_k R)/I_\Delta}.
      \]
      If the ring $R$ is understood from context, we will denote this Cartier algebra simply as $\plainConR$.
  \end{itemize}
\end{notation}
\begin{remark}
  In particular, $\plainConR_e(R)$ is the set of maps $\vp\colon F^e_* R \to R$ that admit a lifting to the tensor product $R\otimes_k R$:
  \[
    \xymatrix{F^e_*\left( R\otimes_k R \right) \ar@{-->}[r] \ar[d]_{F^e_* \mu} & R\otimes_k R \ar[d]^\mu \\
             F^e_* R \ar[r]^\vp & R \\
    }
  \]
  The notation $\mathscr D^{(2)}$ is meant to suggest that one can define $\mathscr D^{(n)}$ as the Cartier algebra of maps on $R$ that lift to the $n$-fold tensor product $R^{\otimes n}$. This theory
  will be developed, along with applications to symbolic powers, in future work joint with Javier Carvajal-Rojas \cite{USTPDiagFReg}. 
\end{remark}
Before we may proceed, we need to recall a general fact about modules:
\begin{lemma}
  Let $R$ and $S$ be commutative algebras over a field $k$. Let $M$ and $N$ be $R$-modules and let $U$ and $V$ be $S$-modules. Suppose also that $M$ and $U$ are finitely presented over their respective rings. Then the canonical map:
  \[
    \Theta\colon \homgp_R(M, N) \otimes_k \homgp_S(U, V) \to \homgp_{R\otimes_k S}\left( M\otimes_k U, N\otimes_k V \right)
  \]
  is an isomorphism. 
  \label{thm:homAndTensor}
\end{lemma}
\begin{proof}
  We have the following chain of natural isomorphisms: 
  \begin{align}
    \homgp_R(M, N) \otimes_k \homgp_S(U, V) &\cong \homgp_R\left( M, N \otimes_k \homgp_S(U, V) \right) \label{eqn1}\\
    & \cong \homgp_R\left( M, \homgp_S(U, N\otimes_k V) \right) \label{eqn2}\\
    & \cong \homgp_R\left( M, \homgp_S(U, \homgp_{R\otimes_k S}(R\otimes_k S, N\otimes_k V)) \right) \\
    & \cong \homgp_R\left( M, \homgp_{R\otimes_k S}\left( R\otimes_k U, N\otimes_k V \right) \right) \label{eqn3} \\
    & \cong \homgp_{R\otimes_k S}\left( M\otimes_R R\otimes_k U, N\otimes_k V \right)\label{eqn4} \\
    & \cong \homgp_{R\otimes_k S}\left( M\otimes_k U, N\otimes_k V \right) 
  \end{align}
  The isomorphism in \autoref{eqn1} follows from the facts that $M$ is finitely presented and $\homgp_S(U,V)$ is a flat $k$-module (\cf~\cite[Chapter XVI, Exercise 11]{LangAlgebra}). The isomorphism in \autoref{eqn2} follows by the same argument. The isomorphisms in \autoref{eqn3} and \autoref{eqn4} follow from Hom-Tensor adjunction.
\end{proof}
\begin{corollary}
  Let $k$ be a perfect field of characteristic $p$ and let $R$ and $S$ be $k$-algebras essentially of finite type. Then the canonical map
  \[
    \Theta\colon \homgp_R(F^e_* R, R) \otimes_k \homgp_S (F^e_*S, S) \to \homgp_{R\otimes_k S}(F^e_*(R \otimes_k S), R\otimes_k S)
  \]
  is an isomorphism.
  \label{cor:homAndTensor}
\end{corollary}
\begin{proof}
  As $R$ is $F$-finite and Noetherian, we see that $F^e_*R$ is a finitely-presented $R$-module. The same goes for $S$. Then this result follows from the above lemma. Note that, since $k$ is perfect, we have $F^e_* R \otimes_k F^e_*R = F^e_*\left( R\otimes_k R \right)$ and similarly for $S$. 
\end{proof}

\begin{theorem}
  Let $k$ be a perfect field of positive characteristic and let $R$ be a reduced $k$-algebra essentially of finite type. Then 
  \[
    \tau(R, \plainConR, \mathfrak a^s \mathfrak b^t) \subseteq \tau(R, \mathfrak a^s) \tau(R, \mathfrak b^t)
  \]
  for all ideals $\mathfrak a, \mathfrak b \subseteq R$ and real numbers $s,t \geq 0$, provided that neither $\mathfrak a$ nor $\mathfrak b$ consists of zero-divisors.
  \label{thm:MySubadd}
\end{theorem}
\begin{proof}
  From \autoref{cor:homAndTensor} it's easy to see that
  \[
    \tau(R\otimes_k R, (\mathfrak a\otimes_k R)^s \cdot (R \otimes_k \mathfrak b)^t) \subseteq \tau(R, \mathfrak a^s) \otimes_k \tau(R, \mathfrak b^t) 
  \]
  by the minimality of the test ideal on the left. If we mod out the above equation by $I_\Delta$, we get $\tau(R, \mathfrak a^s)  \tau(R, \mathfrak b^t)$ on the right-hand side. By \autoref{prop:tauUnderRestriction}, we're done if we can show that
  \[
    \tau(R\otimes_k R, (\mathfrak a\otimes_k R)^s \cdot (R \otimes_k \mathfrak b)^t) \cdot(R\otimes_kR)/ I_\Delta
  \]
  contains a regular element and that $\ConR$ is non-degenerate. To that end, let $f \in \mathfrak a$ and $g\in \mathfrak b$ be regular elements such that $R_f$ and $R_g$ are regular. As we know $(R\otimes R)_{f \otimes g} = R_f \otimes  R_g$, we see that
  \begin{align*}
    \tau(R\otimes_k R, (\mathfrak a\otimes_k R)^s \cdot (R \otimes_k \mathfrak b)^t) (R\otimes R)_{f\otimes g} &= \tau(R_f \otimes_k R_g, (\mathfrak aR_f \otimes_k R_g)^s \cdot (R_f \otimes_k \mathfrak bR_g)^t)\\
   &= \tau(R_f \otimes_k R_g) = R_f \otimes_k R_g,
  \end{align*}
  where the last equality follows from the regularity of $R_f \otimes_k R_g$. Thus, we have
  \[
    (f\otimes g)^n \in \tau(R\otimes_k R, (\mathfrak a\otimes_k R)^s \cdot (R \otimes_k \mathfrak b)^t),
  \]
  for some $n$, and so 
    $(fg)^n \in \tau(R\otimes_k R, (\mathfrak a\otimes_k R)^s \cdot (R \otimes_k \mathfrak b)^t)\cdot (R\otimes_k R)/I_\Delta$. 

    It remains to check that $\ConR$ is non-degenerate. We know that $\fullCA R_e$ contains a non-degenerate map $\vp$ for some $e>0$, by \autoref{lemma:wellKnown}. Further, since $R$ is reduced, we know that $R_\eta$ is regular for all minimal primes $\eta \in \spec R$ and that every zero-divisor of $R$ is contained in a minimal prime. As the singular locus of $\spec R$ is Zariski-closed, it follows by prime avoidance that there exists a regular element $f\in R$ such that $R_f$ is regular. As $R_f\otimes_k R_f$ is regular\footnote{Note that this relies on the fact that $k$ is perfect. Indeed, $R\otimes_k R$ need not be regular if $k$ is not perfect, see \cite{MOtensor}. On the other hand, since $k$ is perfect, we know that $R_f$ is in fact \emph{smooth}, 
    %meaning $\Omega_{R_f/k}$ is a free $R_f$-module. Thus  $\Omega_{R_f\otimes_k R_f/k} = (\Omega_{R_f/k}\otimes_k R_f) \oplus (R_f \otimes_k \Omega_{R_f/k})$ is a free $R_f\otimes_k R_f$-module  and $R_f\otimes_k R_f$ is also smooth, and therefore regular. }
  and the tensor product of two smooth $k$-algebras over $k$ is smooth.}
    and $F$-finite, we have $F^e_*(R_f \otimes_k R_f)$ is a  projective $R\otimes_k R$-module, by \autoref{thm:kunz}. It follows that there exists some $\hat \vp \in \homgp_{R\otimes_k R}(F^e_* (R\otimes_k R), R\otimes_k R)_{f\otimes f}$ such that the diagram,
  \[
    \xymatrix{F^e_*\left( R_f\otimes_k R_f \right) \ar@{-->}[r]^{\hat \vp} \ar[d]_{F^e_* \mu} & R_f\otimes_k R_f \ar[d]^\mu \\
             F^e_* R_f \ar[r]^\vp & R_f \\
    }
  \]
  commutes. Thus there exists some $N$ with $(f\otimes_k f)^N \hat \vp \in\homgp_{R\otimes_k R}(F^e_* (R\otimes_k R), R\otimes_k R)$. We see that $f^{2N} \vp$ is a non-degenerate element of $\plainConR_e(R)$, as desired. 
\end{proof}
%Since $R$ is reduced, we know that $R_\eta$ is regular for all minimal primes $\eta \in \spec R$. This implies that $\jac(R)$ is not contained in any minimal prime, so $\jac(R)$ contains a regular element. Call this element $\xi$. We also know that $\fullCA R$ contains a non-degenerate map $\vp$ by \autoref{lemma:wellKnown}. Then $\xi \vp$ is non-degenerate, and by \autoref{thm:jacThm}, we have $\xi \vp \in \ConR$. 

The next theorem shows that this subadditivity formula is sharper than the one found in \cite{TakagiSingMultIdeals}.

\begin{theorem}
  Let $k$ be a perfect field of positive characteristic and let $R$ be a $k$-algebra essentially of finite type. Suppose also that $R$ is equidimensional and reduced. Then we have $\jac(R) \fullCA R \subseteq \ConR$. In particular, 
  \[
    \jac(R) \tau\left(R, \prod_i \mathfrak a_i^{t_i} \right) \subseteq \tau\left(R, \plainConR, \prod_i \mathfrak a_i^{t_i} \right)
  \]
  for all formal products of ideals $\prod \mathfrak a_i^{t_i}$ such that each $\mathfrak a_i$ contains a regular element.
   \label{thm:jacThm}
 \end{theorem}
 \begin{proof}
   For any multiplicative subset $W\subseteq R$, one checks that 
   \[
     W\invrs(\jac(R) \fullCA R) = \jac(W\invrs R)\fullCA{W\invrs R}
   \]
   and  $W\invrs \ConR \subseteq \plainConR(W\invrs R)$. Thus we may assume that $R$ is a finitely generated $k$-algebra.
   
   Next, we reduce to the case that $k$ is infinite. Suppose that  $k$ is finite, let $t$ be an indeterminate over $k$, and let $L = k(t^{1/p^{\infty}})$ be the perfection of $k(t)$. Set $R_L = R \otimes_k L$ and suppose $\jac(R_L/L) \fullCA{R_L} \subseteq \plainConR(R_L)$. Let $e \geq 0$ and set $q = p^e$. As $L$ is perfect, %Let $\iota \colon L^{1/q} \to L$ denote the inclusion map. Note that $\iota$ is a map of $L$-modules. In other words, $\iota$ is an element of $\fullCA{L}_e$. 
   any map $\vp\in \fullCA{R}_e$ induces a map,
   \[
     \vp_L \coloneqq \left( R\otimes_k L\right)^{1/q} = R^{1/q} \otimes_k L \to R\otimes_k L,
   \]
   in $\fullCA{R_L}_e$. Further, any $x\in \jac{R}$ gives an element $x\otimes_k 1 \in \jac(R_L/L)$. By assumption, we have a lifting,
   \begin{equation}
     \begin{aligned}
     \xymatrix{
       R_L^{1/q} \otimes_L R_L^{1/q} \ar[r]^{\hat \vp} \ar[d]_{\mu_L^{1/q}} & R_L\otimes_L R_L \ar[d]^{\mu_L}\\
       R_L^{1/q} \ar[r]^{(x\otimes 1)\cdot \vp_L} & R_L
     } 
     \label{eq:diagram}
     \end{aligned}
   \end{equation}

   Observe that the inclusion $k \subseteq L$ splits as a map of $k$-modules. Indeed, we can write $k\left( t^{1/p^\infty} \right)$ as the direct limit, 
   \[
     k\left( t^{1/p^\infty} \right) = \varinjlim\left( k(t) \xrightarrow{F} k(t) \xrightarrow{F} k(t) \xrightarrow{F} \cdots  \right)
   \] 
   where $F$ is the Frobenius map. Define the $k$-linear map $\sigma: k(t) \to k$ as follows: any element of $k(t)$ can be written as $t^i \frac{f}{g}$ for some $i \in \bZ$ and some $f,g \in k[t]$. Then we set  
   \begin{equation*}
     \sigma\left( t^i \frac{f}{g} \right) =  \begin{cases} 0, & i \neq 0\\
       \frac{f(0)}{g(0)}, & i = 0
         
       \end{cases}
   \end{equation*}
   As $\sigma = \sigma \circ F$, this induces a map $\widehat \sigma: k\left( t^{1/p^\infty} \right) \to k$.  Note that $\widehat \sigma\left( t^{i/p^\infty} \right) = 0$ for all $i \neq 0$ and that $\widehat \sigma$ acts as the identity on $k$. We consider $k$ to be an $L$-module via $\widehat \sigma$. Then  all we have to do is apply the functor $-\otimes_L k$ to the diagram in equation \autoref{eq:diagram};  this shows that $\widehat \vp \otimes_L k$ is a lifting of $x \vp$ to $\fullCA{R\otimes_k R}_e$, as desired. 

   Now assume that $k$ is infinite. Let $x\in \jac(R)$ and let $\vp \in \fullCA{R}_e$. As $k$ is infinite, we can use \cite[Theorem 3.4]{HHsymbolic} (\cf \cite[Corollary 1.5.4]{HHequalCharZero}) to say there exists a (generically separable) Noether normalization $A\subseteq R$ such that $xR^{1/q} \subseteq A^{1/q}[R]$ and $A^{1/q}[R] \cong A^{1/q} \otimes_A R$. As $A$ is a polynomial ring, we have $A^{1/q}$ is a free $A$-module, and so $A^{1/q}\otimes_A S$ is a free $S$-module for any $A$-algebra $S$. In particular, we have that $A^{1/q}[R] \otimes_k R^{1/q} \cong A^{1/q}\otimes_A R \otimes_k R^{1/q}$ is a free $R\otimes_k R^{1/q}$-module. 
   
   Further, the usual multiplication map
   \[
     \mu\colon R^{1/q} \otimes_k R^{1/q} \to R^{1/q}
   \]
   induces an $R^{1/q} \otimes_k R^{1/q}$-module structure on $R^{1/q}$. By definition we have that $\mu$ is $R^{1/q}\otimes_k R^{1/q}$-linear, and in particular $R\otimes_k R^{1/q}$-linear. As $A^{1/q}[R]\otimes_k R^{1/q}$ and $R\otimes_k R^{1/q}$ are contained in $R^{1/q} \otimes_k R^{1/q}$, the map $\mu$ restricts to $R\otimes_k R^{1/q}$-linear maps,
   \begin{align*}
     \mu\colon & A^{1/q}[R]\otimes_k R^{1/q} \to R^{1/q}\\
     \mu\colon & R\otimes_k R^{1/q} \to R^{1/q}.
   \end{align*}
   It follows that there exists an $R\otimes_k R^{1/q}$-linear (and, \emph{a fortiori}, $R\otimes_k R$-linear) map,
   \[
     \Psi\colon A^{1/q}[R] \otimes_k R^{1/q} \to R\otimes_k R^{1/q},
   \]
   making the following diagram commute:
   \[
     \xymatrix{
       A^{1/q}[R] \otimes_k R^{1/q} \ar@{-->}[r]^{\Psi} \ar[d]_{\mu}& R \otimes_k R^{1/q} \ar[d]^{\mu} \\
       R^{1/q} \ar[r]^{\operatorname{id}} & R^{1/q}
     }
   \]
   The fact that $\vp$ is $R$-linear means that the diagram
   \[
     \xymatrix@C=4em{
       R\otimes_k R^{1/q} \ar[r]^{1 \otimes \vp} \ar[d]_{\mu}& R \otimes_k R \ar[d]^{\mu}\\
       R^{1/q} \ar[r]^{\vp} & R
     }
   \]
   commutes. The map $1\otimes \vp$ is $R\otimes_k R$-linear as $\vp$ is $R$-linear. Finally, we have a commuting diagram 
   \[
     \xymatrix@C=4em{
       R^{1/q} \otimes_k R^{1/q} \ar[r]^{x \otimes 1 \cdot -} \ar[d]_\mu & A^{1/q}[R] \otimes_k R^{1/q} \ar[d]^{\mu} \\
       R^{1/q} \ar[r]^{x\cdot -}  & R^{1/q}
     }
   \]
   where the horizontal maps are given by multiplication. Putting these three diagrams together, we have a commutative diagram,
   \[
     \xymatrix@C=3.5em{
       R^{1/q} \otimes_k R^{1/q} \ar[r]^{x \otimes 1 \cdot -} \ar[d]_\mu & A^{1/q}[R] \otimes_k R^{1/q} \ar[r]^{\Psi} \ar[d]_{\mu}& R\otimes_k R^{1/q} \ar[r]^{1 \otimes \vp} \ar[d]_{\mu}& R \otimes_k R \ar[d]_{\mu}\\
       R^{1/q} \ar[r]^{x\cdot -} &  R^{1/q} \ar[r]^{\operatorname{id}} & R^{1/q} \ar[r]^{\vp} & R
     }
   \]
   where each of the maps in the top row is $R \otimes_k R$-linear. This proves the first assertion. The second assertion follows from \autoref{thm:regTauIsSum}. 
 \end{proof}

\section{Test ideals and multiplier ideals along a closed subscheme}
In this section we introduce two new definitions. The first is a notion of test ideals ``along a closed subscheme.'' This is a generalization of Takagi's generalized test ideal along $I$ \cite{TakagiAdjointHighCodim, BSTZ_discreteness_and_rationality}. We also define a similar generalization of Takagi's adjoint ideal. In the next section, we show that the expected reduction theorem holds: the adjoint ideal reduces to the test ideal mod $p \gg 0$. We use this result to show that our subadditivity formula for test ideals is sharper than the one obtained by reducing \cite[Theorem 6.5]{EisensteinRestrictionThm} mod $p \gg0$. 

    \begin{remark}
      In this section and the next, we will make heavy use of the floor and ceiling functions,  $\floor{\cdot}$ and $\ceil{\cdot}$. It will be helpful to keep in mind the following inequalities: for any  $a,b \in \bR$ we have $\floor{a} + \floor{b} \leq \floor{a+b}$ and $\floor{a-b} \leq \floor{a} - \floor{b}$. Similarly, $\ceil*{a + b} \leq \ceil{a} + \ceil{b}$ and $\ceil{a} - \ceil{b} \leq \ceil{a-b}$. 
    \end{remark}

  \subsection{Test Ideals along a Closed Subscheme}
  For the rest of this subsection, we will be working in the following setting.
  \begin{setting}
    $R$ is a Noetherian $F$-finite domain, $I\subseteq R$ is a prime ideal, and $\cartalg C$ a nonzero Cartier algebra on $R$ such that $I$ is compatible with $\cartalg C$. We assume there exists $e > 0$ and  $\psi \in \cartalg C_e$ with $\psi(F_*^e R) \not \subseteq I$.
    \label{setting:tauI}
  \end{setting}

  \begin{remark}
    We suspect that the constructions in this section can be done just as well in the setting where $I$ is an intersection of different prime ideals of a fixed height. However, for our current purposes it suffices to work in the setting where $I$ is prime. 
    \label{}
  \end{remark}

  \begin{definition}
    Let $R, I, \cartalg C$ be as in  \autoref{setting:tauI}. Then we define \emph{the test ideal of $\cartalg C$ along $I$}, denoted $\tau_I(R, \cartalg C)$, to be the unique smallest ideal of $R$ not contained in $I$ that is compatible with $\cartalg C$. We also call this the \emph{test ideal of $\cartalg C$ along the closed subscheme $\spec(R/I)$}.
 \label{def:tauAlongI}
    \end{definition}
  The proof that $\tau_I(R, \cartalg C)$ exists is a standard though technical argument. We relegate it to \autoref{appendix:tauI}. For now, we just note some examples of interest where the conditions of \autoref{setting:tauI} are satisfied. 

    \begin{example}
      Suppose $R$ is a domain essentially of finite type over a field $k$. Then $I = I_\Delta \subseteq R\otimes_k R$ and $\cartalg C = \ConProd$ satisfy \autoref{setting:tauI}. Indeed, we have that $\cartalg C$ is compatible with $I$ by construction, so we just need to check that there exists $e>0$ and $\vp \in \cartalg C_e$ with $\vp(R) \not \subseteq I$. This is equivalent to checking that $\plainConR_e(R) \neq 0$ for some $e>0$, which follows, for instance,  from \autoref{thm:jacThm}. 
    \end{example}

  \begin{notation}
    Let $\mathfrak a_i$ be a collection of ideals and $t_i$ a set of non-negative real numbers. Then we denote, for all $e$, 
    \[
      \tauOfa{e} \coloneqq \prod_i \mathfrak a_i^{\ceil{t_i (p^e-1)}}.
    \]
  \end{notation}
  \begin{notation}
    Work in \autoref{setting:tauI}. If $ \cartalg C = \cartalg C^{R, \prod_i \mathfrak a_i^{t_i}}$ for some ideals $\mathfrak a_i$ and non-negative numbers $t_i$, then we denote
    \[
      \tau_I\left( R, \prod_i \mathfrak a_i^{t_i} \right) \coloneqq \tau_I(R, \cartalg C)
    \]
  \end{notation}
  \begin{example} \label{ex:mainExampleForTau}
    Suppose $R_I$ is regular, and $\dim R_I = c$. Let $\{ \mathfrak b_i \}$ be a set of ideals, none of which is contained in $I$,  and let $\left\{ t_i \right\}$ be some collection of non-negative rational numbers. Then $I^c \prod_i \mathfrak b_i^{t_i}\cdot \cartalg C^{R}$ satisfies the conditions of \autoref{setting:tauI}. In this case, $\tau_I\left( R, I^c \prod_i \mathfrak b_i^{t_i} \right)$ is  what Takagi calls $\tilde\tau_I \left( R, \prod \mathfrak b_i^{t_i} \right) $. The Cartier algebras $\cartalg C^{R, I^{(c)} \prod_i b_i^{t_i}}$ and $\cartalg C^{R, \overline{I^{c}} \prod_i b_i^{t_i}}$ satisfy the conditions of  \autoref{setting:tauI} as well.  

    More generally, let $\left\{ \mathfrak a_i \right\}$ be any collection of ideals and $\left\{ t_i \right\}$ any collection of non-negative rational numbers, satisfying the following: for each $i$, suppose there exists some natural number $n_i$ with $\mathfrak a_i R_I = I^{n_i} R_I$, and suppose $\sum_i t_i n_i = c$. Further, for all $i$ such that $\mathfrak a_i \subseteq I$, we assume the denominator of $t_i$ is not divisible by $p$. Then $\cartalg C^{R, \prod_i \mathfrak a_i^{t_i}}$ satisfies the conditions of \autoref{setting:tauI}. To see this, we need to check that
    \begin{enumerate}
      \item $\cartalg C^{R,\prod_i \mathfrak a_i^{t_i}}$ is compatible with $I$, and 
      \item There exist $e>0$ and $\psi \in F^e_*\left( \tauOfa{e} \right) \cdot \homgp_R(F^e_* R, R)$ with $\psi(F^e_* R) \not \subseteq I$. 
    \end{enumerate}
    To show condition (a), let $x\in \tauOfa{e}$ and let  $\vp\in \homgp_R(F^e_* R, R)$. We wish to show that $\vp(F^e_* xI) \subseteq I$. Whether or not $\vp(F^e_* xI)$ is contained in  $I$ is not affected by localizing at $I$, 
    so we localize at $I$. Then $IR_I$ is generated by $c$ elements, since $R_I$ is regular,  and $x\in I^{cp^e-c}R_I$, since
    \[
      \tauOfa{e}R_I = \prod_i I^{n_i \ceil{t_i (p^e-1)}} \subseteq \prod_i I^{\ceil{n_i t_i (p^e-1)}} \subseteq I^{\ceil{\sum_i n_i t_i (p^e-1)}} = I^{cp^e-c}R_I.
    \]
    Thus $xIR_I \subseteq I^{cp^e-c+1}R_I \subseteq I^{[p^e]}R_I$. This shows that (a) is satisfied. To see condition (b) is satisfied, we notice again that this question can be checked locally at $I$. As $R_I$ is regular local and $F$-finite, it follows from \autoref{thm:kunz} that $F^e_* R_I$ is a free $R_I$ module for all $e$. Then it follows that $\vp(F^e_* x) \in I$ for all $\vp \in \homgp_{R_I}\left( F^e_* R_I, R_I \right)$ if and only if $x \in I^{[p^e]}R_I$. As the denominator of each $t_i$ is not divisible by $p$, we know there exists some $e$ so that $t_i(p^e-1)$ is an integer for all $i$. For this $q$, we have $\tauOfa{e} R_I = I^{cp^e-c} \not \subseteq  I^{[p^e]}$, so we're done. 
  \end{example}

  \subsection{Basic properties of \texorpdfstring{$\tau_I(R, \cartalg C)$}{tauI(R, C)}}
 Here we explore the basic theory of $\tau_I(R, \cartalg C)$. With the exception of the restriction theorem found in \autoref{prop:restriction}, the following are properties satisfied by all objects deserving of the name ``test ideal.'' The restriction theorem suggests that $\tau_I(R, \cartalg C)$ is a good candidate for a positive-characteristic analog to the adjoint ideals of birational geometry. 
\begin{lemma}
  Suppose that $\cartalg C \subseteq \cartalg D$ are two Cartier algebras on $R$ satisfying \autoref{setting:tauI}. Then $\tau_I(R, \cartalg C) \subseteq \tau_I(R, \cartalg D)$.
  \label{lemma:TauIsMonotonic}
\end{lemma}
\begin{proof}
  The ideal $\tau_I(R, \cartalg C)$ is the minimum element of the set
  \[
    S_{\cartalg C} \coloneqq \left\{ \mathfrak a \subseteq R \suchthat \mathfrak a \not \subseteq I \textrm{ and } \mathfrak a \textrm{ is } \cartalg C\textrm{-compatible}\right\}
  \]
  Similarly, $\tau_I(R, \cartalg D)$ is the minimum element of the set
  \[
    S_{\cartalg D} \coloneqq \left\{ \mathfrak a \subseteq R \suchthat \mathfrak a \not \subseteq I \textrm{ and } \mathfrak a \textrm{ is } \cartalg D\textrm{-compatible}\right\}
  \]
  We see that $S_{\cartalg C} \supseteq S_{\cartalg D}$, whence the minimum of $S_{\cartalg C}$ is smaller than the minimum of $S_{\cartalg D}$. 
\end{proof}
\begin{proposition}[Restriction theorem] Let $R$, $I$, and $\cartalg C$ be as in \autoref{setting:tauI}. Then $\tau_I(R, \cartalg C) R/I = \tau(R/I, \cartalg C|_{R/I})$.
  \label{prop:restriction}
\end{proposition}
\begin{proof}
  The proof is very similar to that of \autoref{prop:tauUnderRestriction}. Let $\vp \in \cartalg C_e|_{R/I}$. By definition, there exists some $\hat \vp \in \cartalg C_e$ such that the following diagram commutes:
  \[
    \xymatrix{
      F^e_*R  \ar[r]^{\widehat \vp} \ar[d]_{F^e_* \pi} &  R \ar[d]^{\pi} \\
      F^e_*(R/I) \ar[r]^{\vp} & R/I
    }
  \]
  We see that 
  \[
    \vp\left(F^e_* \tau_I(R, \cartalg C) R/I \right) = \hat \vp(F^e_* \tau_I(R, \cartalg C)) R/I \subseteq \tau_I(R, \cartalg C) R/I.
  \]
  By the minimality of $\tau(R/I, \cartalg C|_{R/I})$, it follows that $\tau_I(R, \cartalg C) R/I \supseteq \tau(R/I, \cartalg C|_{R/I})$.

  To get the reverse inclusion, it suffices to show that $\tau_I(R, \cartalg C) \subseteq \pi\invrs\left( \tau(R/I, \cartalg C|_{R/I}) \right)$. By definition, $\tau(R/I, \cartalg C|_{R/I}) \neq 0$, which means that $\pi\invrs \left( \tau(R/I, \cartalg C|_{R/I})  \right) \not \subseteq I$. Thus it suffices to check that $\pi\invrs \left( \tau(R/I, \cartalg C|_{R/I})  \right)$ is compatible with $\cartalg C$. To that end, let $\psi \in \cartalg C_e$ be arbitrary. As $\cartalg C$ is compatible with $I$, there exists some $\bar\psi \in \cartalg C|_{R/I}$ such that the following diagram commutes:
  \[
    \xymatrix{
      F^e_*R  \ar[r]^{\psi} \ar[d]_{F^e_* \pi} &  R \ar[d]^{\pi} \\
      F^e_*(R/I) \ar[r]^{\bar \psi} & R/I
    }
  \]
  It follows from the above diagram and the $\cartalg C|_{R/I}$-compatibility of $\tau(R/I, \cartalg C|_{R/I})$  that:
  \begin{align*}
    \pi\circ \psi\left( F^e_* \pi\invrs \left( \tau(R/I, \cartalg C|_{R/I})  \right)  \right) &= \bar \psi \circ F^e_* \pi\left(  F^e_* \pi\invrs \left( \tau(R/I, \cartalg C|_{R/I}) \right) \right) \\
    &= \bar \psi\left( F^e_* \tau(R/I, \cartalg C|_{R/I}) \right)  \\
    &\subseteq  \tau(R/I, \cartalg C|_{R/I}). 
  \end{align*}
  In other words, 
  \[
    \psi\left( F^e_* \pi\invrs \left( \tau(R/I, \cartalg C|_{R/I})  \right)  \right) \subseteq \pi\invrs \left( \tau(R/I, \cartalg C|_{R/I})  \right), 
  \]
  as desired. 
\end{proof}

 \begin{definition}
   We say that an element $b$ in $\tau_I(R, \cartalg C) \setminus I$ is a \emph{$\cartalg C$-test element along $I$.} If $\cartalg C = \cartalg C^{R, \prod_i \mathfrak a_i^{t_i}}$, then we say that $b$ is an $\prod_i \mathfrak a_i^{t_i}$-test element along $I$.  
 \end{definition}

\begin{lemma}
   Let $c$ be a $\cartalg C$-test element along $I$. Then $\tau_I(R, \cartalg C) = \sum_{e\geq 0} \sum_{\vp \in \cartalg C_e} \vp(F^e_* c)$. 
   \label{cor:tauIsSum}
 \end{lemma}
 \begin{proof}
   For ease of notation, set 
    \[
      \tau_I(R, \cartalg C; c) = \sum_{e \geq 0} \sum_{\vp \in \cartalg C_e} \vp(F^e_* c).
    \]
   As $c \not \in I$, have that $\tau_I(R, \cartalg C; c) \not \subseteq I$ by \autoref{prop:testElement}. Now let $J\not\subseteq I$ be an ideal compatible with $\cartalg C$. By definition of $\tau_I(R, \cartalg C)$, we know that $c\in J$. It follows that $\tau_I(R, \cartalg C; c) \subseteq J$, since  $J$ is compatible with $\cartalg C$. 
 \end{proof}

\begin{lemma}
  Let $c$ be a $\cartalg C$-test element along $I$ and let $e' \geq 0$. Then
  \[
    \tau_I(R, \cartalg C) = \sum_{e \geq e'} \sum_{\vp \in \cartalg C_e} \vp(F^e_* c).
  \]
  \label{lemma:tauIsAsymptotic}
\end{lemma}
\begin{proof}
  Set $J  = \sum_{e \geq e'} \sum_{\vp \in \cartalg C_e} \vp(F^e_* c)$. By \autoref{cor:tauIsSum}, we see that $J \subseteq \tau_I(R, \cartalg C)$.  Thus it suffices to show that $J \not \subseteq I$ and that $J$ is compatible with $\cartalg C$. The latter is obvious. To see that $J \not \subseteq I$, note that there exists some $e''$ and $\vp  \in \cartalg C_{e''}$ with $\vp(F^{e''}_* c) \not \in I$. This follows, for instance, from \autoref{cor:tauIsSum}. Then by \autoref{lemma:bSquared}, we have $\vp^n(F^{ne''}_* cR) \not \subseteq I$ for all $n$. In particular, we get the desired result by taking $n > e''/e'$. 
\end{proof}

\begin{lemma}
  Work in \autoref{setting:tauI} and let $W \subseteq R \setminus I$ be a multiplicative set. Then
  \[
    W\invrs \tau_I(R, \cartalg C) = \tau_{IW\invrs R}\left( W\invrs R, W\invrs \cartalg C \right).
  \]
\end{lemma}
\begin{proof}
  As $R$ is a domain and $W\cap I = \emptyset$, the ideal $\tau_{IW\invrs}(W\invrs R, W\invrs \cartalg C)$ is well-defined. As $R$ is $F$-finite and Noetherian, we have $W\invrs \homgp_R(F^e_* R, R) = \homgp_{W\invrs R}\left( F^e_* W\invrs R, W\invrs R \right)$ for all $e$. Then this lemma follows quickly from \autoref{thm:TauIsSum}. 
\end{proof}

  \subsection{Multiplier ideals along a closed subscheme}

  \begin{notation}
    Let $\pi: X' \to X$ be a birational map. By the \emph{exceptional locus} of $\pi$, or $\exclocus(\pi)$, we mean the locus of points in $X'$ at which $\pi$ is not an isomorphism.  \label{notation:exc}
  \end{notation}

     \begin{definition}
       Let $A$ be a $\bQ$-Gorenstein scheme of finite type over a field of characteristic $0$ and let $X$ be a reduced subscheme of $A$ of pure codimension $c$.  Suppose also that $A$ is smooth at the generic points of $X$. Let $Z$ be a formal $\bQ$-sum of subschemes of $A$ such that $Z$ equals $cX$ at the generic points of $X$. We define $\multideal_X(A, Z)$ as follows: let $\pi: \bar A \to A$ be a factorizing resolution\footnote{See \cite[Section 2]{EisensteinRestrictionThm} for the definition of a factorizing resolution. Note that, by definition, $\exclocus(\pi)$ is a divisor on $\overline A$. } of $X\subseteq A$ such that $\pi\invrs Z \cup \supp(\exclocus (\pi))$ is a simple normal crossings variety and the components of $Z$ not vanishing along $X$ lift to divisors. This is possible by \cite[Corollary 3.2]{EisensteinRestrictionThm}. Let $\bar X \subseteq \bar A$ be the strict transform of $X$ in $\bar A$ and let $\psi\colon A' \to \bar A$ be the blow up of $\bar A$ along $\bar X$.  We get the following diagram:
       \[
         \xymatrix@C-2pc{
           X'  \ar[d] &  \subseteq& A' \ar[d]^\psi\\
           \overline X \ar[d] & \subseteq &  \overline A \ar[d]^\pi\\
           X & \subseteq& A \ .
         }
       \]
       Then $X'\coloneqq \exclocus(\psi)$ is a prime divisor dominating $\bar X$. Let $f = \pi \circ \psi$.  We define \emph{the multiplier ideal of $(A, Z)$ along $X$} to be the ideal: 
    \[
      \multideal_X(A, Z) \coloneqq f_* \ssheaf_{A'}\left(\ceil*{K_{A'/A} - f\invrs Z} + X'\right).
    \]
    \label{def:adjointIdeal}
         \end{definition}

     The following lemma shows that $\multideal_X(A, Z)$ is a generalization of Takagi's adjoint ideal. 
    \begin{lemma}[{\cite[Proof of proposition 3.5]{EisensteinRestrictionThm}}]
      If $X$ is a subscheme of $A$ with pure codimension $c$, and $A$ is smooth at the generic points of $X$,  and none of the $Z_i$ vanish at the generic points of $X$, then, using Takagi's definition for $\adj_X(A, Z)$, 
      \[
        \adj_X(A,Z) =  \multideal_X(A, Z + cX)
      \]
    \label{lem:EisensteinAdj}
    \end{lemma}

\section{Comparison with Eisenstein's Subadditivity theorem}
Our main technical result in this section is that the adjoint ideal equals the test ideal along a subscheme mod $p \gg 0$, even when $A$ is singular.  This extends earlier results by Takagi in the setting where $X$ is a divisor \cite{TakagiCharPadjoint} and in the setting where $A$ is regular \cite{TakagiLCvsFP}. We will mainly work in the following setting. 
 \begin{setting}
   \label{setting:set2}
   $R$ is a $\bQ$-Gorenstein ring essentially of finite type over a perfect field $k$. % {\color{red}(is $F$-finite domain enough?)}.
   $I\subseteq R$ is a prime ideal of height $c$ and $R_I$ is regular. $\mathfrak a_i\subseteq R$ is a collection of ideals, $1 \leq i \leq N$,  and $t_i \geq 0$ a collection of rational numbers. We further assume that $\mathfrak a_i R_I = I^{n_i}R_I$ for each $i$ and also $\sum_i n_it_i = c$. Set $\cartalg C =\cartalg C^{R, \prod_i \mathfrak a_i^{t_i}}$. 

   Further, set $A = \spec R$, $X = \spec(R/I) \subseteq A$, and $Z_i = \spec(R/\mathfrak a_i) \subseteq A$. Let $Z$ denote the formal sum $Z = \sum_i t_i Z_i$.  If $R$ has characteristic 0, let $S\subseteq k$ be a descent datum\footnote{See \autoref{remark:modp}} and let $s \in \operatorname{MaxSpec} S$. Set $\kappa \coloneqq \kappa(s)$ and $p = \charp \kappa$. Then we write $R_\kappa$ to denote the mod-$p$ reduction of $R$ at $s$, and similarly for $I, \mathfrak a_i, A, X,$ and $Z_i$. We further denote $\sum_i t_i (Z_i)_\kappa$ by $Z_\kappa$ and we denote $\prod_i (\mathfrak a_i)_\kappa^{t_i}$ by $\idealsProd_\kappa$. 
 \end{setting}
 Note that this setting is agnostic to the characteristic of $R$. 

\begin{theorem}
  Work in \autoref{setting:set2} and assume that $\charp R = 0$. Then $\left( \multideal_{X}(A, Z) \right)_\kappa = \tau_{I_\kappa}(R_\kappa, \idealsProd_\kappa)$ for all $s \in \mspec S$ sufficiently general. 
  \label{thm:tauEqualsJ}
\end{theorem}
\begin{remark}[Reduction mod $p$] \label{remark:modp} Here we briefly review the process of reducing rings and schemes modulo $p$. See \cite[Chapter 2]{HHequalCharZero} for a detailed reference and \cite[Section 2]{TakagiLCvsFP} for a succinct reference.

  Let $R$ be an algebra essentially of finite type over a field $k$ of characteristic 0. We can find a finitely generated $\bZ$-subalgebra $B$ of $k$ and a $B$-subalgebra $R_B$ of $R$ such that the inclusion $R_B \subseteq R$ induces an isomorphism $R = R_B \otimes_B k$. This algebra $B$ is called \emph{descent datum} for $R$, and the algebra $R_B$ is called a \emph{model} for $R$ over $B$. For instance, if 
  \[
    R = \frac{\bC[x_1, \ldots, x_n]}{(f_1, \ldots, f_m)},
  \]
  we can choose $B$ to be the $\bZ$-subalgebra of $\bC$ generated by all of the coefficients of all of the polynomials $f_i$, and we can set
  \[
    R_B = \frac{B[x_1, \ldots, x_n]}{(f_1, \ldots, f_m)}.
  \] 
  For any maximal ideal $\mathfrak m \subseteq B$, the residue field $B/\mathfrak m$ will have positive characteristic $p$. Then $R_{\kappa(\mathfrak m)} \coloneqq R_B \otimes_B B/\mathfrak m$ is a \emph{mod $p$ reduction} of $R$. More specifically, we call $R_{\kappa(\mathfrak m)}$ the \emph{mod $p$ reduction of $R$ at $\mathfrak m$}. For any finitely generated  $R$-module $M$, we can find a model $M_B$ for $M$ over $B$ in the same manner. The choice of descent datum $B$ is not unique. However, given two descent data $B$ and $B'$, we can find a third descent datum $C$ containing $B$ and $B'$, such that $R_B \otimes_B C = R_{B'} \otimes_{B'} C$. Given any descent datum $B$, many properties of $R$ will be preserved by $R_{\kappa(\mathfrak m)}$ for $\mathfrak m$ sufficiently general, that is, for all $\mathfrak m$ in a dense open subset of $\operatorname{MaxSpec} B$. For instance, if $R$ is regular, then so is $R_{\kappa(\mathfrak m)}$ for all $\mathfrak m$ sufficiently general. 

  Note that, for all $N$, the subset of $\mathfrak m$ in $\operatorname{MaxSpec} B$  such that $\charp \kappa(\mathfrak m) > N$ forms a dense open set. Thus we have $\charp R_{\kappa(\mathfrak m)} \gg 0$ for $\mathfrak m$ sufficiently general.

  It will be useful for us later to note that reduction modulo $p$ commutes with tensor products, in the following sense. If $B$ is a descent datum for $R$, then $B$ is also a descent datum for $R\otimes_k R$, and we have $R_B \otimes_B R_B = (R \otimes_k R)_B$. Then we compute:
  \begin{align*}
    R_{\kappa(\mathfrak m)} \otimes_{\kappa(\mathfrak m)} R_{\kappa(\mathfrak  m)} &= R_B \otimes_B \kappa(\mathfrak m) \otimes_{\kappa(\mathfrak m)} R_B \otimes_B {\kappa(\mathfrak  m)}\\ & = R_B \otimes_B R_B \otimes_B {\kappa(\mathfrak  m)} = (R\otimes_k R)_{\kappa(\mathfrak m)}
  \end{align*}

  An important tool in the study of reduction modulo $p$ is the \emph{generic freeness lemma} \cite[(2.1.4)]{HHequalCharZero}. By this lemma, we can always enlarge our choice of descent datum $B$ to ensure that, for any finite collection $\left\{ M_i \right\}$ of finitely generated $R$-modules, the models $(M_i)_B$ will be free. 

  Given any map $\phi\colon M \to M'$ of finitely generated $R$-modules, we can find a suitable descent datum $B$ so that $\phi(M_B) \subseteq M_B'$. Then we say $\phi_B \coloneqq \phi\big |_{M_B}$ is a model for $\phi$, and we have $\phi_B \otimes_B k = \phi$. Given a bounded exact sequence of finitely generated $R$-modules, we can choose our descent datum $B$ so that the models of these maps over $B$ form an exact sequence of $B$-modules. 

  Similarly, given any scheme $X$ of finite type over $k$, we can find a descent datum $B \subseteq k$ and a $B$-scheme $X_B$ of finite type such that $X = X_B \times_{\spec B} \spec k$. We can perform a similar construction for coherent sheaves on $X$, morphisms $f\colon Y \to X$, and divisors on $X$. Given any closed point $\mu \in \spec B$, the residue field $\kappa(\mu)$ will have positive characteristic $p$. We call the fiber $X_{\kappa(\mu)} \coloneqq (X_B)_\mu$ a \emph{mod $p$ reduction} of $X$. Given a morphism of $k$-schemes $f\colon Y\to X$, we get an induced morphism $f_{\kappa(\mu)}\colon Y_{\kappa(\mu)} \to X_{\kappa(\mu)}$. If $f$ is projective, then so is $f_{\kappa(\mu)}$ for all $\mu$ sufficiently general. If $K_X$ is the canonical divisor of $X$, then $(K_X)_{\kappa(\mu)}$ is the canonical divisor of $X_{\kappa(\mu)}$ for all $\mu$ sufficiently general.
  \end{remark}

  \begin{remark} 
  Working in the setting of \autoref{def:adjointIdeal}, we can find a descent datum $B \subseteq k$ and reduce the maps $f\colon A' \xrightarrow{\psi} \overline A \xrightarrow{\pi} A$ modulo $p$. For $\mu \in \mspec B$ sufficiently general, $\pi_\mu \colon \overline A_\mu \to A_\mu$ will still be a factorizing resolution of $X_\mu \subseteq A_\mu$. In this way, we can define the multiplier ideal $\multideal_{X_\mu}(A_\mu, Z_\mu)$ for $\mu$ sufficiently general. By generic freeness, we can choose our descent datum so that $\ssheaf_{A'_B}(\ceil*{K_{A'_B/A_B} - f_B\invrs Z_B} + X_B')$, as well as all of its cohomology sheaves, are flat over $B$. It follows from \cite[Lemma 4.1]{HaraInjectivity} that the mod $p$ reduction of  $\multideal_{X}(A, Z)$ at $\mu$ equals $\multideal_{X_{\kappa(\mu)}}\left( A_{\kappa(\mu)}, Z_{\kappa(\mu)} \right)$ for $\mu$ sufficiently general. Thus, to prove \autoref{thm:tauEqualsJ},  it suffices to show that $\multideal_{X_{\kappa}}(A_{\kappa}, Z_{\kappa} ) = \tau_{I_{\kappa}}\left(R_{\kappa}, \idealsProd_\kappa \right)$ for all $s \in \mspec S$  sufficiently general.
\end{remark}

Before proving \autoref{thm:tauEqualsJ}, we show how this theorem allows us to compare our subadditivity formula with the one obtained in  \cite{EisensteinRestrictionThm}. Consider the following setting:
\begin{setting}
  \label{setting:eisenStuff}
  $R$ is a $\bQ$-Gorenstein ring essentially of finite type over a field $k$ of characteristic $0$ and $\mathfrak a, \mathfrak b \subseteq R$ are ideals. Set $X = \spec R$, $Z_1 = V(\mathfrak a)$, and $Z_2 = V(\mathfrak b)$.  Let $s, t \geq 0 $ be rational numbers. Let $S\subseteq k$ be a descent datum and $s \in \mspec S$. Set $\kappa = \kappa(s)$. 
\end{setting}

In \cite{EisensteinRestrictionThm}, Eisenstein derives a new formula for restriction multiplier ideals to closed subschemes. By carefully studying the case $\Delta \subseteq X \times_k X$, where $\Delta = X$ is the diagonal, he arrives at the containment:
  \begin{equation}
    \adj_\Delta\left( X \times_k X, s p_1^*Z_1 + t p_2^* Z_2 \right) \cdot \ssheaf_{\Delta} \supseteq \overline{\jac_X} \multideal\left( X, sZ_1 + tZ_2 \right)
    \label{eq:eisensteinDiag}
  \end{equation}
  where $p_i \colon X\times_k X \to X$ are the projection maps \cite[Proof of Theorem 6.5]{EisensteinRestrictionThm}. The left-hand side of \autoref{eq:eisensteinDiag} is easily seen to be contained in the product of multiplier ideals, $\multideal(X, Z_1) \multideal(X, Z_2)$ \cite[Lemma 6.2]{EisensteinRestrictionThm}. Now, 
  \[
    \adj_\Delta\left( X \times_k X, s p_1^*Z_1 + t p_2^* Z_2 \right) = \multideal_\Delta\left( X \times_k X, s p_1^*Z_1 + t p_2^* Z_2 + d \Delta \right)
  \]
  where $d = \dim X$. By \cite{HaraYoshidaSubadd}, we have
  \[
    \left( \overline{\jac_X} \multideal\left( X, sZ_1 + tZ_2 \right) \right)_\kappa = \overline{\jac(R_\kappa)} \tau\left( R_\kappa, (\mathfrak a_\kappa)^s (\mathfrak b_\kappa)^t  \right)
  \]
  for $s$ sufficiently general. So we see, combining \autoref{thm:tauEqualsJ} with \autoref{eq:eisensteinDiag}, that
  \[
    \tau_{(I_\Delta)_\kappa}\left( (R\otimes_k R)_\kappa, \left( \mathfrak a \otimes_k R \right)_\kappa^s \left( R \otimes_k \mathfrak b \right)_\kappa^t (I_\Delta)_\kappa^d \right) \cdot  (R\otimes_k R)_\kappa/(I_\Delta)_\kappa \supseteq \overline{\jac(R_\kappa)} \tau\left( R_\kappa, (\mathfrak a_\kappa)^s (\mathfrak b_\kappa)^t  \right)
  \]
  for all $s$ sufficiently general. As $\cartalg C^{ (R\otimes_k R)_\kappa, (I_\Delta)_\kappa^d}$ is compatible with $(I_\Delta)_\kappa$ (\cf \autoref{ex:mainExampleForTau}), it follows from \autoref{lemma:TauIsMonotonic} that 
  \begin{align*}
    &\tau_{ (I_\Delta)_\kappa }\left( (R\otimes_k R)_\kappa, \left( \mathfrak a \otimes_k R \right)_\kappa^s \cdot \left( R \otimes_k \mathfrak b \right)_\kappa^t \cdot (I_\Delta)_\kappa^d \right) \\
    \subseteq &   \tau_{ (I_\Delta)_\kappa }\left( (R\otimes_k R)_\kappa, \fullCA{R_\kappa \otimes_\kappa R_\kappa,\, (I_\Delta)\kappa \compSubAlg }, \left( \mathfrak a \otimes_k R \right)_\kappa^s \cdot \left( R \otimes_k \mathfrak b \right)_\kappa^t \right) 
 \end{align*}
  Here we're using the fact that $R_\kappa \otimes_\kappa  R_\kappa = (R \otimes_k R)_\kappa$. By \autoref{prop:restriction}, we have 
  \begin{align*}
     &\tau_{(I_\Delta)_\kappa}\left( (R\otimes_k R)_\kappa, \fullCA{R_\kappa\otimes_\kappa R_\kappa,\, (I_\Delta)\kappa \compSubAlg}, \left( \mathfrak a \otimes_k R \right)_\kappa^s \cdot \left( R \otimes_k \mathfrak b \right)_\kappa^t \right) \cdot (R \otimes_k R)_\kappa/(I_\Delta)_\kappa \\
      = &\tau(R_\kappa, \plainConR, \mathfrak a_\kappa^s \mathfrak b_\kappa^t).
  \end{align*}
%
%  \begin{align*}
%     \tau_{(I_\Delta)_p}\left( (R\otimes_k R)_p, \plainConR, \left( \mathfrak a \otimes_k R \right)_p^s \cdot \left( R \otimes_k \mathfrak b \right)_p^t \right) \cdot (R \otimes_k R)_p/(I_\Delta)_p &\subseteq \tau(R_p, \ConProd|_{R_p}, \mathfrak a_p^s \mathfrak b_p^t)\\
%     & = \tau(R_p, \plainConR, \mathfrak a_p^s \mathfrak b_p^t).
%  \end{align*}
   Thus we have shown, assuming \autoref{thm:tauEqualsJ}:
\begin{corollary}
  Work in \autoref{setting:eisenStuff}.  Then
\[
  \overline{\jac(R_\kappa)} \tau(R_\kappa, (\mathfrak a_\kappa)^s (\mathfrak b_\kappa)^t) \subseteq \tau(R_\kappa, \plainConR, (\mathfrak a_\kappa)^s (\mathfrak b_\kappa)^t)
\]
for all  $s$ sufficiently general. 
\label{cor:mainresult}
\end{corollary}
\noindent
Namely, our subadditivity formula, 
\[
  \tau(R_\kappa, \plainConR, (\mathfrak a_\kappa)^s (\mathfrak b_\kappa)^t) \subseteq \tau(R_\kappa, (\mathfrak a_\kappa)^s) \tau(R_\kappa,  (\mathfrak b_\kappa)^t)
\]
is a sharper containment than the previously known formula, 
\[
  \overline{\jac(R_\kappa)} \tau(R_\kappa, (\mathfrak a_\kappa)^s (\mathfrak b_\kappa)^t) \subseteq \tau(R_\kappa, (\mathfrak a_\kappa)^s) \tau(R_\kappa,  (\mathfrak b_\kappa)^t).
\]

\subsection{Proof that \texorpdfstring{$\left( \multideal_{X}(A, Z) \right)_\kappa \supseteq \tau_{I_\kappa}(R_\kappa, \idealsProd_\kappa)$}{Multiplier ideal contains test ideal}}
  We prove \autoref{thm:tauEqualsJ} in two parts. The first part is easier and just uses the minimality of $\tau_I$. The second part follows an argument similar to \cite[Theorem 5.3]{TakagiCharPadjoint} and requires a variant of Hara's surjectivity theorem \cite{HaraInjectivity}.

  Recall that for any normal $F$-finite scheme\footnote{Let $X$ be a scheme of positive characteristic. Then the Frobenius map on each affine chart of $X$ induces a morphism of schemes $F\colon X \to X$ called the \emph{absolute Frobenius morphism}. We say $X$ is $F$-finite if $F^e_* \ssheaf_X$ is a coherent $\ssheaf_X$-module for some (equivalently, all) $e > 0$.} $X$ of characteristic $p$ and any map $\vp\colon F^e_* \ssheaf_X \to \ssheaf_X$ we can associate an effective $\bQ$-divisor $\Delta_\vp$ on $X$ such that $K_X + \Delta_{\vp}$ is $\bQ$-Cartier with Cartier index not divisible by $p$ (\cf \cite{BlickleSchwedeSurvey}, \cite[Section 3]{SchwedeFadj}). If $h$ is a global section of $\ssheaf_X$ and we set $\vp_h = \vp(F^e_* h \cdot -)$, then $\Delta_{\vp_h} = \Delta_\vp + 1/(p^e-1) \divisor h$. 

  Now suppose that $X$ is an $F$-finite integral scheme with fraction field sheaf $\fracsheaf X$. As localization commutes with the $F^e_*$ functor, any map $\vp\colon F^e_* \ssheaf_X \to \ssheaf_X$ induces a map $\hat \vp\colon F^e_* \fracsheaf X \to \fracsheaf X$. 

  \begin{lemma}[\cf \cite{SchwedeCenters, HaraWatanabeFregvsLogT} ]
  Let $R$ be a normal $F$-finite domain and set $X = \spec R$. Suppose $\pi\colon Y \to X$ is a log-resolution of the ideals $\mathfrak a_i \subseteq R$ and set $\mathfrak a_i \ssheaf_Y = \ssheaf_Y(-G_i)$. Let $t_i > 0$ be a collection of rational numbers.  Then for any map $\vp\colon F^e_* \ssheaf_X  \to \ssheaf_X$ we have 
  \begin{align*}
    &\hat\vp\left( F^e_* \tauOfa{e} \ssheaf_Y\left( \ceil*{K_Y - \pi^* (K_X+\Delta) - \sum_i t_i G_i} + E \right) \right)\\
    &\subseteq \ssheaf_Y\left( \ceil*{K_Y - \pi^* (K_X+\Delta) - \sum_i t_i G_i} + E \right),
  \end{align*}
  where $\hat \vp$ is the induced map $F_*^e \fracsheaf X \to \fracsheaf X$, $E$ is any effective divisor on $Y$, and $\Delta$ is any divisor on $X$ such that $\Delta \leq \Delta_\vp$ and  $K_X + \Delta$ is $\bQ$-Cartier.
    \label{lemma:FcentersCompatibility}
  \end{lemma}
  \begin{proof}
    We work locally, so that $\pi_* K_Y = K_X$. %Fix a map $\vp: F^e_* \ssheaf_X \to \ssheaf_X$. Then $\vp$ induces a map $\hat \vp: F^e_* \fracsheaf X \to \fracsheaf X$.
  Let $h \in \tauOfa{e}$. Then $\Delta_{\vp_h} = \Delta_\vp  + 1/(p^e-1)\divisor h$. By the proof of \cite[Theorem 6.7]{SchwedeCenters}, we  have
    \begin{align*}
      & \hat \vp_h\left( F^e_* \ssheaf_Y\left( \ceil*{K_Y - \pi^*\left(K_X + \Delta_\vp+ \frac{1}{p^e-1} \divisor h \right)} + F \right) \right) \\
      & \subseteq  \ssheaf_Y\left( \ceil*{K_Y - \pi^* \left(K_X + \Delta_\vp+ \frac{1}{p^e-1} \divisor h \right)} + F \right) 
    \end{align*}
    for any integral effective divisor $F$. Set 
    \[
      F = \ceil*{K_Y - \pi^*(K_X + \Delta_\vp) - \sum t_i G_i} - \ceil*{K_Y - \pi^* \left(K_X + \Delta_\vp + \frac{1}{p^e-1} \divisor h \right)}
    \]
    Then $F$ is integral. Also $F$ is effective, as $\divisor h \geq \sum \ceil{t_i(p^e-1)} G_i$, which means $1/(p^e-1) \divisor h \geq \sum t_i G_i$. Thus we have
    \[
      \hat \vp_h\left( F^e_* \ssheaf_Y\left( \ceil*{K_Y - \pi^*(K_X + \Delta_\vp) - \sum t_i G_i }  \right) \right) \subseteq  \ssheaf_Y\left( \ceil*{K_Y - \pi^*(K_X + \Delta_\vp) - \sum t_i G_i }  \right) 
    \]
    Then for any effective divisor $E$ we have 
  \begin{align*}
     & \hat \vp_h\left( F^e_* \ssheaf_Y\left( \ceil*{K_Y - \pi^*(K_X + \Delta_\vp) - \sum t_i G_i }  +E\right) \right) \\
     &\subseteq  \ssheaf_Y\left( \ceil*{K_Y - \pi^*(K_X + \Delta_\vp) - \sum t_i G_i } +E \right) 
  \end{align*}
    using the projection formula, the fact that $(F^e)^*(\ssheaf_Y(E)) = \ssheaf_Y(p^eE)$, and the fact that $E \leq p^e E$. 

    Similarly, for any $\Delta \leq \Delta_\vp$ with $K_X + \Delta$ $\bQ$-Cartier, we  have
    \[
      \ceil*{K_Y - \pi^*(K_X + \Delta) - \sum t_i G_i } - \ceil*{K_Y - \pi^*(K_X + \Delta_\vp) - \sum t_i G_i } \geq 0, 
    \]
    and so 
    \begin{align*} 
      &\hat \vp_h\left( F^e_* \ssheaf_Y\left( \ceil*{K_Y - \pi^*(K_X + \Delta) - \sum t_i G_i }  +E'\right) \right) \\
      & \subseteq  \ssheaf_Y\left( \ceil*{K_Y - \pi^*(K_X + \Delta) - \sum t_i G_i } +E' \right),
    \end{align*}
    for any effective divisor $E'$, as desired. 
  \end{proof}

\begin{theorem}
  Work in \autoref{setting:set2} and assume that $\charp R = 0$. Then $\multideal_{X_\kappa}(A_\kappa, Z_\kappa) \supseteq \tau_{I_\kappa} (R_\kappa, \idealsProd_\kappa)$ for $s$ sufficiently general. 
  \label{thm:easyContainment}
\end{theorem}
\begin{proof}
   For  $s$ sufficiently general, we have that $p = \charp \kappa$ does not divide the denominator of any $t_i$ and thus $\tau_{I_\kappa}(R_\kappa, \idealsProd_\kappa )$ is well-defined. Fix such an $s$. We just need to prove two things:
  \begin{itemize}
    \item $\vp\left( F^e_* \tauOfa{e}_\kappa \multideal_{X_\kappa}(A_\kappa, Z_\kappa) \right) \subseteq \multideal_{X_\kappa}(A_\kappa, Z_\kappa)$,  for all $e>0$ and $\vp \in \fullCA{R}_e$, and
    \item $\multideal_{X_\kappa}(A_\kappa, Z_\kappa)\not \subseteq I_\kappa$. 
  \end{itemize}
  Set $\mathfrak a_i \ssheaf_{A'} = \ssheaf_{A'}(-F_i)$. Then, by definition, 
  \[
    \multideal_{X_\kappa}(A_\kappa, Z_\kappa) = (f_\kappa)_* \ssheaf_{A'_\kappa}\left( \ceil*{K_{A'_\kappa} - f_\kappa^* K_{A_\kappa}  - \sum_i t_i (F_i)_\kappa  } + X'_\kappa \right)
  \]
  We see that the first assertion follows from \autoref{lemma:FcentersCompatibility}, using $\Delta = 0$. 
  
  The second assertion is something we can check locally at $I_\kappa$, so now we assume that $R_\kappa$ is a local ring with maximal ideal $I_\kappa$. But then, by assumption, $\ssheaf_{A_\kappa'}\left( -\sum_i t_i (F_i)_\kappa \right) = \ssheaf_{A'_\kappa}\left( -cX'_\kappa \right)$. So we see
  \[
    \multideal_{X_\kappa}(A_\kappa, Z_\kappa) = (f_\kappa)_* \ssheaf_{A'_\kappa}\left( \ceil*{\psi^*K_{\bar A_\kappa} - f_\kappa^* K_{A_\kappa}}  \right)% = h_* \ssheaf_{\overline A}\left( K_{\overline A/A} \right)
  \]
  and $\psi^*K_{\bar A_\kappa} - f_\kappa^* K_{A_\kappa}$ has no support along $X_\kappa'$. 
\end{proof}

\subsection{Proof that \texorpdfstring{$\left( \multideal_{X}(A, Z) \right)_\kappa \subseteq \tau_{I_\kappa}(R_\kappa, \idealsProd_\kappa)$}{Multiplier ideal is contained in test ideal}}
 For the other containment, we use a similar argument to the one in \cite{TakagiCharPadjoint}. % with a couple of differences: we first take a factorizing-resolution of $X\subseteq A$, and we work in the dual setting (relative to Takagi's original proof).
 First, we recall Hara's surjectivity theorem. The following statement is slightly stronger than the one found in \cite{HaraInjectivity}: in \cite{HaraInjectivity}, the author assumes that $X$ is the blow up of an ideal sheaf on $Y$. However, the same proof actually shows the following statement, where $X$ is just assumed to be projective over $Y$ and smooth. 

\begin{theorem}[{\cite[Section 4.3]{HaraInjectivity}}]
  Let $Y = \spec R$, where $R$ is finitely generated over a field $k$ of characteristic $0$, and let $X$ be a smooth Noetherian scheme projective over $Y$. Suppose $E$ is a reduced simple normal crossings divisor on $X$ and $D$ an ample divisor with $\supp (D - \floor{D}) \subseteq \supp E$. Choose some finitely generated $\bZ$-subalgebra $B$ of $k$, over which we do our reduction mod $p$. For any closed point $s\in S = \spec B$ with residue field $\kappa = \kappa(s)$, let $Y_\kappa$, $X_\kappa$, $E_\kappa$, and $D_\kappa$ be the fibers of the corresponding objects over $s$. Then, for sufficiently general closed points $s$, 
  \[
    \begin{array}[]{lll}
      (a) & H^j ( X_\kappa, \Omega_{X_\kappa/\kappa}^i (\log E_\kappa) (- E_\kappa - \floor{-p^eD_\kappa}))=0,& i+j >d, e \geq 0\\
      (b) & H^j ( X_\kappa, \Omega_{X_\kappa/\kappa}^i (\log E_\kappa) (- E_\kappa - \floor{-p^{e+1}D_\kappa}))=0,& i+j >0, e \geq 0
    \end{array}
  \]
  where $d = \dim X$ and  $p = \charp \kappa(s)$.
  \label{thm:haraVanishing}
\end{theorem}
  Combined with \cite[Proposition 3.6]{HaraInjectivity} (as stated), we obtain the following result: 
  \begin{corollary}[{\cf \cite[Section 4.4]{HaraInjectivity}}]
  Using notation as in \autoref{thm:haraVanishing}, the map
  \[
    (F^e)^\vee\colon H^0(X_{\kappa}, F^e_* \omega_{X_\kappa}(\ceil{p^e D_\kappa})) \to H^0(X_{\kappa}, \omega_{X_\kappa}( \ceil{D_\kappa} ))
  \]
  is surjective for $e > 0$ and for sufficiently general $s\in S$, where $\kappa = \kappa(s)$. 
  \label{cor:harasurj}
\end{corollary}

We will also need the following lemmas: 
\begin{lemma}
   Work in \autoref{setting:set2} and assume that $R$ has characteristic 0. There exists $d\in R\setminus I$ such that, for all sufficiently general $s$, a power of $d_\kappa$ (depending on $s$) is an $\idealsProd_\kappa$-test element along $I_\kappa$ in $R_\kappa$.
   \label{lemma:testElementModp}
 \end{lemma}
 \begin{proof}
   If $s$ is sufficiently general, then $p = \charp \kappa(s)$ will not divide the denominator of any $t_i$ and $\cartalg C_\kappa \coloneqq \cartalg C^{R_\kappa, \idealsProd_\kappa}$ will satisfy the conditions of \autoref{setting:tauI}. As $R_I$ is regular, we can find a regular sequence $(x_1, \ldots, x_c)$ in $R$ such that $(x_1, \ldots, x_c)R_I = IR_I$. Set $J = (x_1, \ldots x_c)$.  Then there exists an element $d \in R\setminus I$ such that $R_d$ is regular, $R_d/IR_d$ is regular, and $J^{n_i}R_d \subseteq \mathfrak a_iR_d$ for all $i$. Note that $(R_d)_\kappa$ is the same as $R_{\kappa}$ localized at $d_{\kappa}$. We set $R_{d, \kappa} \coloneqq (R_\kappa)_{d_\kappa}$ and $\cartalg C_{d, \kappa} \coloneqq (\cartalg C_{\kappa})_{d_\kappa}$. Then we have
   \[
     \tau_{I_\kappa}(R_\kappa, \cartalg C_\kappa) R_{d, \kappa} = \tau_{I_\kappa R_{d, \kappa}}(R_{d,\kappa}, \cartalg C_{d, \kappa})  \supseteq \tau_{J_{\kappa}R_{d, \kappa}}(R_{d, \kappa}, J_{\kappa}^c R_{d, \kappa}), 
   \]
   where the containment follows quickly from the minimality of $\tau_{J_{\kappa}R_{d, \kappa}}(R_{d,\kappa}, J_{\kappa}^c R_{d,\kappa})$. By \autoref{lemma:testIdealOfRegSeq}, we see that $\tau_{I_{\kappa}}(R_{\kappa}, \cartalg C_{\kappa})R_{d,\kappa}  = R_{d, \kappa}$. Thus, there exists some $N$ such that $d_{\kappa}^N \in \tau_{I_{\kappa}}(R_{\kappa}, \cartalg C_{\kappa})$. As $d_{\kappa}^N \in R_{\kappa}\setminus I_{\kappa}$, we're done. 
 \end{proof}

 \begin{lemma}
   Let $k$ be a perfect field of characteristic $p$. Let $R$ be a regular $k$-algebra essentially of finite type and $I$ a prime ideal generated by a regular sequence of length $c$. Suppose also that $R/I$ is regular. Then $\tau_I(R, I^c) = R$. 
   \label{lemma:testIdealOfRegSeq}
 \end{lemma}
 \begin{proof}
   This fact is well-known to experts, and is essentially shown in \cite[Theorem 3.2]{TakagiLCvsFP}. 
 \end{proof}

 \begin{lemma}
   Let $I \subseteq R$ be a prime ideal such that $R_I$ is regular. Let $\mathfrak a \subseteq R$ be an ideal such that $\mathfrak a R_I = I^m R_I$ for some $m\geq 0$.  Then there exists some $\xi \in R \setminus I$ such that $\xi \overline{\mathfrak a^{n}} \subseteq \mathfrak a^{n}$ for all integers $n$. 
   \label{lemma:multClosureIntoIdeal} 
 \end{lemma}
 \begin{proof}
   By \cite[Proposition 5.3.4]{HunekeSwansonIntegral}, there is some integer $k$, such that $\overline{\mathfrak a^n} = \mathfrak a^{n-k} \overline{\mathfrak a^{k}}$ for all $n \geq k$. As $R_I$ is regular, we have $IR_I$ is generated by a regular sequence and is therefore a \emph{normal} ideal, meaning $I^nR_I$ is integrally closed for all $n$ (see, for instance, \cite[Exercise 5.7]{HunekeSwansonIntegral}). As integral closure of ideals commutes with localization, we have
   \[
     \overline{\mathfrak a^n}R_I = \overline{\mathfrak a^nR_I} = \overline{I^{nm}R_I} = I^{nm}R_I = \mathfrak a^n R_I,
   \]
   for all $n$. Thus, for $n=1, \ldots k$ there exist elements $\xi_n \in R\setminus I$ satisfying $x_n \overline{\mathfrak a^n} \subseteq \mathfrak a^n$. Then we can set $\xi = \xi_1 \cdots \xi_k$. 
 \end{proof}
\begin{lemma}
    Work in \autoref{setting:set2} and assume that $\charp R  = 0$. There exists some $\xi \in R\setminus I$ such that $\xi \overline{\tauOfaq} \subseteq \tauOfaq$ for all $p \gg 0$ and all $e > 0$ sufficiently divisible, where $q = p^e$. 
    \label{lemma:integralClosureLemmaForJInTau}
  \end{lemma}
  \begin{proof}
    For each $i$, write $t_i = a_i/b_i$. Set $m$ to be the least common multiple of the $b_i$, so that for each $i$ there exists an integer $a'_i$ such that $t_i = a'_i/m$. For $p$ sufficiently large, we have $p$ does not divide $b_i$ for each $i$. Then for $e$ sufficiently divisible, we have $m \mid (p^e-1)$. Thus
    \[
      \prod_i \mathfrak a_i^{\ceil{t_i(p^e-1)}} = \prod_i \mathfrak a_i^{a'_i(p^e-1)/m} = \left( \prod_i \mathfrak a_i^{a_i'} \right)^{\frac{p^e-1}{m}}.
    \]
    Then we can find the desired element $\xi$ by applying \autoref{lemma:multClosureIntoIdeal} to the ideal
   $\displaystyle 
      \mathfrak a = \prod_i \mathfrak a_i^{a_i'}.
   $ 
    %This completes the proof. 
  \end{proof}
  %
% \begin{lemma}
%   Non-ambiguity of $\tau_I(R, \prod_i \mathfrak a_i^{t_i})$.
%   \label{lemma:nonAmbig}
% \end{lemma}
%

  \begin{theorem}
    Work in \autoref{setting:set2} and assume that $\charp R = 0$. Then $\multideal_{X_\kappa}(A_\kappa, Z_\kappa) \subseteq \tau_{I_\kappa}(R_\kappa, \underline{\mathfrak a}_\kappa^{t})$ for all $s$ sufficiently general. 
    \label{thm:hardContainment}
  \end{theorem}
  \begin{proof}
    We start by finding a suitable resolultion of $A$. In particular, we need a resolution $f\colon B \to A$ that can be used to compute $\multideal_X(A, Z)$ such that there exists a  divisor $V$ on $B$ satisfying the following:  
    \begin{itemize}
      \item $f\invrs Z \cup \supp( \exclocus(f)) \cup \supp(f^* K_A)\cup \supp V$ is a simple normal crossings scheme, \footnote{Recall from \autoref{notation:exc} that by $\exclocus(f)$ we mean the locus of points at which $f$ is not an isomorphism. }
      \item $V$ does not dominate $X$, and
      \item $-\varepsilon V - F$ is relatively ample over $A$ for any $\varepsilon > 0$ sufficiently small, where $\mathfrak a_i \ssheaf_{B} = \ssheaf_{B}(-F_i)$ and $F = \sum_i t_i F_i$. 
    \end{itemize}

    Let $  \pi_1 \colon   A_1 \to A$ be a factorizing resolution of $X \subseteq A$ such that $  \pi_1 \invrs Z \cup \supp( \exclocus(  \pi_1)) \cup \supp(  \pi_1^* K_A)$ is a simple normal crossings scheme. Let $  X_1 \subseteq   A_1$ be the strict transform of $X$ in $  A_1$ and let $  \pi_2 \colon   A_2 \to   A_1$ be the blow up along $  X_1$. Let $X_2$ be the reduced $\pi_2$-exceptional divisor dominating $X_1$. We have the following diagram: 
     \[
      \xymatrix@C-2pc{
          X_2  \ar[d] &  \subseteq&   A_2 \ar[d]^{  \pi_2}\\
          X_1 \ar[d] & \subseteq &    A_1 \ar[d]^{  \pi_1}\\
      X & \subseteq& A\ .
      }
    \]
    Let $\mathfrak a_i \ssheaf_{  A_2} = \ssheaf_{  A_2}(-  F'_i)$ and set $  F' = \sum t_i   F'_i$. Setting $\pi = \pi_1 \circ \pi_2$, we get
    \[
    \multideal_X(A, Z) =   \pi_* \ssheaf_{  A_2}\left( K_{  A_2} - \floor{  \pi^* K_A +   F'}  +   X_2 \right),
    \] 
    by definition. However, we are not finished constructing the resolution we need, as we don't know whether a divisor $V$ satisfying the properties above exists on $  A_2$. 
    
    Note that $-  F'$ is relatively big and relatively semi-ample over $A$. The latter implies that there exist natural numbers $a, N > 0$ and a map $\gamma' \colon   X_2 \to \bP_A^N$ such that $-a   F' = (\gamma')^* \ssheaf_{\bP_A^N}(1)$. As $-   F'$ is relatively big, $\gamma'$ is birational onto its image, which we call $B'$. Note that $\gamma'$ is an embedding at the generic point of $  X_2$ and therefor an isomorphism over an open neighborhood of $X$ in $A$. Thus the exceptional locus of $\gamma'$ does not contain $  X_2$, and we can perform a sequence of blowups of schemes not containing $  X_2$, call it $\gamma \colon B \to   A_2$, to get that $\exclocus(\gamma'\circ \gamma) \cup \gamma\invrs   \pi \invrs Z \cup \operatorname{supp}(\gamma^*   \pi^* K_A)$ has simple normal crossings support. Let $\gamma'' = \gamma' \circ \gamma''$ and $f =   \pi \circ \gamma$. These definitions are illustrated by the following diagram:  
\[
\begin{tikzcd}
 & B \arrow[ld, "\gamma"'] \arrow[rd, "\gamma''"] \arrow[ddd, "f", bend left=20] &  \\
  A_2 \arrow[d, "  \pi_2"'] \arrow[rdd, "  \pi"] \arrow[rr, "\gamma'"] &  & B' \arrow[ldd] \\
  A_1 \arrow[rd, " \pi_1"'] &  &  \\
 & A & 
\end{tikzcd}
\]
Note that the exceptional locus of $\gamma''$ does not dominate $\gamma'(  X_2)$. Thus we can write $\gamma'$ as the blow up of some ideal sheaf $\sheaf I \subseteq \ssheaf_{B'}$ that does not vanish along $\gamma'(  X_2)$. Let $W$ be the effective divisor on $B$ such that $\sheaf I \ssheaf_B = \ssheaf_B(-W)$. Then we see that  $\ssheaf_{B}(-W) \otimes (\gamma'')^*\ssheaf_{B'}(b)$ is very ample over $A$ for all $b$ sufficiently large. 
Set $V = \frac{1}{ab}W$. Then $-V - F$ is ample over $A$. Because $W$ is effective, we may assume that
    \[
      \floor{f^* K_A + F + V} = \floor{f^* K_A + F},
    \]
    by choosing $b$ sufficiently large. Fix canonical divisors $K_{  A_1}, K_{A_2}$, and $K_B$ so that  $K_{A_1/A}$ is $\pi_1$-exceptional, $K_{A_2/A_1} = (c-1) X_2$, and $K_{B/A_2}$ is $\gamma$-exceptional. Let $\tilde X = \gamma\invrs_*  X_2$. Then
    \[
      \gamma_*\ssheaf_{B}\left( K_{B}+\tilde X - \floor{f^* K_A + F} \right) = \ssheaf_{  A_2}\left( K_{  A_2} +   X_2  - \floor{  \pi^* K_A +   F'} \right),
    \]
    just because adjoint ideals are  well-defined. Indeed, it suffices to check the above equality after twisting each side by $\ssheaf_{A_2}(-\pi^*_2 K_{A_1})$. As $K_{A_2} = (c-1)X_2 + \pi_2^* K_{A_1}$, we get 

    \begin{align*}
      & \gamma_* \ssheaf_{B}\left( K_{B}+\tilde X - \floor{f^* K_A + F} \right)\otimes \ssheaf_{A_2}\left( -\pi_2^* K_{A_1} \right)\\
      =&  \gamma_* \ssheaf_{B}\left( K_{B/A_2} + \gamma^* K_{A_2}- \gamma^* \pi_2^* K_{A_1}+\tilde X - \floor{f^* K_A + F} \right)\\
      =& \gamma_* \ssheaf_{B}\left( K_{B/A_2} + (c-1)\gamma^* X_2 + \tilde X - \floor{f^* K_A + F} \right).
    \end{align*}
    By our assumptions in \autoref{setting:set2}, we have $F' = cX_2 + G$ for some effective SNC $\bQ$-divisor $G$ not supported along $X_2$. Thus: 
    \begin{align*}
      & \gamma_* \ssheaf_{B}\left( K_{B/A_2} + (c-1)\gamma^* X_2 + \tilde X - \floor{f^* K_A + F} \right)\\
      =& \gamma_* \ssheaf_{B}\left( K_{B/A_2} + (c-1)\gamma^* X_2 + \tilde X - \floor{f^* K_A + \gamma^*(cX_2 + G)} \right)\\
      =& \gamma_* \ssheaf_{B}\left( K_{B/A_2} + \tilde X - \floor{f^* K_A + \gamma^*(X_2 + G)} \right).
    \end{align*}
    On the other hand, we have
    \begin{align*}
      & \ssheaf_{  A_2}\left( K_{  A_2} +   X_2  - \floor{  \pi^* K_A +   F'} \right) \otimes \ssheaf_{A_2}\left( -\pi_2^* K_{A_1} \right)\\
      =& \ssheaf_{  A_2}\left( (c-1)X_2 + X_2  - \floor{  \pi^* K_A +  c X_2 + G} \right)\\
      =& \ssheaf_{  A_2}\left( X_2  - \floor{  \pi^* K_A +  X_2 + G} \right).
    \end{align*}
    By definition, we have
    \begin{align*}
      \adj_X(A, \pi_2^* K_A + G) &= \gamma_* \ssheaf_{B}\left( K_{B/A_2} + \tilde X - \floor{f^* K_A + \gamma^*(X_2 + G)} \right)\\
      &= \ssheaf_{  A_2}\left( X_2  - \floor{  \pi^* K_A +  X_2 + G} \right),
    \end{align*}
    because  both $\gamma$ and the identity map are log-resolutions of $\pi_2^* K_A + G$. 
    In particular, we get $\multideal_X(A, Z) = f_* \ssheaf_{B}\left( K_{B}+\tilde X - \floor{f^* K_A + F} \right)$.

    Employing \autoref{lemma:testElementModp},  choose some $d\in R\setminus I$ such that, for $s$ sufficiently general, there is some $N$ such that $d_\kappa^N$ is an $\idealsProd_\kappa$-test element along $I_\kappa$ in $R_\kappa$. Choose also an element $\xi \in R\setminus I$ as in \autoref{lemma:integralClosureLemmaForJInTau}.

      Next, we show there exists $\eta \in R\setminus I$ satisfying
    \[
      K_{B/A_2} + \gamma^*\pi_2^*K_{  A_1} - \floor{m   f^* K_A} - \divisor_{B} \eta \leq f^*\left( (1-m) K_A\right)  + \gamma^* X_2 - \tilde X
    \]
    for all $m$ such that $(1-m)K_A$ is Cartier. As $f^*\left( (1-m)K_A \right)$ is integral in this case, it suffices to find $\eta\in R\setminus I$ satisfying
    \[
      K_{B/A_2} + \gamma^*\pi_2^*K_{  A_1} -  \divisor_{B} \eta \leq \floor{f^* K_A}  + \gamma^* X_2 - \tilde X.
    \]
    We compute:
    \begin{equation*}
      \floor{f^* K_A}  + \gamma^* X_2 - \tilde X - K_{B/A_2} - \gamma^*\pi_2^* K_{  A_1} =  \gamma^* X_2 - \tilde X - K_{B/A_2} - \ceil*{ \gamma^* \pi_2^* K_{  A_1}- f^* K_A}.
    \end{equation*}
Recall that the exceptional locus of $\gamma$ does not contain $\tilde X$, so $ \gamma^* X_2 - \tilde X - K_{B/A_2}$ is not supported on $\tilde X$.
    Since the support of  $K_{  A_1} - \pi_1^* K_A$ does not contain $  X_1$, the support of 
    \[
        \gamma^*\pi_2^* K_{  A_1} - f^* K_A = \gamma^*\pi_2^*\left( K_{  A_1} - \pi_1^* K_A \right) 
    \]
    does not contain $\tilde X$. Further, this divisor is  $f$-exceptional. Thus we have 
    \[
      H^0\left( A, f_* \ssheaf_{B}\big(\floor{f^* K_A} + \gamma^* X_2 - \tilde X - K_{B/A_2} -  \gamma^*\pi_2^* K_{A_1} \big) \right) \subseteq R
    \]
    and also
    \[
      H^0\left( A, f_* \ssheaf_{B}\big(\floor{f^* K_A} + \gamma^* X_2 - \tilde X - K_{B/A_2} -  \gamma^*\pi_2^* K_{A_1} \big) \right)  \not \subseteq I,
    \]
    so we can find the desired element $\eta$ by taking
    \[
      \eta \in H^0\left( A, f_* \ssheaf_{B}\big(\floor{f^* K_A} + \gamma^* X_2 - \tilde X - K_{B/A_2} -  \gamma^*\pi_2^* K_{A_1} \big) \right)  \setminus I.
    \]
    The utility of this $\eta$ will be made apparent in \autoref{appendix:multIdeal}. 
    
    Next, we define 
    \[
      D = f^* K_{A} + F + V + \varepsilon \divisor_{B}(d\xi\eta),
    \]
    where $\varepsilon>0$ is chosen to be small enough such that
    \[
      \floor{D} = \floor{f^* K_A + F}.
    \]
    Note that $D$ is $f$-anti-ample, as $F + V$ is $f$-anti-ample, $f^*K_A$ is $f$-numerically trivial, and $\divisor_{B}(d\xi \eta)$ is $f$-anti-nef. For all $m\in \bN_{>0}$,  we have a short exact sequence of sheaves on $A'$, 
    \[
      0 \to \ssheaf_{B}\left(K_{B} - \floor{mD} \right) \to \ssheaf_{B}(K_{B} - \floor{mD} + \tilde X) \to \sheaf A_m \to 0.
    \]
    for some $\sheaf A_m$ supported on $\tilde X$. 

    Then we get an exact sequence
    \begin{align*}
      0 &\to H^0(B, \ssheaf_{B}(K_{B} - \floor{mD})) \to H^0(B, \ssheaf_{B}(K_{B} - \floor{mD} + \tilde X)) \\\
      &\to H^0(\tilde X, \sheaf A_m) \to H^1(B, \ssheaf_{B}(K_{B} - \floor{mD} )) ,
    \end{align*}
    Now, $-mD$ is an $f$-big and $f$-nef divisor whose fractional part has simple normal crossings support. By Kawamata-Viehweg vanishing \cite[Corollary 9.1.20]{PositivityII}, we have 
  \[
    H^1(B, \ssheaf_{B}(K_{B} - \floor{mD})) = H^1(B, \ssheaf_{B}(K_{B} + \ceil{-mD})) = 0.
  \]
  Thus, we have the short exact sequence, 
  \begin{align*}
    0 &\to H^0(B, \ssheaf_{B}(K_{B} - \floor{mD}))   \to H^0(B, \ssheaf_{B}(K_{B} - \floor{mD} + \tilde X))\\
    & \to H^0( \tilde X, \sheaf A_m) \to 0.
  \end{align*}
  Next, we compute $\sheaf A_m$. By our assumptions in \autoref{setting:set2}, we can express  $F$ as $F = c \tilde X + G$, where $G$ is a $\bQ$-divisor whose support does not contain $\tilde X$, and $c$ is the codimension of $X$ in $A$.  Thus, we get 
  \begin{align*}
    & \ssheaf_B\left( K_B - \floor{mf^* K_A + mF + mV + m\varepsilon \divisor_B(d \xi \eta)} + \tilde X \right) \\
    = & \ssheaf_B\left( K_B + \tilde X - \floor{mf^* K_A + mG + mV + m \varepsilon \divisor_B(d \xi \eta)} \right) \otimes \ssheaf_B(- mc \tilde X)
  \end{align*}
  Now, $\ssheaf_B(c \tilde X) \otimes \ssheaf_{\tilde X}$ is a line bundle on $\tilde X$, so we can fix some integral divisor $\Sigma$ on $\tilde X$ so that $\ssheaf_{\tilde X}(\Sigma) \cong \ssheaf_B( c \tilde X ) \otimes \ssheaf_{\tilde X}$. As     $(K_B + \tilde X)|_{\tilde X} \sim K_{\tilde X}$, it follows that
  \begin{equation*}
    \sheaf A_m \cong \ssheaf_{\tilde X}\left( K_{\tilde X} - \floor{m \tilde D} \right)
  \end{equation*}
where we set $\tilde D \coloneq \left( f^* K_A + G + V + \varepsilon \divisor_B(d \xi \eta) \right)|_{\tilde X} + \Sigma$. 
Note that restricting to $\tilde X$ here commutes with rounding down because $D$ has simple normal crossings support and $D$ is supported on $\tilde X$. Also, note that $\tilde D$ is anti-ample over $A$. As $\Sigma$ is integral, the fractional part of $\tilde D$ has simple normal crossings support. 

    Set $p = \charp \kappa(s)$. For $s$ sufficiently general, we have:
    \begin{itemize}
      \item The map
        \[
        (F^e_{B_\kappa})^\vee \colon H^0\left(B_{\kappa}, F^e_* \omega_{B_\kappa}\left(- \floor{p^e D_\kappa }\right)\right) \to H^0\left(B_\kappa, \omega_{B_\kappa}\left( -\floor{ D_\kappa } \right) \right)
        \]
        is a surjection for all $e > 0$.  This is possible by \autoref{cor:harasurj}. 
      \item The map 
        \[
          (F^e_{\tilde X_\kappa})^\vee \colon H^0\left(\tilde X_{\kappa}, F^e_* \omega_{\tilde X_\kappa}\left(- \floor{p^e  \tilde D_{\kappa}}\right)\right) \to H^0\left(\tilde X_\kappa, \omega_{\tilde X_\kappa}\left( - \floor{  \tilde D_{\kappa}}\right)\right)
        \]
        is a surjection for all $e$. This is also possible by \autoref{cor:harasurj}. Note that $\tilde X$ is projective over $X$, so the original statement of Hara's surjectivity theorem would not apply to $\tilde X \to X$. 
      \item $p$ does not divide the Cartier index of $K_A$.
      \item $p$ does not divide the denominator of any $t_i$.
      \item Some power of $d$ is an $\idealsProd_\kappa$-test element along $I_\kappa$ in $R_\kappa$.
      \item $H^0(B, \omega_{B}(-\floor{mD}))_\kappa = H^0(B_\kappa, \omega_{B_\kappa}(-\floor{mD_\kappa}))$ for all $m \in \bN_{> 0}$. This is possible because $D$ is a $\bQ$-divisor. Indeed, let $v$ be the least common multiple of the denominators of the coefficients appearing in $D$. Then
        \[
          \left\{\floor{mD} \suchthat m \geq 0 \right\} = \left\{ uv D + \floor{mD} \suchthat 0 \leq m < v, u \geq 0 \right\}.
        \]
        By generic freeness, we can choose our descent datum $S$ to ensure that the coherent sheaves
          $\omega_{B}(-\floor{mD})$ and $\ssheaf_{B}(-vD)$, as well as their cohomologies,
        are flat for $0 \leq m < v$. Then the result follows from \cite[Lemma 4.1]{HaraInjectivity}. 
      \item Similarly, we can ensure that $H^0(B, \omega_{B}(\tilde X - \floor{ mD }))_\kappa =  H^0(B_\kappa, \omega_{B_\kappa}(\tilde X_\kappa - \floor{ mD_\kappa }))$ for all $m \in \bN_{> 0}$,  and 
      \item $H^0(\tilde X, \omega_{\tilde X}(-\floor{ m \tilde D }))_\kappa = H^0(\tilde X_\kappa, \omega_{\tilde X_\kappa}(-\floor{ m  \tilde D _{\kappa} })) $ for all $m \in \bN_{> 0}$.
    \end{itemize}
    Fix such an $s$. Fix also a number $N$ so that $d_\kappa^N$ is an $\idealsProd_\kappa$-test element along $I_\kappa$ in $R_\kappa$. Then for all $e \in \bN$ sufficiently divisible we have:
    \begin{itemize}
      \item $\xi \overline{\tauOfa{e}} \subseteq \tauOfa{e}$,
      \item $p^e \varepsilon > N$, 
      \item $(1-p^e)K_A$ is Cartier, and
      \item $t_i(p^e-1) \in \bZ$ for all $i$. 
    \end{itemize}
    Fix such an $e$. With that taken care of, reduce the whole setup modulo $p$ at $\kappa$ and set $q = p^e$. We get the following diagram:
    \[
      \begin{tikzpicture}
        \matrix(m)[matrix of math nodes, row sep=2em, column sep=0.8em]{
          0 & H^0(B_\kappa, F_*^e \omega_{B_\kappa}(-\floor{ qD_\kappa })) & H^0(B_\kappa, F_*^e \omega_{B_\kappa}(\tilde X_\kappa - \floor{ qD_\kappa })) & H^0(\tilde X_\kappa, F_*^e \omega_{\tilde X_\kappa}(-\floor{ q\tilde D_\kappa })) & 0 \\
          0 & H^0(B_\kappa, \omega_{B_\kappa}(-\floor{ D_\kappa })) & H^0(B_\kappa, \omega_{B_\kappa}(\tilde X_\kappa - \floor{ D_\kappa })) & H^0(\tilde X_\kappa, \omega_{\tilde X_\kappa}(-\floor{ \tilde D_\kappa })) & 0 \\
        };
        \path[-stealth]
        (m-1-1) edge node{} (m-1-2)
        (m-1-2) edge node{} (m-1-3)
        edge node[right] {$(F^e_{B_\kappa})^{\vee}$} (m-2-2)
        (m-1-3) edge node{} (m-1-4)
        edge node[right]{$(F^e_{B_\kappa})^{\vee}$} (m-2-3)
        (m-1-4) edge node{} (m-1-5)
        edge node[right]{$(F^e_{\tilde X_\kappa})^{\vee}$} (m-2-4)
        (m-2-1) edge node{} (m-2-2)
        (m-2-2) edge node{} (m-2-3)
        (m-2-3) edge node{} (m-2-4)
        (m-2-4) edge node{} (m-2-5);
      \end{tikzpicture}
    \]
    By the five lemma, as well as our assumptions on $p$, we see that 
    \[
      (F^e_{B_\kappa})^{\vee}\colon H^0(B_\kappa, F_*^e  \omega_{B_\kappa}(\tilde X_\kappa - \floor{qD_\kappa})) \to H^0(B_\kappa, \omega_{B_\kappa}(\tilde X_\kappa - \floor{D_\kappa}))
    \]
    is a surjection. But the right-hand side is exactly (the global sections of) $\multideal_{X_\kappa}(A_\kappa, Z_\kappa)$ by \autoref{lem:EisensteinAdj}. Thus it's enough to show that the image of this map is contained in $\tau_{I_\kappa}(R_\kappa, \idealsProd_\kappa)$. It follows from a straightforward computation that
    \[
      H^0(B_\kappa, F_*^e  \omega_{B_\kappa}(\tilde X_\kappa - \floor{qD_\kappa})) \subseteq  R_\kappa\left( (1-q) K_{A_\kappa} \right)\overline{\prod_i\mathfrak (a_i)_\kappa^{ t_i(q-1) }}  d_\kappa^N \xi_\kappa^N.
    \]
    See \autoref{appendix:multIdeal} for details. By the construction of $\xi$, we have $\xi_\kappa^N \overline{\prod_i\mathfrak (a_i)_\kappa^{ t_i(q-1) }} \subseteq \prod_i\mathfrak (a_i)_\kappa^{ t_i(q-1) }$. Thus we're done, by \autoref{cor:tauIsSum}. Indeed, the map $(F^e_{A'_\kappa})^{\vee}$ may be identified with the ``evaluate at 1'' map $\homgp_{R_\kappa}(F^e_* R_\kappa, R_\kappa) \to R_\kappa$ via the isomorphism
    \[
      \homgp_{R_\kappa}(F^e_* R_\kappa, R_\kappa) \cong F^e_* R_\kappa \left( (1-q)K_{A_\kappa} \right),
    \]
    see \cite{SchwedeFadj}.  So we have shown that
    \[
      \multideal_{X_\kappa}(A_\kappa, Z_\kappa) \subseteq \vp\left( F^e_* \tauOfa{e}_\kappa d_\kappa^N \right)
    \]
    for some $\vp \in \fullCA{R_\kappa}_e$. 
  \end{proof}

It's worth noting the  following analog to \autoref{lemma:wellKnown}(b). This proposition follows from \autoref{lemma:multClosureIntoIdeal}, and it provides further evidence that $\tau_I(R, \cartalg C)$ is well-behaved in \autoref{setting:set2}. 
 \begin{proposition}
   Work in \autoref{setting:set2} and assume that $\charp R = p$. For each $i$, let $\mathfrak b_i = \overline{\mathfrak a_i}$ be the integral closure of $\mathfrak a_i$.  Then
   \[
     \tau_I\left(R, \prod_i \mathfrak a_i^{t_i} \right) = \tau_I \left(R, \prod_i \mathfrak b_i^{t_i} \right).
   \]
 \end{proposition}
 \begin{proof}
   The $\subseteq$ inclusion follows from \autoref{lemma:TauIsMonotonic}, so we prove the reverse inclusion. 
   As $\mathfrak b_i^n \subseteq \overline{\mathfrak a_i^n}$ for all $i$ and $n$, it follows from \autoref{lemma:multClosureIntoIdeal} that there exists some $\xi \in R \setminus I$ such that
   \[
     \xi \prod_i \mathfrak b_i^{t_i(p^e-1)} \subseteq \prod_i \mathfrak b_i^{t_i(p^e-1)} .
   \]
   Let $d\in \tau_I(R, \prod_i \mathfrak a_i) \setminus I$ be arbitrary.  Then \autoref{cor:tauIsSum} tells us:
   \begin{align*}
     \tau_I\left(R, \prod_i \mathfrak b_i\right) &= \sum_{e \geq 0} \sum_{\vp \in \fullCA{R}} \vp\left( F^e_* d\xi \prod_i \mathfrak b_i^{t_i(p^e-1)} \right) \\
     &\subseteq  \sum_{e \geq 0} \sum_{\vp \in \fullCA{R}} \vp\left( F^e_* d \prod_i \mathfrak a_i^{t_i(p^e-1)} \right) = \tau_I\left(R, \prod_i \mathfrak b_i \right).
   \end{align*}
 \end{proof}

 \section{Computing \texorpdfstring{$\ConR$}{the diagonal cartier algebra} for Affine Toric Varieties}
Our next goal is to find a nice description of $\plainConR$ that allows us to compute test ideals $\tau(R, \plainConR, \prod_i \mathfrak a_i^{t_i})$. The case where $R$ is a normal affine semigroup ring over a field $k$ (equivalently, $\spec R$ is an affine toric variety over $k$) turns out to be quite tractable.
\begin{setting}
  We let $\Sigma$ be a fan in $\bR^n$ and $X = X(\Sigma)$ the associated toric variety over a perfect field $k$ of characteristic $p > 0$. By $\Sigma(1)$ we mean the set of rays (i.e. 1-dimensional cones) in $\Sigma$. We assume, for simplicity, that $\Sigma$ has a unique $n$-dimensional cone $\sigma$.   We let $R$ be the coordinate ring of $X$.  In particular, $R$ is a \emph{toric ring}, that is, a normal affine semigroup ring. For all rays $\rho \in \Sigma(1)$, we let  $v_\rho$ denote the primitive generator of $\rho$. That is, $v_{\rho}$ is the shortest nonzero vector in $\bZ^n \cap \rho$.
 \label{setting:toric}
\end{setting}
Choose $e > 0$ and set $q = p^e$. In this section, it will be more convenient to use the notation $R^{1/q}$ rather than $F^e_* R$. Working in \autoref{setting:toric}, we have a nice $k$-basis for $\homgp_R(R^{1/q}, R)$. 
In particular, we let 
\[
  k[T] = k\left[ x_1^{\pm 1}, \ldots, x_n^{\pm 1} \right]
\]
be the coordinate ring of the $n$-dimensional torus $T$. Similarly, we let
\[
  k[T\times T] = k[T]\otimes_k k[T]
\]
be the coordinate ring of the $2n$-dimensional torus, $T\times T$. We let $q = p^e$, and we adopt the notation, 
\[
  \frac{1}{q} \bZ^n \coloneqq \left\{ \left( \frac{a_1}{q}, \cdots, \frac{a_n}{q} \right) \suchthat a_1, \cdots, a_n \in \bZ \right\}.
\]
For any vector $u \in \frac{1}{q} \bZ^n$, we adopt the shorthand notation
\[
  x^u \coloneqq \prod_{i=1}^n x_i^{u_i}.
\]
Then for all $a\in \frac{1}{q}\bZ^n$ we define a map $\pi_a$ in terms of its action on monomials:  for all $u\in \frac{1}{q} \bZ^n$, set
\[
  \pi_a(x^u) = \begin{cases}
    x^{a+u}, & a+u\in \bZ^n\\
    0, & \textrm{otherwise}
  \end{cases}
\]
It's not hard to see that these maps $\pi_a$ generate all the maps $k[T]^{1/q} \to k[T]$: 
\begin{lemma}[\cite{PayneToric}]
  In the notation above, the set $\{\pi_a \suchthat a \in \frac{1}{q} \bZ^n \}$ is a  $k$-vector space basis of $\homgp_{k[T]}\left( k[T]^{1/q}, k[T] \right)$.
\end{lemma}
Further, every map $R^{1/q} \to R$ extends to a unique map $k[T]^{1/q} \to k[T]$. Thus, each map in $\homgp(R^{1/q}, R)$ is just a map $k[T]^{1/q} \to k[T]$ that happens to send $R^{1/q}$ into $R$. Payne has characterized such maps:  we define the \emph{anticanonical polytope} of $R$, 
\[
  \acptope = \left\{ u \in \bR^n \suchthat \braket{u, v_\rho}  \geq -1 \textrm{ for all rays } \rho \in \Sigma(1) \right\}
\]
 Then we have: 

\begin{proposition}[\cite{PayneToric}]
  Work in \autoref{setting:toric}. Then the set of maps $\pi_a$ where $a$ is in $\operatorname{int}(P_{-K_R})\cap \frac{1}{q} \bZ^n$ forms a $k$-vector space basis for $\homgp_R(R^{1/q}, R)$. 
  \label{prop:PayneExtensionCriterion}
\end{proposition}

%Payne used this to prove the main theorem in  his paper: 
%\begin{theorem}[\cite{PayneToric}]
%  Let $X = X(\Sigma)$ be a toric variety. Then $X$ is diagonally split if and only if the interior of $P_{-K}\cap (-P_{-K})$ contains a representative of each equivalence class in $\frac{1}{q}M/M$. 
%  \label{payneDiagSplit}
%\end{theorem}

Our characterization of $\plainConR$ is as follows:

\begin{theorem}
  Work in \autoref{setting:toric}. Then $\plainConR_e(R)$ is generated as a $k$-vector space by the maps $\pi_a$ where $a\in \frac{1}{p^e}\bZ^n \cap \operatorname{int}(\acptope)$ and the interior of $\acptope \cap\left( a - \acptope \right)$ contains a representative of each equivalence class in $\frac{1}{p^e}\bZ^n / \bZ^n$. 
  \label{prop:resCriterion}
\end{theorem}
First, we must prove a lemma. This lemma is similar to \cite[Theorem 7.3]{PayneUnimodular}. Note however that we do not assume that $\vp$ is a splitting. 

\begin{lemma}
  Let $\vp = \sum c_{a, a'} \pi_{a}\otimes \pi_{a'}$ be a map in $\homgp_{k[T \times T]}\left( k[T\times T]^{1/q} , k[T \times T] \right)$. Then $\vp$ is compatible with $I_\Delta$ if and only if for all equivalence classes $[u_1], [u_2] \in  \frac{1}{q}\bZ^n/\bZ^n$, we have, for all $d\in \frac{1}{q}\bZ^n$, 
  \[
    \sum_{a \in [u_1]} c_{a, d-a} = \sum_{b\in [u_2]} c_{b, d-b}
  \]
  \label{lemma:keylemma}
\end{lemma}
\begin{proof}
  Note that the ideal $I_{\Delta}^{1/q} \subseteq k[T\times T]^{1/q}$ is generated by the elements 
  \[
    \left\{ x^{u}\otimes x^{-u} - 1 \suchthat u \in \frac{1}{q} \bZ^n \right\}.
  \]
  Since $k[T \times T]$ is a smaller ring than $k[T\times T]^{1/q}$, we need more elements to generate $I_{\Delta}^{1/q}$ as a  $k[T\times T]$-module. However, elements of the form $x^v \otimes x^{v'}$, where $v$ and $v'$ are vectors in $\frac{1}{q}\bZ^n$, generate $k[T\times T]^{1/q}$ as a $k[T\times T]$-module (indeed, as a $k$-vector space). Thus, the set 
  \[
    \left\{ x^{v}\otimes x^{v'} \left( x^u \otimes x^{-u} - 1 \right) \suchthat u, v, v' \in \frac{1}{q}\bZ^n\right\}
  \]
  generates $I_{\Delta}^{1/q}$ as a module over $k[T\times T]$. 

Suppose $\varphi = \sum_{a, a'} c_{a, a'} \pi_{a}\otimes \pi_{ a'}$ is compatible with the diagonal. This is equivalent to asserting that 
  \begin{equation}
    \varphi \left( x^{v}\otimes x^{v'} \left( x^u \otimes x^{-u} - 1 \right)\right) \equiv 0 \textrm{ mod }I_\Delta
    \label{eq:keylemmaeq1}
  \end{equation}
  for all $u, v, v' \in \frac{1}{q}\bZ^n$. Set $\varphi_{v, v'}\coloneqq  \varphi(x^{v}\otimes x^{v'} \cdot \underscore )$. Then the condition in \autoref{eq:keylemmaeq1} is equivalent to saying
  \begin{equation*}
    \vp_{v, v'} \left(  x^u \otimes x^{-u} - 1 \right) \equiv 0 \textrm{ mod } I_\Delta, 
    %\label{eq:keylemmaeq2}
  \end{equation*}
  for all $u, v, v' \in \frac{1}{q}\bZ^n$, or in other words, 
  \begin{equation}
     \vp_{v, v'} \left(  x^u \otimes x^{-u} \right) \equiv \vp_{v, v'}(1) \textrm{ mod } I_\Delta.
    \label{eq:keylemmaeq2}
  \end{equation}
  Now, it's easy to see that $\pi_{a}\otimes \pi_{a'}\left( x^v \otimes x^{v'} \cdot \underscore \right) = \pi_{a + v}\otimes \pi_{a'+v'}$. Thus
  \begin{equation*}
    \vp_{v, v'} = \sum_{a, a' \in \frac{1}{q}\bZ^n} c_{a, a'} \pi_{a+v}\otimes \pi_{a'+v'} = \sum_{a, a' \in \frac{1}{q}\bZ^n} c_{a - v, a'-v'} \pi_{a}\otimes \pi_{a'}
  \end{equation*}
  This means that \autoref{eq:keylemmaeq2} is equivalent to saying
  \[
    \sum_{\substack{a \in -u + \bZ^n,\\ a' \in u + \bZ^n}} c_{a - v, a' - v'} x^{a+u}\otimes x^{a'-u} \equiv \sum_{b, b' \in \bZ^n} c_{b - v, b'-v'} x^b \otimes x^{b'} \textrm{ mod } I_{\Delta}, 
  \]
  and this  is the case if and only if
  \begin{equation}
    \sum_{\substack{a\in -u + \bZ^n,\\ a' \in u + \bZ^n}} c_{a - v, a' - v'} x^{a+a'} = \sum_{b, b'\in \bZ^n} c_{b - v, b' - v'} x^{b + b'}.
    \label{eq:something}
  \end{equation}
  Now, the above is an equality of Laurent polynomials, so it holds if and only if the corresponding coefficients for each exponent of $x$ are the same. So our initial assertion \autoref{eq:keylemmaeq1} holds if and only if, for all $d \in \bZ^n$ and all $u, v, v'\in \frac{1}{q}\bZ^n$, we have
  \[
    \sum_{a \in -u + \bZ^n} c_{a - v, d-a - v'} = \sum_{b\in \bZ^n} c_{b - v, d-b - v'}
  \]
  (In other words, we're setting $d = a + a' = b +b'$ in  \autoref{eq:something}.) By setting $U = -u - v$ and $D = d-v'-v$,  the above is equivalent to 
  \[
    \sum_{a \in U + \bZ^n} c_{a, D - a} = \sum_{b \in -v + \bZ^n} c_{b, D - b }
  \]
  where $U, v,$ and $D$ independently range over $\frac{1}{q}\bZ^n$. This completes the proof.
\end{proof}

\begin{corollary}
  Let $R$ be a toric ring. The Cartier algebra $\ConProd$ is ``graded'', in the sense that the map $\sum_{a, a'}c_{a, a'} \pi_a \otimes \pi_{a'}$ is compatible with the diagonal if and only if, for each $d \in \frac{1}{q}\bZ^n$, we have $\sum_{a + a' = d} c_{a, a'} \pi_a \otimes \pi_{a'}$ is compatible with the diagonal. It follows that $\ConR$ is generated over $k$ by the maps $\pi_d$ in $\ConR$.
  \label{cor:gradedCartAlg}
\end{corollary}
\begin{proof}
  First, we focus on the case that $R = k[T]$. The first part follows immediately from the above lemma: a map $\sum_{a, a'}c_{a, a'} \pi_a \otimes \pi_{a'}$ satisfies the condition in \autoref{lemma:keylemma} if and only if the maps $\sum_{a + a' = d} c_{a, a'} \pi_a \otimes \pi_{a'}$ satisfy the condition in \autoref{lemma:keylemma} for all $d$. 

  For the second assertion, let $\psi \in \ConProd$ be arbitrary. Then $\overline \psi\coloneqq \psi\big |_{R\otimes_k R/I_\Delta}$ is an arbitrary element of $\ConR$. If 
  \[
    \psi =  \sum_{a, a'} b_{a, a'} \pi_a \otimes \pi_{a'}, 
  \]
  then by the first assertion, we have $\psi'\coloneqq \sum_{a + a' = u} b_{a, a'} \pi_a \otimes \pi_{a'}$ is also in $\ConProd$ for all $u \in \frac{1}{q} \bZ^n$. Now let $v \in \frac{1}{q}\bZ^n$ be arbitrary. We compute:
  \begin{align}
    \psi'|_\Delta(x^v) &= \left.\left( \sum_{a}b_{a, u-a} \pi_a(x^v) \otimes \pi_{u-a}(1) \right)\right |_{R\otimes_k R/I_\Delta}\\
    &= \left.\left( \sum_{\substack{a\in -v + \bZ^n\\ \textrm{where } u-a \in \bZ^n}} b_{a, u-a} x^{a+v}\otimes x^{u-a}\right)\right|_{R\otimes_k R/I_\Delta}\\
    &= \begin{cases}
      \left( \sum_{a \in -v + \bZ^n} b_{a, u-a}\right) x^{u+v}, & u+ v \in \bZ^n\\
      0 & \textrm{ otherwise}
    \end{cases}\\
    &= \left( \sum_{a \in -v + \bZ^n} b_{a, u-a} \right)\pi_u(x^v) \label{eq:restrictionForumla}
  \end{align}
  So if we write $\overline \psi = \sum_{a} c_a \pi_a$, we see that $\psi'|_{R\otimes_k R/I_\Delta} = c_u \pi_u$. Thus either $c_u = 0$ or $\pi_u\in \ConR$, as desired.

  For the general case, we just note that each map $(R \otimes_k R)^{1/q} \to R\otimes_k R$ extends to a unique map $k[T \times T]^{1/q} \to k[T \times T]$, and one of these maps is compatible with the diagonal whenever the other map is (\cf \cite[Lemma 1.1.7]{BKFsplitting}).% Then we're done, as $\psi = \sum_{a, a'} c_{a, a'} \pi_{a, a'}$ is regular on $R\otimes R$ if and only if $\pi_{a, a'}$ is for all pairs $a, a'$ appearing in the summation (\cite[Proposition 4.3]{PayneToric}). 
  \end{proof}
  
Note that, a priori, it looks like the coefficient of $\pi_u(x^v)$ in \autoref{eq:restrictionForumla} depends on $v$ (or even worse, on our choice of lifting of $x^v$ to $R \otimes_k R$), but by \autoref{lemma:keylemma}, this is not the case. %Indeed, one can probably use this calculation to prove that ``only if'' part of that lemma. 

\begin{proof}[Proof of \autoref{prop:resCriterion}]
  Any map in $\psi \in \homgp_{R\otimes_k R}( (R\otimes_k R)^{1/q}, R\otimes_k R)$ extends to a map $\hat \psi \in \homgp_{k[T\times T]}( k[T\times T]^{1/q}, k[T\times T])$, and two maps 
  agree if and only if their extensions to the torus agree. By \cite[Lemma 1.1.7]{BKFsplitting}, $\psi$ is compatible with the diagonal
  if and only if $\hat \psi$ is compatible with the diagonal. 
  %Note that a map on $X \times X$ compatible with the diagonal is precisely a map on the dense torus $T\times T$ compatible with the diagonal that happens to extend to a map on $X \times X$ (c.f. \cite[Lemma 1.1.7]{BKFsplitting}). % Just need to show: restriction to torus determines the map
  This means the map $\pi_d$ is in $\ConR$ if and only if there exists some map 
  \[
    \vp = \sum_{a, a'} c_{a, a'} \pi_a \otimes \pi_{a'} \colon k[T\times T]^{1/q} \to k[T \times T]
  \]
  compatible with $I_\Delta$ that restricts to a map $ (R\otimes_k R)^{1/q} \to R \otimes_k R$, and such that the sum $\sum_{a+a'=d} c_{a, a'}$ is non-zero. By \autoref{lemma:keylemma}, this means  for all $[u_1], [u_2]\in \frac{1}{q}\bZ^n/ \bZ^n$, we have 
  \[
    \sum_{a \in [u_1]} c_{a, d - a} = \sum_{b \in [u_2]} c_{b, d-b}.
  \]
  Further, if $\sum_{a+a' = d}c_{a, a'} \neq 0$, then there exists some $[u] \in \frac{1}{q}\bZ^n/\bZ^n$ such that $\sum_{a \in [u]} c_{a, d-a}\neq 0$. This is just because
  \[
    \sum_{a + a' = d} c_{a, a'} = \sum_{[u] \in \frac{1}{q} \bZ^n/ \bZ^n} \sum_{a \in [u]} c_{a, d-a}
  \]
  Using \autoref{lemma:keylemma} again, this means that for \emph{all} $[u]\in \frac{1}{q}\bZ^n/\bZ^n$, the sum $\sum_{a\in [u]}c_{a, d-a}$ is nonzero. In particular, for all $[u] \in \frac{1}{q}\bZ^n/\bZ^n$, there is some $a \in [u]$ such that $c_{a, d-a}\neq 0$. Since $\vp$ restricts to a map $F_*^e (R\otimes_k R) \to R\otimes_k R$, this means  $a, d-a \in \operatorname{int}(\acptope)$, by \autoref{prop:PayneExtensionCriterion}. In other words, $a$ is in the interior of $\acptope \cap \left( d - \acptope \right)$. 

  Conversely, given some $d$, suppose that each equivalence class $[u]$ has a representative in the interior of $\acptope \cap \left( d - \acptope \right)$. Then we can label these representatives $a_1, \ldots, a_N$. Then $\sum_i \pi_{a_i}\otimes \pi_{d - a_i}$ is a map compatible with the diagonal, and its restriction to the diagonal is $\pi_d$. This is just an application of equation \autoref{eq:restrictionForumla}; in this case, there is only one nonzero coefficient $b_{a, d-a}$ where $a$ is in any particular equivalence class of $\frac{1}{q}\bZ^n/\bZ^n$.
\end{proof}

\begin{remark}
  \autoref{prop:resCriterion} can be seen as a generalization of  \cite[Theorem 1.2]{PayneToric}. Indeed, the following lemma shows that an affine toric variety is diagonally $F$-split if and only if $\pi_0 \in \plainConR$.  %{\color{red} (I think this is also implicit in Payne's work).} 
  Thus we recover \cite[Theorem 1.2]{PayneToric} by setting $a=0$ in the statement of \autoref{prop:resCriterion}.
\end{remark}
\begin{lemma}
  Let $R$ be a toric ring. The following are equivalent:
  \begin{enumerate}[(i)]
    \item $R$ is diagonally $F$-split \label{dfs1}
    \item $\pi_0 \in \plainConR_e(R)$ for some $e > 0$ \label{dfs2}
    \item $\pi_0 \in \plainConR_1(R)$.  \label{dfs3}
  \end{enumerate}
  \label{lemma:diagFSplit}
\end{lemma}
\begin{proof}
  To see that \eqref{dfs1} implies \eqref{dfs2}, suppose that $R$ is diagonally $F$-split. By \cite[Proposition 4.5]{PayneToric}, there exists some $e>0$ and some map 
  \[
    \vp = \sum_{a \in \frac{1}{q}\bZ^n} c_a \pi_a \in \plainConR_e.
  \]
   with $c_0 \neq 0$. By \autoref{cor:gradedCartAlg},  we have $\pi_0 \in \plainConR_e(R)$. 
   
   To see that \eqref{dfs2} implies \eqref{dfs3}, suppose that $\pi_0 \in \plainConR_e(R)$ and set $q = p^e$. As $R^{1/p} \subseteq R^{1/q}$, the map $\pi_0$ restricts to a map $R^{1/p} \to R$. One checks that this restriction is in $\plainConR_1(R)$, for instance by using \autoref{prop:resCriterion}.  
   
   Finally, \eqref{dfs3} implies \eqref{dfs1} by definition. 
\end{proof}

\begin{example}
  Consider the case $R = k[x,y,z]/(xy-z^2)$, and assume that $\charp k > 2$. To use the techniques in this section, we use the presentation $R=  k[y, xy, xy^2]$.  Then \autoref{fig:quadric_cone_lattice} shows the polytope $\acptope$. Using \autoref{prop:resCriterion}, one can compute $\ConR$:
  \begin{equation}
    \ConR = \bigoplus_{e\geq 0} F^e_* \braket{x^{p^e+1}y^{\frac{p^e+1}{2}}, x^{p^e}y^{\frac{p^e+1}{2}}, xy^{\frac{p^e+1}{2}}, y^{\frac{p^e+1}{2}},  x^{p^e-1}y^{p^e-1} } \homgp_R(F^e_*R, R)
    \label{eq:CAlgGens}
  \end{equation}
  In terms of the more familiar presentation, $R \cong k[x^2, xy, y^2]$, this formula becomes
  \[
    \ConR = \bigoplus_{e \geq 0} F^e_* \braket{x^{p^e+1}, x^{p^e}y, x^{p^e-1} y^{p^e-1}, xy^{p^e}, y^{p^e+1}}\homgp_R\left( F^e_* R, R \right)
  \]
  \begin{figure}[h]
  \centering
  \includegraphics[width=0.5\textwidth]{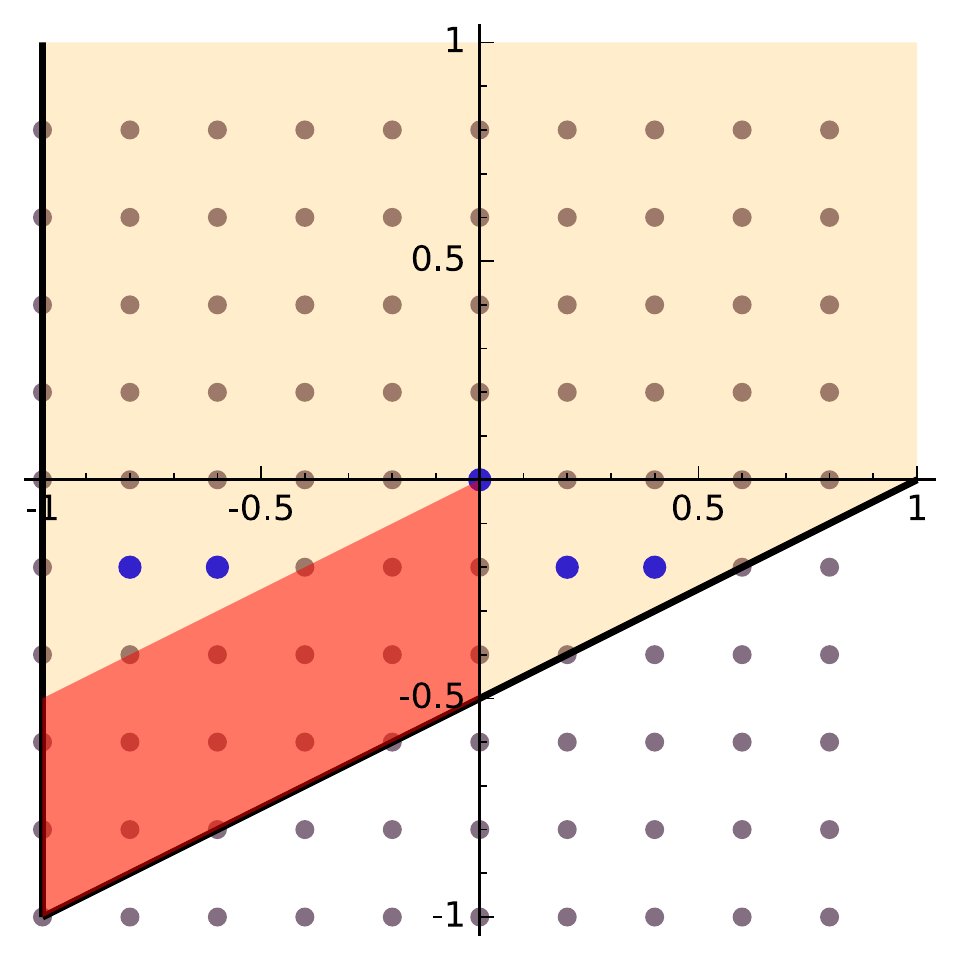}
  \caption{The polytope $\acptope$ for the quadric cone $R = k[y, xy, x^2 y]$, along with the fractional lattice $\frac{1}{5} \bZ^2$. The area in red denotes the set of maps not in $\plainConR_1(R)$. The points in blue denote the generators of $\plainConR_1(R)$ over $F^1_* R$. }
  \label{fig:quadric_cone_lattice}
\end{figure}
To see this, we will prove the following: 
\begin{enumerate}[(i)]
  \item If $a > -1$ and $b > a/2$, then $\pi_{(a,b)} \in \ConR$
  \item If $a > 0$ and $b > (a-1)/2$, then $\pi_{(a,b)} \in \ConR$
  \item $\pi_{(0, 0)} \in \ConR$
  \item The maps $\pi_{\left( 0, -1/q \right)}, \pi_{\left(-1/q, -1/q\right)}, $ and $\pi_{(-2/q, -2/q)}$ are not in $\ConR$. 
\end{enumerate}
Because $\ConR$ is a Cartier algebra and $\pi_a(x^{b} \cdot -) = \pi_{a+b}$ for all $a, b \in \frac{1}{q} \bZ^n$, it follows from (iv) that the maps described in (i)--(iii) are the only maps in $\ConR$. (Another way to see this is to notice that, for any map $\pi_v$ not among those described in (i)--(iii), the corresponding polytope $\acptope \cap(v - \acptope)$ is contained in the polytope corresponding to one of the maps described in (iv)). Consequently we see that $\plainConR_e$ is generated over $F^e_* R$ by the maps $\pi_v$, where
\[
  v \in \left\{\left(\frac{1 - q}{q}, \frac{(q+1)/2}{q}  \right), \left(\frac{2 - q}{q}, \frac{(q+1)/2}{q}  \right), \left( 0,0 \right), \left(\frac{1 }{q}, \frac{(q+1)/2}{q}  \right), \left(\frac{2}{q}, \frac{(q+1)/2}{q}  \right)  \right\}
\]
As $\homgp_R(F^e_* R, R)$ is generated as an $F^e_*R$-module by $\pi_{\left(\frac{1-q}{q}, \frac{1-q}{q}  \right)}$, we get that $\ConR$ has the description given in equation \autoref{eq:CAlgGens}.

So, let $(a, b), (\alpha, \beta) \in \bR^2$. Then 
\begin{align*}
  P_{(a,b)}&\coloneqq \operatorname{int}\left( \acptope \cap ( (a,b) - \acptope) \right)\\
  &= \left\{ (x,y) \suchthat -1 < x < a+1, -1 < 2y - x < 1 - a + 2b \right\}
\end{align*}
We wish to find an integer translation of $(\alpha, \beta)$ in $P_{(a,b)}$, where $(a,b)$ is as in (i) or (ii). Let $\overline \alpha= \alpha - \ceil{\alpha}$. If $(a,b)$ is as in (i), then we have 
\[
  \begin{cases}
    \left( \overline \alpha, \beta - \floor{\beta} \right) \in P_{(a,b)}, & 2(\beta - \floor{\beta}) - \overline \alpha \leq 1\\
    \left( \overline \alpha, \beta - \floor{\beta} - 1 \right) \in P_{(a,b)}, & 2(\beta - \floor{\beta}) - \overline \alpha > 1\\
  \end{cases}
\]
If $(a, b)$ is as in (ii), then we have
\[
  \begin{cases}
    \left( \overline \alpha +1, \beta - \floor \beta - 1 \right) \in P_{(a,b)}, & \frac{1}{2}\overline \alpha + \frac{1}{2} < \beta - \floor{\beta}\\
    \left( \overline \alpha, \beta - \floor{\beta} - 1 \right) \in P_{(a,b)}, & \frac{1}{2}\overline \alpha - \frac{1}{2} < \beta - \floor{\beta} \leq \frac{1}{2}\overline \alpha + \frac{1}{2} \\
    \left( \overline \alpha + 1, \beta - \floor{\beta}  \right) \in P_{(a,b)}, &  0 \leq \beta - \floor{\beta} \leq \frac{1}{2}\overline \alpha - \frac{1}{2}\\
  \end{cases}
\]
To check (iii), we note that 
\[
  \begin{cases}
    \left( \overline \alpha, \beta - \floor{\beta} \right) \in P_{(0,0)}, & 2(\beta - \floor{\beta}) - \overline \alpha < 1\\
    \left( \overline \alpha, \beta - \floor{\beta} - 1 \right) \in P_{(0,0)}, & 2(\beta - \floor{\beta}) - \overline \alpha > 1\\
    \left( \overline \alpha+1, \beta - \floor{\beta}  \right) \in P_{(0,0)}, & 2(\beta - \floor{\beta}) - \overline \alpha = 1\\
  \end{cases}
\]
Here we're using the fact that $\charp k \neq 2$ to see that $\overline \alpha <0$ if $2(\beta - \floor{\beta}) - \overline \alpha =1$. Indeed, the point $(0, \frac{1}{2})$ 
has no integer translation in $P_{(0,0)}$, so $\pi_0 \not \in \ConR$ if $\charp k = 2$. 

Finally, to check (iv), we note that $(0, \frac{(q-1)/2}{q})$ has no integer translations in $P_{\left( 0, -1/q \right)}$. The polytope $P_{\left( -1/q, -1/q \right)}$ 
has no integer translations of $(0, \frac{(q-1)/2}{q})$ nor, for that matter, of  $(-\frac{1}{q}, \frac{(q-1)/2}{q})$. The polytope $P_{\left( -2/q, -1/q \right)}$ has
no integer translations of $(-\frac{1}{q}, \frac{(q-1)/2}{q})$.

We can use this calculation to compute the $F$-signature of $\ConR$ in the sense of  \cite{BST_Fsig_pairs}. 
%
%
%By \cite{VKToric}, the generators of the maximal free summand of $R^{1/q}$ correspond to lattice points 
%  \[
%    \left\{x \in \frac{1}{p^e} \bZ^2  \suchthat 0 \leq \braket{x, v_i} < 1 \right\}.
%  \]
%The splittings of these generators then correspond to lattice points 
%  \[
%    \left\{ x \in \frac{1}{p^e} \bZ^2 \suchthat -1 < \braket{x, v_i} \leq 0 \right\}.
%  \]
%This is the same as the lattice points in the red shape in \autoref{fig:quadric_cone_lattice}. 
Indeed, we see that the only splitting of $R \to R^{1/q}$ contained in $\ConR$ is $\pi_{(0,0)}$. Thus $s(\ConR) = 0$.
\end{example}

\begin{example}
  Let $R = k[x,y,z,xyz\invrs] \cong k[s,t,u,v]/(st-uv)  $. The cone $\sigma$ of $R$ is given by the extremal rays: 
  \[
    \left( 1, 0, 0 \right), \quad \left( 0, 1, 0 \right),\quad \left( 1, 0, 1 \right),\quad \left( 0, 1, 1 \right)
  \]
  We will show the following:
  \begin{claim}
    Let $e > 0$. If $(a,b,c) \in P_{-K} \cap \frac{1}{q}\bZ^3$ and $a + b + c > -1$, then $\pi_{(a,b,c)}$ is in $\plainConR_e$. 
  \end{claim}
   To see this claim, it's enough to consider the case $-1 < a,b \leq 0$, since $\vp(F^e_* x \cdot -)\in \plainConR$ whenever $\vp \in \plainConR$ and $x\in R$. The key point is that then 
  \begin{align}
    2 + a + c &> 1-b \geq  1, \textrm{ and }  \label{eq:first} \\
    2 + b + c &> 1-a \geq  1. %\label{eq:second}
  \end{align}
  Set $\vec d = (a,b,c)$. Let $(\alpha, \beta, \gamma) \in \bR^3$. We wish to find an integer translation of $(\alpha, \beta, \gamma)$ is in $P \cap (\vec d - P)$. We start
  by translating the polytope $P \cap (\vec d - P)$ by $(1,1,0)$: the resulting polytope is described as the set of $(x,y,z)$ satisfying the inequalities
  \begin{align*}
    0 & < x < 2+a\\
    0 & < y < 2+b\\
    0 & < x+z < 2 + a+c\\
    0 & < y+z < 2 + b+c
  \end{align*}
  We may assume, without loss of generality, that $0 < \alpha, \beta \leq 1$ and $0 \leq \gamma  < 1$. Note that we automatically have
  \[
    0 < \alpha< 2+a, \quad 0 < \beta < 2+b, \quad 0 < \alpha+\gamma, \quad 0 < \beta + \gamma
  \]
  If we happen to have $\alpha + \gamma < 2 + a + c$ and $\beta + \gamma < 2 + b + c$, then we're done. So suppose otherwise. 
  Without loss of generality, we may assume that $\alpha + \gamma \geq 2 + a + c$. If we also have $\beta + \gamma \geq 2 + b + c$, then the point $(\alpha, \beta, \gamma - 1)$ is in 
  $P \cap (\vec d - P)$: since $\alpha + \gamma < 2$, we have $ \alpha + \gamma - 1 < 1 < 2+a+c$ by equation  \autoref{eq:first}. On the other hand, $\alpha + \gamma \geq 2+b+c > 1$, so $\alpha + \gamma - 1 > 0$. 
  Similarly, we have $0 < \beta + \gamma - 1 < 2 + b + c$. 

  Now suppose that $\alpha + \gamma \geq 2 + a +c$ but $\beta + \gamma < 2 + b+c$.  If $\beta + \gamma > 1$, then the point $(\alpha, \beta, \gamma-1)$ is again in $P \cap (\vec d - P)$, as clearly we have $0 < \beta + \gamma -1  < 2 + b + c$. On the other hand, if $\beta + \gamma \leq 1$, then $(\alpha, \beta + 1, \gamma - 1)$ is in $P \cap (\vec d - P)$. Indeed, we just have to check that $\beta + 1 < 2 + b$, or in other words, that $b > \beta - 1$. We know, by assumption, that $b  > -1  - a -c $, so it suffices to check that $-a -c \geq \beta$. As $\gamma \geq 2 + a + c - \alpha$, we have

  \[
    \beta \leq 1 - \gamma \leq 1 - (2 + a + c - \alpha)  = -a - c + \alpha - 1 < -a - c
  \]
  The last inequality comes from the assumption that $0 \leq \alpha < 1$. This proves the claim. 

  By \cite{VKToric}, the splittings of $R^{1/q}$ correspond to the points in $\frac{1}{q}\bZ^3 \cap P_{\textrm{sig}}$, where $P_{\textrm{sig}}$ is the polytope given by 
  \[
    P_{\textrm{sig}} \coloneqq \left\{ x \in \bR^3 \suchthat \begin{array}[]{l}
      -1 < \braket{1, 0, 0} \cdot x \leq 0\\
      -1 < \braket{0, 1, 0} \cdot x \leq 0\\
      -1 < \braket{1, 0, 1} \cdot x \leq 0\\
      -1 < \braket{0, 1, 1} \cdot x \leq 0
  \end{array}\right\}.
  \]
  This polytope is depicted in \autoref{fig:FSigOfThreeFold}. As seen in the figure, the plane $x +y + z = -1$ cuts this polytope in half.  This shows that $s(\ConR) \geq s(R)/2 = 1/3$. Calculations in Macaulay2 \cite{M2} suggest that there are no further maps in $\ConR$ and that $s(\ConR) = 1/3$.  % graphic of plane x+y+z = -1 intersecting the von Korff polytope

  \begin{figure}[h]
    \centering
    \includegraphics[width=0.3\textwidth]{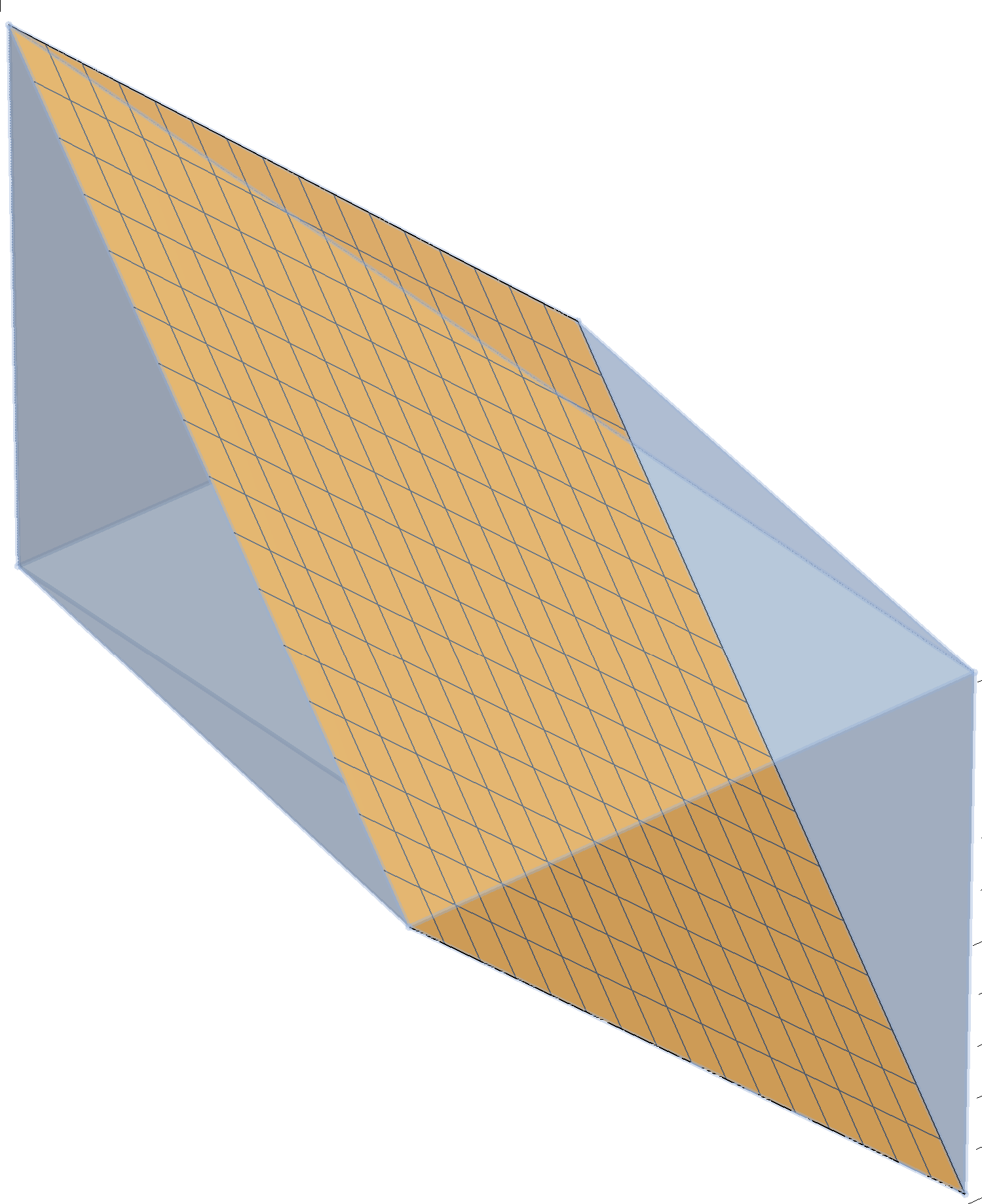}
    \caption{Comparison of maps in $\ConR$ and a polytope whose volume is $s(R)$, according to \cite{VKToric}. The plane is given by $x + y + z = -1$. All fractional lattice points
  lying above the plane correspond to maps in $\ConR$.}
    \label{fig:FSigOfThreeFold}
  \end{figure}
\end{example}

\section{Classifying singularities in terms of \texorpdfstring{$\plainConR$}{the diagonal cartier algebra}}
\autoref{thm:jacThm} suggests that rings with milder singularities have larger Cartier algebras $\plainConR$. We wonder whether the singularities of a ring can be well understood just by considering, in some sense, the size of $\plainConR$. The following conjecture would be a natural place to start in order to develop such a theory:

\begin{conjecture}
  Let $R$ be a finitely-generated algebra over a perfect field of positive characteristic. Then $\ConR = \fullCA R$ if and only if $R$ is regular. \label{conj:fullCA}
\end{conjecture}
One direction is clear: if $R$ is regular then so is $R\otimes_k R$, meaning that $F^e_*(R\otimes_k R)$ is a projective $R\otimes_k R$-module for all $e > 0$. As $R\otimes_k R \to R$ is a surjective map of $R\otimes_k R$ modules, the universal property of projective modules tells us that any map $F_*^e R \to R$ will lift: 
\[
    \xymatrix@C=5em{
      F^e_*(R \otimes_k R)  \ar@{-->}[r]^{\widehat \vp} \ar[d]_{F^e_* \mu} &  R\otimes_k R \ar[d]^{\mu} \\
      F^e_* R \ar[r]^{\vp} & R
    }
  \]

  Here, we're thinking of $R$ and $F^e_* R$ as  $R \otimes_k R$-modules via $\mu$. We have the following partial converses. 
  \begin{proposition}
   Suppose $k$ is a perfect field of positive characteristic and $R$ is a reduced $k$-algebra essentially of finite type. If $\ConR = \fullCA R$, then $R$ is strongly $F$-regular. 
  \end{proposition}
  \begin{proof}
    Apply  \autoref{cor:mainresult} to the case where $\mathfrak a = \mathfrak b = R$. We see that  $0\neq \tau(R) \subseteq \tau(R)^2$, so $\tau(R) = \tau(R)^2$. As $\tau(R)$ contains a regular element of $R$, it follows from Nakayama's lemma that $\tau(R) = R$. 
  \end{proof}

  We also know that the converse of \autoref{conj:fullCA} holds in the $\bQ$-Gorenstein toric case, using Watanabe and Yoshida's characterisation of the $F$-signature in that setting. The point of the $\bQ$-Gorenstein
  condition is just so that we know $P_{-K_X}$ is a translation of the dual cone $\sigma^\vee$ of $R$. Note that, by
  \cite[Proposition 4.2.7]{CLSToric}, this includes the case when the cone $\sigma$ of $R$ is simplicial, and 
  in particular all toric surfaces.  

  \begin{proposition}
    Work in \autoref{setting:toric}. Suppose also that $R$ is $\bQ$-Gorenstein and $\ConR = \fullCA R$. Then $R$ is regular. 
  \end{proposition}
  \begin{proof}
    %We can assume that $R$ has no torus factors, so that the dual cone $\sigma^\vee$ of $R$ is a strongly
    %convex
    %rational cone $C$. 
    It follows from \autoref{lemma:PisShift} that $\acptope = \sigma^\vee + v$ for some $v\in \bR^n$ such that
    $\braket{v, v_\rho} = -1$ for all $\rho \in \Sigma(1)$. Let $Q = \left\{ x \suchthat \forall \rho \in \Sigma(1): 0 <  \braket{x, v_\rho}<   1  \right\}$. By \cite{VKToric}, we know that $s(R) = \volume(Q)$. The key point is to notice that 
    \[
      \acptope \cap (v - \acptope) = (\sigma^\vee+ v) \cap (-\sigma^\vee) = \left\{ x \suchthat \forall \rho \in \Sigma(1): -1 < \braket{x, v_\rho} < 0 \right\} = -Q. 
    \] 
    Thus $s(R) = \volume(\acptope \cap (v - \acptope) )$. 

    Now, the function $\bR^n \to \bR$ given by $d\mapsto \volume(\acptope \cap (d - \acptope))$ is continuous, as this intersection is always compact. By taking $e$ sufficiently large, we can find a lattice point $d \in \frac{1}{p^e} \bZ^n \cap \acptope$ arbitrarily close to $v$. Thus we can make $\volume(\acptope \cap (d - \acptope))$ arbitrarily close to $\volume(\acptope \cap (v - \acptope))$. Since $\ConR = \fullCA R$, we have that $\pi_d \in \ConR$ and also that $R$ is diagonally $F$-split. Then $\volume(\acptope \cap (v - \acptope)) \geq 1$ by \autoref{lemma:volumeLemma}. Thus $s(R) =1$ and $R$ is regular. 
  \end{proof}

  \begin{lemma}
    Work in \autoref{setting:toric} and suppose that $R$ is $\bQ$-Gorenstein. Then $\acptope = \sigma^\vee + v$ where $v$ is a vector satisfying $\braket{v, v_\rho} = -1$ for all $\rho \in \sigma(1)$.
    \label{lemma:PisShift}
  \end{lemma}
  \begin{proof}
    This can be seen in a few different ways, but here's one. Let $r$ be the Cartier index of $R$, so that $rK_{X}$ is Cartier. As $-rK_X = \sum_{\rho \in \sigma(1)} rD_{\rho}$, we have by \cite[Theorem 4.2.8]{CLSToric} that there exists $w$ such that $\braket{w,v_\rho}= -r $ for all $\rho \in \sigma(1)$. Then we certainly
    have $\frac{1}{r}w + \sigma^{\vee} \subseteq P_{-K}$. On the other hand, for any $x \in P_{-K}$ we have $\braket{x - \frac{1}{r}w, v_\rho} > 0$ for all $\rho$, meaning $\frac{1}{r}w + \sigma^{\vee} = P_{-K}$. So we set $v = \frac{1}{r}w$. 
    %\ref[Proposition 4.2.2]{CLSToric} as well as \cite[p. 192]{CLSToric} .
  \end{proof}

  \begin{lemma}
    Let $R$ be a diagonally split $n$-dimensional affine toric variety. For all $e$ and all $d\in \frac{1}{p^e} \bZ^n$, if $\pi_d\colon F^e_* R \to R$ is in $\plainConR_e(R)$ then $\volume(\acptope \cap (d - \acptope)) \geq 1$. 
    \label{lemma:volumeLemma}
  \end{lemma}
  \begin{proof}
    For all $e'> e$, let $\pi_d^{e'} = \pi_d \cdot (\pi_0)^{e'-e} \in \fullCA R_{e'}$. This is the map $F^{e'}_*R \to R$ corresponding to the lattice point $d \in \frac{1}{p^{e'}}\bZ^n$. The map $\pi_0$ is in $\plainConR_1(R)$ by \autoref{lemma:diagFSplit}, so we have $\pi_d^{e'}\in \ConR$ since $\ConR$ is a Cartier algebra. By \autoref{prop:resCriterion} the polytope $\acptope \cap(d - \acptope)$ contains at least $p^{e'n}$ fractional lattice points in $\frac{1}{p^{e'}} \bZ^n$. Then we're done, as for any polytope $P\subseteq \bR^n$ we have
    \[
      \volume(P) = \lim_{m\to \infty} \frac{\#\left\{ \frac{1}{m} \bZ^n \cap P\right\}}{m^n}
    \]
    This is a well-known fact; see for instance \cite[Theorem 2.2]{MillerSturmfelsCA}.
 \end{proof}

 \appendix
 \section{Proof that Test Ideals Along Closed Subschemes Exist}
 \label{appendix:tauI}
 In this appendix,  we show there's a notion of test elements for test ideals along closed subschemes, working in \autoref{setting:tauI}. Consequently, these test ideal exist. We remark that these proofs are essentially the same as those in \cite[\S 6]{SchwedeFadj}. The salient difference between our setting and the one in \cite{SchwedeFadj} is that here we're not assuming that $I$ is an $F$-pure center of $R$. Instead, we're just assuming that $\cartalg C$ is compatible with $I$. 

  \begin{lemma}
    \label{lemma:localizeThing}
    Work in \autoref{setting:tauI}. There exists some $\gamma \in R \setminus I$  such that %(or $\mathfrak a \setminus I$?) i don't see why
                                                            % we need this, but i'll ask Karl
    \begin{enumerate}
      \item All proper ideals of $R_\gamma$ compatible with  $\cartalg C_\gamma$ are contained in $IR_\gamma$, and  
      \item The Cartier algebra $\cartalg C_\gamma$   is $F$-pure
    \end{enumerate}
  \end{lemma}
  \begin{proof}
    Let $\pi\colon R \to R/I$ be the quotient map. It follows from our assumptions in \autoref{setting:tauI} that $\cartalg C|_{R/I}$ is a nondegenerate Cartier algebra on $R/I$, so $\tau\left( R/I, \cartalg C|_{R/I} \right)$ is well-defined (and, in particular, nonzero) \cite{SchwedeNonQGor}. Choose $\gamma_1\in R$ so that $\pi(\gamma_1) \in \tau\left( R/I, \cartalg C|_{R/I} \right)$ and $\pi(\gamma_1) \neq 0$. Then all proper ideals of $R$ compatible with $\cartalg C_{\gamma_1}$ are contained in $IR_{\gamma_1}$. Indeed, we have the following diagram for all $\vp\in \cartalg C_{\gamma_1}$:
    \[
      \xymatrix{
        F^e_* R_{\gamma_1}  \ar[d]_\pi  \ar[r]^{\vp} &  R_{\gamma_1} \ar[d]_\pi\\
        F^e_* (R/I)_{\gamma_1} \ar[r]^{\bar \vp} & \left(  R/I\right)_{\gamma_1}
      }
    \]
    If $\vp(J) \subseteq J$, then $\bar \vp(\pi( J)) \subseteq \pi(J)$. Note that as $\vp$ runs through all maps in $\cartalg C_{\gamma_1}$, $\bar \varphi$ will run through all maps in $\left( \cartalg C|_{R/I} \right)_{\gamma_1}$.  So if $J$ is a proper ideal of $R_{\gamma_1}$ compatible with $\cartalg C_{\gamma_1}$, then $\pi(J)$ is compatible with $(\cartalg C|_{R/I})_{\gamma_1}$. But we have that $\tau(R/I, (\cartalg C|_{R/I})_{\gamma_1}) = \tau(R/I, \cartalg C|_{R/I})(R/I)_{\gamma_1} = (R/I)_{\gamma_1}$, so $(\cartalg C|_{R/I})_{\gamma_1}$ is $F$-regular. This means that $\pi(J) = 0$, meaning $J \subseteq I R_{\gamma_1}$.

    We note that all proper ideals of $R_{\gamma_1 \gamma_2}$ compatible with $\cartalg C_{\gamma_1 \gamma_2}$ are contained in $I R_{\gamma_1 \gamma_2}$, for all $\gamma_2 \in R \setminus I$. So choose $e>0$ and $\psi \in \cartalg C_e$ to be some map whose image is not contained in $I$ and let $\gamma_2 \in \psi(F_*^e R) \setminus I$. Then the element $\gamma = \gamma_1 \gamma_2$ satisfies the conclusion of the lemma. 
    
    %If we need $\gamma \in \mathfrak a \setminus I$, then just multiply the $\gamma$ we just got by any element in $\mathfrak a\setminus I$
  \end{proof}

%  \begin{notation}
%    Let $\cartalg D$ be a Cartier algebra and let $\Psi = \sum_{i} \psi_i \in \cartalg D$, where the sum is finite, each $\psi_i$ is nonzero, and $\psi_i \in \cartalg D_{e_i}$ for each $i$. Then we adopt the notation
%    \[
%      \Psi(F^\bullet_* d) \coloneqq \sum_i \psi_i\left( F^{e_i}_* d \right).
%    \]
%    Further, we say $\Psi$ has \emph{minimal degree $e_0$} if $e_0 = \min_i\left\{ e_i \right\}$. 
%  \end{notation}

  \begin{proposition}[{\cf \cite[Lemma 6.12]{SchwedeFadj}}]
    Work in \autoref{setting:tauI}. There exists an element $\gamma \in R \setminus I$ such that, for all $d\in R\setminus I$,  there exists an integer $m$ and a map $\Psi \in \cartalg C$ of minimal degree greater than 0 such that  $\gamma^m = \Psi \cdot d$. 
    \label{prop:before614}
  \end{proposition}
  % (do we need all $n$ sufficiently large?)  no
  \begin{proof}
    Choose $\gamma$ as in \autoref{lemma:localizeThing}. It suffices to prove that
    \[
      J\coloneqq \sum_{e > 0} \sum_{\vp \in \cartalg C_e} \vp(F^e_* d) R_{\gamma} = R_\gamma
    \]
    for all $d \in R\setminus I$. By definition of $\gamma$, it suffices to show that $J$ is compatible with $\cartalg C_\gamma$ and not contained in $IR_\gamma$. It's clear that $J$ is compatible with $\cartalg C_\gamma$, so we'll just show $J$ is not contained in $IR_\gamma$. Let $\pi: R \to R/I$ be the natural surjection. As $(\cartalg C|_{R/I})_\gamma$ is $F$-regular and $\pi(d) \neq 0$, there exist $e$ and  $\overline \vp \in (\cartalg C|_{R/I})_\gamma$ such that $\overline\vp(F^e_* \pi (d)) = 1$. This means that $\vp(F^e_* d) = 1+x$ for some $x\in I$, where $\vp$ is any map  in $\cartalg C_\gamma$ that induces $\overline \vp$. But then $1 + x \in J$, so $J$ is not contained in $IR_\gamma$. 
  \end{proof}

  \begin{lemma}[{\cf \cite[Lemma 6.13]{SchwedeFadj}}]
    Suppose $\vp \in \fullCA{R}$, $c\in R$, and $b \in \vp\cdot (cR)$. Then $b^2 \in \vp^n\cdot (cR)$ for all $n > 0$. 
    \label{lemma:bSquared}
  \end{lemma}
  \begin{proof}
    This proof is essentially the same as that of Lemma 6.13 of \cite{SchwedeFadj}. We include it here for completeness. 

    We proceed by induction. The base case is given by the hypothesis. Suppose that $\vp = \sum_i \vp_i$, where $\vp_i \in \fullCA{R}_{e_i}$. Then we compute: 
    \begin{align*}
      b^2 &\in b\, \vp\cdot (c R)  = b \sum_i \vp_i\left( F^{e_i}_*cR \right) = \sum_i \vp_i\left( F^{e_i}_* b^{p^{e_i}} c R \right) \subseteq \sum_i \vp_i\left( F^{e_i}_* b^2 c R \right)\\
      &= \vp\cdot ( b^2 c R) \subseteq \vp\cdot \left( \left( \vp^n\cdot (c R)\right) c \right) \subseteq \vp^{n+1}\cdot(cR) 
    \end{align*} %
    %By induction, we're done, as $\vp\left(F^\bullet_* \vp^n(F^\bullet_* R) \right)= \vp^{n+1}(F^\bullet_* R)$.
  \end{proof}

  \begin{proposition}[{\cf \cite[Proposition 6.14]{SchwedeFadj}}]
    There is an element $b\in R\setminus I$ such that for all $d \in R \setminus I$, there exists $\Psi \in \cartalg C$ such that $b  = \Psi \cdot d$. 
    \label{prop:testElement}
  \end{proposition}
  \begin{proof}
    This proof is essentially the same as that of Proposition 6.14 of \cite{SchwedeFadj}. We include it here for completeness. 

    Choose $\gamma$ as in \autoref{prop:before614}. Then there exists $m$  and $\Psi$, of minimal degree $e_0 > 0$, such that $\gamma^m = \Psi \cdot 1$. By \autoref{lemma:bSquared}, $\gamma^{2m} \in \Psi^n \cdot R$ for all $n > 0$. We will show that $b = \gamma^{3m}$ works. 

    Let $d\in R\setminus I$ be arbitrary. Then there exists $\Psi_1$ and $m_1$ such that $\gamma^{m_1} = \Psi_1 \cdot d$. If $m_1 < 3m$, then we're done. Otherwise, choose $n$ such that $m_1 < mp^{n e_0}$ and write $\Psi^n = \sum_i \psi_i$ with $\psi_i \in \cartalg C_{e_i}$ for all $i$. Note that $e_i \geq ne_0$ for all $i$.  Then we have:
    \begin{align*}
      \gamma^{3m} &= \gamma^m \gamma^{2m} \in \gamma^m \Psi^n\cdot ( R) = \gamma^m \sum_i \psi_i(F^{e_i}_* R) = \sum_i \psi_i(F^{e_i}_* \gamma^{m p^{e_i}}R) \subseteq \sum_i \psi_i(F^{e_i}_* \gamma^{m p^{n e_0}}R)\\
      & \subseteq \sum_i \psi_i(F^{e_i}_* \gamma^{m_1}R) = \Psi^n\cdot (\gamma^{m_1}R) \subseteq \Psi^n \Psi_1\cdot (dR),
    \end{align*}
    as desired. 
  \end{proof}

  \begin{theorem}[{\cf \cite[Lemma 6.17 and Theorem 6.18]{SchwedeFadj}}]
    Let $b$ be as in \autoref{prop:testElement}. Then 
    \[
      \tau_I(R, \cartalg C) = \sum_{e \geq 0} \sum_{\vp \in \cartalg C_e} \vp(F^e_* b)
    \]
    \label{thm:TauIsSum}
  \end{theorem}
  \begin{proof}
    This proof is essentially the same as Lemma 6.17 and Theorem 6.18 of \cite{SchwedeFadj}. We include it here for completeness. 

    Let $\tau_I(R, \cartalg C; b)$ denote the ideal
    \[
      \sum_{e \geq 0} \sum_{\vp \in \cartalg C_e} \vp(F^e_* b),
    \]
    and note that we have 
    \[
      \tau_I(R, \cartalg C; b) = \sum_{\vp \in \cartalg C} \vp \cdot b.
    \]
    Then we need to show:
    \begin{enumerate}
      \item $\tau_I(R, \cartalg C; b) \not \subseteq I$,  \label{item:a}
      \item $\tau_I(R, \cartalg C; b)$ is compatible with $\cartalg C$, and \label{item:b}
      \item $\tau_I(R, \cartalg C; b)$ is contained in any other ideal satisfying (a) and (b). 
    \end{enumerate}
    For (a), it's enough to show that $b \in \tau_I(R, \cartalg C; b)$. This follows from \autoref{prop:testElement}, using $d = b$. Assertion (b) is clear from the construction of $\tau_I(R, \cartalg C; b)$. 
    
    For the final assertion, let $J$ be some ideal satisfying (a) and (b). Choose some $d \in J \setminus I$. Then
    \[
      \sum_{\vp \in \cartalg C} \vp \cdot  d \subseteq \sum_{\vp \in \cartalg C} \vp \cdot J \subseteq J
    \]
    By \autoref{prop:testElement}, we have $b\in J$. But then 
    \[
      \tau_I(R, \idealsProd; b) = \sum_{\vp \in \cartalg C} \vp \cdot  b \subseteq \sum_{\vp \in \cartalg C} \vp\cdot J \subseteq J
    \] 
 \end{proof}

 \section{Multiplier Ideal Computations}
 \label{appendix:multIdeal}
 This appendix is devoted to proving the following containment, in the context of the proof of \autoref{thm:hardContainment}:
 \[
   H^0(B_\kappa, F_*^e  \omega_{B_\kappa}(\tilde X_\kappa - \floor{qD_\kappa})) \subseteq  R\left( (1-q) K_{A_\kappa} \right)\overline{\prod_i\mathfrak (a_i)_\kappa^{ t_i(q-1) }}  d_\kappa^N \xi_\kappa^N.
 \]
 Recall the following  notation: first, we work in \autoref{setting:set2} with $\charp R = 0$. Further, $\pi_1\colon A_1 \to A$ is a factorizing resolution of $X \subseteq A$. We denote by $X_1$ the strict transform of $X$ in $A_1$ and we let $\pi_2\colon A_2 \to A_1$ be the blow up along $X_1$. We set $X_2$ to be the $\pi_2$-exceptional divisor dominating $X_1$. Finally, $\gamma\colon B \to A_2$ is a blow up of a subscheme not containing $X_2$ and set $\tilde X = \gamma_*\invrs X_2$. For each $i$, we set $\mathfrak a_i \ssheaf_{B} = \ssheaf_{B}(-F_i)$ and we set $F = \sum_i t_i F_i$. We have that $Z = cX$ at the generic point of $X$, where $c$ is the codimension of $X$ in $A$. Further, $d$ and  $\xi$ and elements of $R$, and  $\varepsilon > 0$. We also have that $\eta \in R\setminus I$ is an element satisfying 
    \[
      K_{B/A_2} + \gamma^* \pi_2^*K_{A_1} - \floor{q f^* K_A} - \divisor \eta \leq f^*\left( (1-q) K_A\right) + \gamma^* X_2 - \tilde X,
    \]
and $V$ is an effective divisor on $B$. We define
 \[
    D = f^* K_A + F + V+ \varepsilon \divisor_{A'}(d\xi \eta).
 \]
 We also choose $e>0$ so that $t_i(q-1)\in \bZ$ and $(1-q)K_A$ is Cartier, where $q = p^e$. We also have an integer $N> 0$ and we assume that $q\varepsilon > N$. From now on, we work exclusively modulo $p$ at $\kappa$ and we abuse notation by omitting the subscripts $\kappa$. 
 We compute:
    \[
      \begin{array}[]{ll}
                  & H^0(B, F_*^e  \omega_{B}(\tilde X - \floor{qD}))\\
                = & H^0\left( A, f_* F_*^e \ssheaf_{B}\left(K_{B} + \tilde X - \floor{qf^*K_A + qF + qV + q\varepsilon \divisor_{B}(d\xi\eta)}  \right) \right)\\
                \subseteq & H^0\left( A, f_* F_*^e \ssheaf_{B}\left(K_{B} + \tilde X - \floor{qf^*K_A} - \floor{qF} - \floor{qV} - \floor{q\varepsilon} \divisor_{B}(d\xi\eta) \right) \right)\\
        \subseteq & H^0\left( A, f_* F_*^e \ssheaf_{B}\left(K_{B} + \tilde X - \floor{qf^*K_A} - \floor{qF} - N\divisor_{B}(d\xi\eta)\right) \right)
      \end{array}
    \]
    Here, we're using the fact that $H$ is anti-effective. Note that $K_{A_2} = \pi_2^* K_{A_1} + (c-1)X_2$ by \cite[Exercise II.8.5]{HartshorneAG}, and so $K_B = K_{B/A_2} + \gamma^*( \pi_2^* K_{A_1} + (c-1)X_2 )$. Then  it follows from our choice of $q$ and the construction of $\eta$ that
    \begin{align*}
        & H^0\left( A, f_* F_*^e \ssheaf_{B}\left(K_{B} + \tilde X - \floor{qf^*K_A} - \floor{qF} - N\divisor_{B}(d\xi\eta)\right) \right)\\
        = & H^0\left( A, f_* F_*^e \ssheaf_{B}\left(K_{B/A_2} + \gamma^* \pi_2^* K_{A_1} + (c-1) \gamma^* X_2 + \tilde X - \floor{qf^*K_A} - \floor{qF} - N\divisor_{B}(d\xi\eta)\right) \right)\\
        \subseteq & H^0\left( A, f_* F_*^e \ssheaf_{B}\left( f^*\left( (1-q)K_A \right) + c\gamma^* X_2 - \floor{qf^*K_A} - \floor{qF} - N\divisor_{B}(d\xi)\right) \right)
    \end{align*}

    Next, we examine the term $-\floor{qF}$. For each $i$, we can write $F_i = \tilde F_i + a_i \gamma^* X$ for some $a_i \in \bN$, where $\tilde F_i$ is not supported along $\tilde X$. Since we assumed $Z = cX$ at the generic point of $X$, we have $\sum_i t_i a_i = c$. Then we see
    \[
      -\floor{qF} = - \floor{\sum_i qt_i \tilde F_i  + qt_i a_i \gamma^* X} = - \floor{\sum_i qt_i \tilde F_i}  - qc \gamma^* X.
    \]
     Thus we have:
    \begin{align*}
        & H^0\left( A, f_* F_*^e \ssheaf_{B}\left( f^*\left( (1-q)K_A \right) + c\gamma^* X_2 - \floor{qf^*K_A} - \floor{qF} - N\divisor_{B}(d\xi)\right) \right)\\
      = & H^0\left( A, f_* F_*^e \ssheaf_{B}\left( f^*\left( (1-q)K_A \right) - \floor{qf^*K_A} - \floor{\sum_i q t_i \tilde F_i} + (c- qc)\gamma^* X_2 - N\divisor_{B}(d\xi)\right) \right)\\
      \subseteq & H^0\left( A, f_* F_*^e \ssheaf_{B}\left( f^*\left( (1-q)K_A \right) - \floor{qf^*K_A} - \sum_i \floor{q t_i} \tilde F_i + (c- qc)\gamma^* X_2 - N\divisor_{B}(d\xi)\right) \right)\\
    \end{align*}
  Note that for all $i$, $\floor{qt_i} \geq \floor{(q-1)t_i} = (q-1)t_i$. As $\tilde F_i$ is effective for each $i$, we get
  \begin{align*}
     - \sum_i \floor{q t_i} \tilde F_i + (c- qc) \gamma^* X_2 & \leq - \sum_i (q-1) t_i \tilde F_i - (q-1)c \gamma^* X_2 \\
     & \leq - \sum_i \left((q-1) t_i \tilde F_i + (q-1)t_i a_i \gamma^*X_2 \right) \\
     & \leq - \sum_i  (q-1) t_i F_i
  \end{align*}
  By the construction of $\eta$, it follows that
    \begin{align*}
                & H^0\left( A, f_* F_*^e \ssheaf_{B}\left( f^*\left( (1-q)K_A \right) - \floor{qf^*K_A} - \sum_i \floor{q t_i} \tilde F_i + (c- qc) \gamma^* X_2 - N\divisor_{B}(d\xi)\right) \right)\\
      \subseteq & H^0\left( A, f_* F_*^e \ssheaf_{B}\left(f^*((1-q) K_A) - \sum_i  (q-1) t_i F_i  - N\divisor_{B}(d\xi)\right) \right)\\
      \subseteq & R\left( (1-q) K_A \right)\overline{\prod_i\mathfrak a_i^{ t_i(q-1) }}  d^N \xi^N,
    \end{align*}
    as desired. 

  \bibliographystyle{skalpha}
  \bibliography{/home/zdorovo/Math/MainBib}

\newcommand{\etalchar}[1]{$^{#1}$}
\providecommand{\bysame}{\leavevmode\hbox to3em{\hrulefill}\thinspace}
\providecommand{\MR}{\relax\ifhmode\unskip\space\fi MR}
\providecommand{\MRhref}[2]{%
  \href{http://www.ams.org/mathscinet-getitem?mr=#1}{#2}
}
\providecommand{\href}[2]{#2}
\begin{thebibliography}{CHP{\etalchar{+}}16}

\bibitem[Bli13]{BlickleCartier}
{\sc M.~Blickle}: \emph{Test ideals via algebras of {$p^{-e}$}-linear maps}, J.
  Algebraic Geom. \textbf{22} (2013), no.~1, 49--83. {\sf\scriptsize 2993047}

\bibitem[BS13]{BlickleSchwedeSurvey}
{\sc M.~Blickle and K.~Schwede}: \emph{{$p^{-1}$}-linear maps in algebra and
  geometry}, Commutative algebra, Springer, New York, 2013, pp.~123--205.
  {\sf\scriptsize 3051373}

\bibitem[BSTZ10]{BSTZ_discreteness_and_rationality}
{\sc M.~Blickle, K.~Schwede, S.~Takagi, and W.~Zhang}: \emph{Discreteness and
  rationality of {$F$}-jumping numbers on singular varieties}, Math. Ann.
  \textbf{347} (2010), no.~4, 917--949. {\sf\scriptsize 2658149}

\bibitem[BST12]{BST_Fsig_pairs}
{\sc M.~Blickle, K.~Schwede, and K.~Tucker}: \emph{F-signature of pairs and the
  asymptotic behavior of frobenius splittings}, Advances in Mathematics
  \textbf{231} (2012), no.~6, 3232--3258.

\bibitem[BK07]{BKFsplitting}
{\sc M.~Brion and S.~Kumar}: \emph{Frobenius splitting methods in geometry and
  representation theory}, Progress in Mathematics, Birkh{\"a}user Boston, 2007.

\bibitem[CS18]{USTPDiagFReg}
{\sc J.~{Carvajal-Rojas} and D.~{Smolkin}}: \emph{{The Uniform Symbolic
  Topology Property for Diagonally $F$-regular Algebras}}, ArXiv:1807.03928
  (2018).

\bibitem[CHP{\etalchar{+}}16]{PayneUnimodular}
{\sc J.~{Chou}, M.~{Hering}, S.~{Payne}, R.~{Tramel}, and B.~{Whitney}}:
  \emph{{Diagonal splittings of toric varieties and unimodularity}}, ArXiv
  e-prints (2016).

\bibitem[CLS11]{CLSToric}
{\sc D.~A. Cox, J.~B. Little, and H.~K. Schenck}: \emph{Toric varieties},
  Graduate Studies in Mathematics, vol. 124, American Mathematical Society,
  Providence, RI, 2011. {\sf\scriptsize 2810322}

\bibitem[DEL00]{DELsubadd}
{\sc J.-P. Demailly, L.~Ein, and R.~Lazarsfeld}: \emph{A subadditivity property
  of multiplier ideals}, Michigan Math. J. \textbf{48} (2000), 137--156,
  Dedicated to William Fulton on the occasion of his 60th birthday.
  {\sf\scriptsize 1786484}

\bibitem[ELS01]{ELSsymbolicPowers}
{\sc L.~Ein, R.~Lazarsfeld, and K.~E. Smith}: \emph{Uniform bounds and symbolic
  powers on smooth varieties}, Invent. Math. \textbf{144} (2001), no.~2,
  241--252. {\sf\scriptsize 1826369}

\bibitem[Eis10]{EisensteinRestrictionThm}
{\sc E.~Eisenstein}: \emph{Generalizations of the restriction theorem for
  multiplier ideals}, arXiv:1001.2841v1.

\bibitem[GS]{M2}
{\sc D.~R. Grayson and M.~E. Stillman}: \emph{Macaulay2, a software system for
  research in algebraic geometry}.

\bibitem[HW02]{HaraWatanabeFregvsLogT}
{\sc N.~Hara and K.~Watanabe}: \emph{F-regular and f-pure rings vs. log
  terminal and log canonical singularities}, Journal Of Algebraic Geometry
  \textbf{11} (2002), no.~2, 363--392 (English).

\bibitem[Har98]{HaraInjectivity}
{\sc N.~Hara}: \emph{A characterization of rational singularities in terms of
  injectivity of {F}robenius maps}, Amer. J. Math. \textbf{120} (1998), no.~5,
  981--996. {\sf\scriptsize 1646049}

\bibitem[Har01]{HaraMultIdealIsTestIdeal}
{\sc N.~Hara}: \emph{Geometric interpretation of tight closure and test
  ideals}, Trans. Amer. Math. Soc. \textbf{353} (2001), no.~5, 1885--1906.
  {\sf\scriptsize 1813597}

\bibitem[HY03]{HaraYoshidaSubadd}
{\sc N.~Hara and K.-I. Yoshida}: \emph{A generalization of tight closure and
  multiplier ideals}, Trans. Amer. Math. Soc. \textbf{355} (2003), no.~8,
  3143--3174. {\sf\scriptsize 1974679}

\bibitem[Har77]{HartshorneAG}
{\sc R.~Hartshorne}: \emph{Algebraic geometry}, Springer-Verlag, New
  York-Heidelberg, 1977, Graduate Texts in Mathematics, No. 52. {\sf\scriptsize
  0463157}

\bibitem[HH94]{HHbasechange}
{\sc M.~Hochster and C.~Huneke}: \emph{{$F$}-regularity, test elements, and
  smooth base change}, Trans. Amer. Math. Soc. \textbf{346} (1994), no.~1,
  1--62. {\sf\scriptsize 1273534}

\bibitem[HH99]{HHequalCharZero}
{\sc M.~Hochster and C.~Huneke}: \emph{Tight closure in equal characteristic
  zero}, Preprint. URL: \url{www.math.lsa.umich.edu/~hochster/tcz.ps}.

\bibitem[HH02]{HHsymbolic}
{\sc M.~Hochster and C.~Huneke}: \emph{Comparison of symbolic and ordinary
  powers of ideals}, Invent. Math. \textbf{147} (2002), no.~2, 349--369.
  {\sf\scriptsize 1881923}

\bibitem[hz]{MOtensor}
{\sc Y.~Z. (https://mathoverflow.net/users/1877/ying zhang)}: \emph{Is the
  tensor product of regular rings still regular},
  URL:https://mathoverflow.net/q/50075 (version: 2010-12-21).

\bibitem[Kun69]{Kunz69}
{\sc E.~Kunz}: \emph{Characterizations of regular local rings for
  characteristic {$p$}}, Amer. J. Math. \textbf{91} (1969), 772--784.
  {\sf\scriptsize 0252389}

\bibitem[Lan05]{LangAlgebra}
{\sc S.~Lang}: \emph{Algebra}, Graduate Texts in Mathematics, Springer New
  York, 2005.

\bibitem[Laz04]{PositivityII}
{\sc R.~Lazarsfeld}: \emph{Positivity in algebraic geometry. {II}}, Ergebnisse
  der Mathematik und ihrer Grenzgebiete. 3. Folge. A Series of Modern Surveys
  in Mathematics [Results in Mathematics and Related Areas. 3rd Series. A
  Series of Modern Surveys in Mathematics], vol.~49, Springer-Verlag, Berlin,
  2004, Positivity for vector bundles, and multiplier ideals. {\sf\scriptsize
  2095472}

\bibitem[MS06]{MillerSturmfelsCA}
{\sc E.~Miller and B.~Sturmfels}: \emph{Combinatorial commutative algebra},
  Graduate Texts in Mathematics, Springer New York, 2006.

\bibitem[Pay09]{PayneToric}
{\sc S.~Payne}: \emph{Frobenius splittings of toric varieties}, Algebra Number
  Theory \textbf{3} (2009), no.~1, 107--119. {\sf\scriptsize 2491910}

\bibitem[{Sch}08]{SchwedeCenters}
{\sc K.~{Schwede}}: \emph{{Centers of F-purity}}, ArXiv e-prints (2008).

\bibitem[Sch09]{SchwedeFadj}
{\sc K.~Schwede}: \emph{{$F$}-adjunction}, Algebra Number Theory \textbf{3}
  (2009), no.~8, 907--950. {\sf\scriptsize 2587408}

\bibitem[Sch11]{SchwedeNonQGor}
{\sc K.~Schwede}: \emph{Test ideals in non-{$\Bbb{Q}$}-{G}orenstein rings},
  Trans. Amer. Math. Soc. \textbf{363} (2011), no.~11, 5925--5941.
  {\sf\scriptsize 2817415}

\bibitem[ST11]{SchwedeTuckerSurvey}
{\sc K.~Schwede and K.~Tucker}: \emph{A survey of test ideals},
  arXiv:1104.2000.

\bibitem[Smi00]{SmithMultIdealIsTestIdeal}
{\sc K.~E. Smith}: \emph{The multiplier ideal is a universal test ideal}, Comm.
  Algebra \textbf{28} (2000), no.~12, 5915--5929, Special issue in honor of
  Robin Hartshorne. {\sf\scriptsize 1808611}

\bibitem[SH06]{HunekeSwansonIntegral}
{\sc I.~Swanson and C.~Huneke}: \emph{Integral closure of ideals, rings, and
  modules}, Integral closure of ideals, rings, and modules, no. v. 13,
  Cambridge University Press, 2006.

\bibitem[Tak06]{TakagiSingMultIdeals}
{\sc S.~Takagi}: \emph{Formulas for multiplier ideals on singular varieties},
  Amer. J. Math. \textbf{128} (2006), no.~6, 1345--1362. {\sf\scriptsize
  2275023}

\bibitem[Tak08]{TakagiCharPadjoint}
{\sc S.~Takagi}: \emph{A characteristic {$p$} analogue of plt singularities and
  adjoint ideals}, Math. Z. \textbf{259} (2008), no.~2, 321--341.
  {\sf\scriptsize 2390084}

\bibitem[Tak10]{TakagiAdjointHighCodim}
{\sc S.~Takagi}: \emph{Adjoint ideals along closed subvarieties of higher
  codimension}, J. Reine Angew. Math. \textbf{641} (2010), 145--162.
  {\sf\scriptsize 2643928}

\bibitem[Tak13]{TakagiLCvsFP}
{\sc S.~Takagi}: \emph{Adjoint ideals and a correspondence between log
  canonicity and {$F$}-purity}, Algebra Number Theory \textbf{7} (2013), no.~4,
  917--942. {\sf\scriptsize 3095231}

\bibitem[vK11]{VKToric}
{\sc M.~von Korff}: \emph{{F}-signature of affine toric varities},
  arXiv:1110.0552.

\end{thebibliography}
\end{document}